\documentclass[11pt]{article}

\usepackage{amsmath,amssymb,amsthm,amsfonts,verbatim}
\usepackage{enumerate}
\usepackage[all,2cell]{xy}
\usepackage[margin=1.3in]{geometry}
\usepackage[vcentermath]{youngtab}

\title{Representation stability in cohomology and\\ asymptotics for families of varieties over finite fields}

\author{Thomas Church, Jordan S. Ellenberg, and Benson Farb
}

\theoremstyle{plain}
\newtheorem{theorem}{Theorem}[section]
\newtheorem{theoremone}{Theorem}
\newtheorem{proposition}[theorem]{Proposition}
\newtheorem{prop}[theorem]{Proposition}

\newtheorem{numberedremark}[theorem]{Remark}
\newtheorem{lemma}[theorem]{Lemma}
\newtheorem{claim}[theorem]{Claim}

\newtheorem{observation}[theorem]{Observation}
\newtheorem{xample}[theorem]{Example}
\newtheorem*{theorem:pc}{Theorem \ref{theorem:principal congruence}}
\newtheorem*{theorem:jk}{Theorem \ref{theorem:johnson}}
\newtheorem*{theorem:br}{Theorem \ref{theorem:brunnian}}
\newtheorem{corollary}[theorem]{Corollary}
\theoremstyle{definition}
\newtheorem{cor}[theorem]{Corollary}
\newtheorem{definition}[theorem]{Definition}
\newtheorem*{aside}{Aside}
\newtheorem{numberedremarknotitalicized}[theorem]{Remark}

\newcommand{\nc}{\newcommand}
\newcommand{\set}[1]{\{#1\}}
\newcommand{\beq}{\begin{displaymath}}
\newcommand{\eeq}{\end{displaymath}}
\newcommand{\tensor} {\otimes}
\renewcommand{\AA}{\mathcal{A}}

\newcommand{\bs}{\backslash}
\nc{\dmo}{\DeclareMathOperator}

\nc{\I}{\mathcal{I}}
\nc{\K}{\mathcal{K}}
\nc{\U}{\mathcal{U}}
\renewcommand{\L}{\mathbb{L}}
\nc{\ra}{\to}
\nc{\Q}{\mathbb{Q}}
\nc{\R}{\mathbb{R}}
\nc{\Z}{\mathbb{Z}}
\nc{\C}{\mathbb{C}}
\nc{\N}{\mathbb{N}}
\nc{\F}{\mathbb{F}}
\nc{\A}{\mathbb{A}}
\nc{\T}{\mathcal{T}}
\nc{\tT}{\widetilde{\T}}
\nc{\V}{\mathcal{V}}
\nc{\PP}{\mathbf{P}}
\nc{\LL}{\mathbf{L}}
\nc{\cL}{\mathcal{L}}
\nc{\G}{\mathbb{G}}
\nc{\CP}{\C\PP}
\nc{\BB}{\mathcal{B}}
\nc{\sss}{s'}
\nc{\ttt}{t'}
\nc{\FF}{\mathcal{F}}
\nc{\Torus}{\mathbb{T}}
\dmo{\GL}{GL}
\dmo{\PSL}{PSL}
\dmo{\Teich}{Teich}
\dmo{\spec}{spec}
\nc{\gin}{i}
\nc{\ga}{\Gamma}
\dmo{\Gal}{Gal}
\dmo{\Out}{Out}
\dmo{\Aut}{Aut}
\dmo{\brun}{Brun}
\dmo{\Conf}{Conf}
\dmo{\AConf}{AConf}
\dmo{\PConf}{PConf}
\dmo{\Frob}{Frob}
\dmo{\fix}{fix}
\dmo{\trans}{trans}
\dmo{\sign}{sign}
\dmo{\mult}{mult}
\dmo{\Fix}{Fix}
\dmo\Tr{Tr}
\dmo\Hom{Hom}

\dmo\im{im}
\dmo\id{id}
\dmo\SL{SL}
\dmo\Sp{Sp}
\dmo\Mod{Mod}
\dmo\PMod{PMod}
\dmo\genus{genus}
\dmo\Jac{Jac}
\dmo\fd{fd}
\dmo\IA{IA}
\dmo\Schur{{\Bbb S}}
\dmo\Sym{Sym}
\dmo\Ind{Ind}
\dmo\Res{Res}
\dmo\tr{tr}
\dmo\Irr{Irr}
\dmo\st{st}
\dmo\Stab{Stab}
\DeclareMathOperator*{\expect}{\mathbb{E}}
\DeclareMathOperator*{\E}{\mathbb{E}}
\dmo\End{End}
\def\et{\textit{\'et}}

\def\Fqbar{\overline{\F}_q}
\def\kbar{\overline{k}}
\dmo\Tor{\mathcal{T}}

\nc{\bwedge}{{\textstyle\bigwedge}}

\def\cF{\mathcal{F}}
\nc{\cst}{\chi^{\st}}
\nc{\ctwo}{\chi^{\wedge^2}}
\nc{\ck}{\chi^{\wedge^k}}
\nc{\case}[1]{\left\{\begin{array}{c}#1\end{array}\right.}
\nc{\abs}[1]{\left\lvert#1\right\rvert}
\dmo\St{St}

\dmo\FI{FI}
\dmo\FIMod{FI-Mod}
\dmo\dMod{-Mod}
\nc{\FIsharp}{\FI\sharp}
\nc{\FIsharpMod}{\FIsharp\dMod}
\dmo\Spec{Spec}
\nc{\op}{\text{op}}
\renewcommand{\epsilon}{\varepsilon}
\nc{\coloneq}{\mathrel{\mathop:}\mkern-1.2mu=}
\nc{\margin}[1]{\marginpar{\scriptsize #1}}
\nc{\para}[1]{\medskip\noindent\textbf{#1.}}

\begin{document}

\maketitle

\begin{abstract}
We consider two families $X_n$ of varieties on which the symmetric group $S_n$ acts:  the configuration space of $n$ points in $\C$ and the space of $n$ linearly independent lines in $\C^n$.  Given an irreducible $S_n$-representation $V$, one can ask how the multiplicity of $V$ in the cohomology groups $H^*(X_n;\Q)$ varies with $n$.  We explain how the Grothendieck--Lefschetz Fixed Point Theorem converts a formula for this multiplicity to a formula for the number of polynomials over $\F_q$ (resp.\ maximal tori in $\GL_n(\F_q)$) with specified properties related to $V$. In particular, we explain how representation stability in cohomology, in the sense of \cite{CF} and \cite{CEF}, corresponds to asymptotic stability 
of various point counts as $n\to \infty$.
\end{abstract}

%\tableofcontents

\section{Introduction}
%The goal of this paper is to explain a strong connection between two different phenomena:
%\begin{enumerate}
%\item representation stability (in the sense of \cite{CF} and \cite{CEF}) for the cohomology 
%of the complex points of a family $X_n$ of algebraic varieties; and 
%\item  combinatorial stability for counting problems on the $\F_q$-points $X_n(\F_q)$.
%\end{enumerate}
%
%We will see how each of these types of stability is 
%reflected in, and indeed can be converted to, the other.  This is accomplished via the Grothendieck--Lefschetz Theorem (with twisted coefficients) in \'{e}tale cohomology.  Of particular interest here is the use of coefficients twisted by representations of the symmetric group $S_n$, each such representation giving the solution to a different counting problem. 

In this paper we consider certain families $X_1,X_2,\ldots$ of algebraic varieties for which $X_n$ is endowed with a natural action of the permutation group $S_n$. In particular $S_n$ acts on the complex solution set $X_n(\C)$, and so each cohomology group $H^i(X_n(\C))$ has the structure of an $S_n$-representation. We will attach to $X_n$ a variety $Y_n$ over the finite field $\F_q$. The goal of this paper is to explain how representation stability  for $H^i(X_n(\C))$, in the sense of \cite{CF} and \cite{CEF}, corresponds to asymptotic stability for certain counting problems on the  $\F_q$-points $Y_n(\F_q)$,
and vice versa.

We will concentrate on two such families of varieties in this paper. The first family is the configuration space of $n$ distinct points in $\C$:
\[X_n(\C)=\PConf_n(\C)=\big\{(z_1,\ldots,z_n)\,\big|\,z_i\in \C, z_i\neq z_j\big\}\]
In this case $Y_n(\F_q)$ is the space $\Conf_n(\F_q)$ of monic squarefree degree-$n$  polynomials in $\F_q[T]$. The second family is the space of $n$ linearly independent lines in $\C^n$:
\[X_n(\C)=\big\{(L_1,\ldots,L_n)\,\big|\,L_i \text{ a line in }\C^n,\ \ L_1,\ldots,L_n \text{ linearly independent}\big\}\]
In this case $Y_n(\F_q)$ is the space parametrizing the set of maximal tori in the finite group $\GL_n(\F_q)$.  
In both cases, the action of $S_n$ on $X_n(\C)$ simply permutes the points $z_i$ or the lines $L_i$.

The relation between $X_n(\C)$ and $Y_n(\F_q)$ is given by the Grothendieck--Lefschetz fixed point theorem %(with twisted coefficients)
in \'{e}tale cohomology.  For any irreducible $S_n$-representation 
$V_n$ with character $\chi_n$, the Grothendieck--Lefschetz theorem with twisted coefficients $V_n$ can be thought of as a machine that, under sufficiently nice geometric circumstances, converts topological input to algebraic output, as follows:
\begin{equation*}
\boxed{\text{Multiplicity of $V_n$ in  $H^i(X_n;\C)$}}
 \  \text{ \Huge $\rightsquigarrow$}\ 
\boxed{\text{Point count in $Y_n(\F_q)$ weighted by $\chi_n$}}
\end{equation*}

%Here we use ``multiplicity'' to mean ``inner product of characters''; the two notions agree when $V_n$ is irreducible.  

Thus every  representation $V_n$ corresponds to a different counting problem or statistic on $Y_n(\F_q)$.
%See \S\ref{section:GL} and \S\ref{s:purebraid} below for many examples and for details.
For the two families of varieties we consider, the situation is so favorable that the input and output can even be reversed, allowing us to draw conclusions about cohomology from combinatorial point-counting results; this is certainly not the case in general.

\para{Representation stability} We can further ask about the \emph{asymptotics} of these statistics: for example, how does a given statistic for squarefree polynomials in $\F_q[T]$ vary as the degree of the polynomial tends to $\infty$? The answer is provided by \emph{representation stability}.

The cohomology groups $H^i(X_n(\C);\Q)$ were studied  for both of these families in \cite{CF} and \cite{CEF} (among many other papers), where we proved that these cohomology groups are representation stable as $n\to\infty$. This implies that the multiplicity of any irreducible $S_n$-representation (suitably stabilized) in $H^i(X_n(\C))$ is eventually constant. Via the Grothendieck--Lefschetz theorem, representation stability for $H^i(X_n (\C))$ implies an asymptotic stabilization for  statistics on $Y_n(\F_q)$ as $n\to \infty$.
%   This stability is then reflected in the combinatorial stability of various point counting problems on $X_n(\F_q)$.

\medskip
Our first result makes this connection precise for the first family, relating the cohomology of $X_n(\C)=\PConf_n(\C)$ with statistics on $Y_n(\F_q)=\Conf_n(\F_q)$, the space of monic squarefree degree-$n$ polynomials $f(T)\in \F_q[T]$.

If $f(T)$ is a polynomial in $\F_q[T]$, let $d_i(f)$ denote the number of irreducible degree $i$ factors of $f(T)$. For any polynomial $P\in \Q[x_1,x_2,\ldots]$, we have the ``polynomial statistic'' on $\Conf_n(\F_q)$ defined by  $P(f)=P(d_1(f),d_2(f),\ldots)$.
Similarly, let $\chi_P(\sigma)$ be the class function $\chi_P(\sigma)=P(c_1(\sigma),c_2(\sigma),\ldots)$ on $S_n$, where $c_i(\sigma)$ denote the number of $i$-cycles of $\sigma$. We define the degree $\deg P$ as usual, except that $\deg x_k=k$.%The results of \cite{CEF} imply that for any fixed polynomial $P$ and any $i\geq 0$, the inner product of class functions $\langle \chi_P,\chi_{H^i(\PConf_n(\C))}\rangle_{S_n}$ is eventually constant as $n\to \infty$.

\begin{theoremone}[{\bf Stability of polynomial statistics}]
\label{theorem:general}
For any polynomial $P\in\Q[x_1,x_2,\ldots ]$, the limit 
 \[\langle \chi_P,H^i(\PConf(\C))\rangle\coloneq
 \lim_{n\to \infty}\big\langle \chi_P,\chi_{H^i(\PConf_n(\C))}\big\rangle_{S_n} \] 
 exists; in fact, this sequence is constant for $n\geq 2i+\deg P$.  Furthermore, for each prime power $q$:
\beq
\lim_{n \ra \infty} q^{-n} \sum_{f \in \Conf_n(\F_q)} P(f) = \sum_{i=0}^\infty (-1)^i\frac{
\langle \chi_P,H^i(\PConf(\C))\rangle}{q^i}
\eeq
In particular, both the limit on the left and the series on the right converge, and they converge to the same limit.
\end{theoremone}

Theorem~\ref{theorem:general} is proved as Proposition~\ref{pr:purebraidlimit} below, as a special case of the more general Theorem~\ref{th:limit} for arbitrary FI-hyperplane arrangements. We also have an analogue of Theorem~\ref{theorem:general} for asymptotics of polynomial statistics  for maximal tori in $\GL_n(\F_q)$, which is proved as Theorem~\ref{th:torilimit}.

\medskip
Table A gives a sampling of the results that we will explain and prove in this paper. 
Formulas (1)-(5) in each column are obtained from the Grothendieck--Lefschetz theorem  with 
$V_n$ equal to the trivial representation, the standard representation $\C^n$, its exterior power $\bigwedge^2\C^n$, 
the sign representation, and the $n$-cycle character, respectively. In particular, Formulas (1)-(3) can be seen as applications of Theorem~\ref{theorem:general} with $P=1$, $P=X_1$, and $P=\binom{X_1}{2}-X_2$ respectively. One key message of this paper is that representation stability provides a single underlying mechanism  for all such formulas.

\begin{center}
\begin{tabular}{rll}
&Counting theorem for\qquad\qquad\qquad\qquad\ &Counting theorem for\\
&\underline{squarefree polys in $\F_q[T]$}&\underline{maximal tori in $\GL_n \F_q$}\\[14pt]

(1)&\# of degree-$n$ squarefree& \# of maximal tori in $\GL_n\F_q$\\
&polynomials $=q^n-q^{n-1}$&  (both split and non-split) $=q^{n^2-n}$\\[12pt]

(2)&expected \# of linear factors& expected \# of eigenvectors in $\F_q^n$\\
&$=1-\frac{1}{q}+\frac{1}{q^2}-\frac{1}{q^3}+\cdots \pm \frac{1}{q^{n-2}}$& $=1+\frac{1}{q}+\frac{1}{q^2}+\cdots +\frac{1}{q^{n-1}}$\\[12pt]

(3)&expected excess of \emph{irreducible} & expected excess of \emph{reducible} \\
&vs. \emph{reducible} quadratic factors &vs. \emph{irreducible} dim-2 subtori\\
&\ \ $\to\  \frac{1}{q}-\frac{3}{q^2}+\frac{4}{q^3}-\frac{4}{q^4}$&\ \ $\to\  \frac{1}{q}+\frac{1}{q^2}+\frac{2}{q^3}+\frac{2}{q^4}$\\
&\qquad$\ \ +\frac{5}{q^5}-\frac{7}{q^6}+\frac{8}{q^7}-\frac{8}{q^8}+\cdots$&\qquad$\ \ +\frac{3}{q^5}+\frac{3}{q^6}+\frac{4}{q^7}+\frac{4}{q^8}+\cdots$\\
&as $n\to \infty$&as $n\to \infty$\\[12pt]

(4)&discriminant of random squarefree&\# of irreducible factors is more\\
&polynomial is equidistributed in $\F_q^\times$&likely to be $\equiv n\bmod{2}$ than not,\\
&between residues and nonresidues&with bias $\sqrt{\text{\# of tori}}$\\[12pt]

(5)&Prime Number Theorem for $\F_q[T]$:&\# of irreducible maximal tori\\
& \# of irreducible polynomials &$=\frac{q^{\binom{n}{2}}}{n}(q-1)(q^2-1)\cdots(q^{n-1}-1)$\\
&\quad$=\sum_{d|n}\frac{\mu(n/d)}{n}q^{d}\sim \frac{q^n}{n} $&\qquad\qquad\qquad$\sim \frac{q^{n^2-n}}{n}$
\end{tabular}
Table A
\end{center}
%% TC: this is so that the equation numbering starts with (6)
\setcounter{equation}{5}

%\[
%\begin{array}{c}
%\text{\underline{\bf Input}}\\
%\text{A rational representation} \\
%S_n\to \GL(V)
%\end{array}\ \ \ 
%\begin{array}{c}
%\text{\underline{\bf Output}}\\
%\text{Formula for $\#$ of polynomials (resp.\ maximal}\\
%\text{tori in $\GL_n(\F_q)$) with specified properties}
%\end{array}
%\]
%
%\bigskip
%More specifically, here is a first list of examples exhibiting the input/output of this procedure for polynomials over $\F_q$.  We note that all squarefree polynomials considered are understood to be {\em monic} unless specified otherwise

The formulas in Table A are by and large not original to the present paper.  The formulas in the left column can be proved by direct means, and Lehrer has also analyzed them in the light of the Grothendieck--Lefschetz formula \cite{lehrer,lehrer:survey,lehrer:discriminantvarieties,kisinlehrer}. In contrast,  the formulas for maximal tori in $\GL_n \F_q$  may be known but are not so easy to prove. For example, formula (1) is the $\GL_n$ case of a well-known theorem of Steinberg; proofs using the Grothendieck--Lefschetz formula have been given by  Srinivasan~\cite{Sr} and Lehrer~\cite{lehrer:rationaltori}.

\para{Outline of paper}
This paper has two goals: 1) to provide a readable introduction to the connections between topology and combinatorics  given by the Grothendieck--Lefschetz theorem, and 2) to emphasize the \emph{stabilization} in these formulas as $n\to \infty$, and its connections with representation stability in topology. Although the details of our approach differ somewhat from the previous literature, our real aim is to make these connections accessible to a wider audience.

In the remainder of this introduction we give a detailed description, without proofs, of the connections between topology and combinatorics that lead to formulas like those in Table A.  In Section~\ref{section:GL} we give an introduction to the Grothendieck--Lefschetz theorem, with examples of its application to Theorem~\ref{theorem:general}. %(For further details, including the explanation of the precise powers of $q$ appearing in the Grothendieck--Lefschetz formulas, see Section~\ref{s:polynomials}.)
In Section~\ref{s:polynomials} we prove a general version of Theorem~\ref{theorem:general} for hyperplane complements  that can be generated in a uniform way by a finite set of ``generating hyperplanes''. In Section~\ref{s:purebraid} we focus on the configuration space $\PConf_n(\C)$ %, which is the complement of the hyperplanes $\{z_i=z_j\}$,
and prove the formulas (1)-(5) on the left side of Table A, as well as   formulas for more complicated statistics. In Section~\ref{s:maximaltori} we establish the analogue of Theorem~\ref{theorem:general} for maximal tori in $\GL_n\F_q$, and prove the formulas on the right side of Table A.

\subsection{Relating topology and combinatorics}
There are three distinct types of \emph{stability} present in the formulas in Table A, and each corresponds to a different topological phenomenon. We will describe each type of combinatorial stability in turn, and for each we will highlight its reflection on the topological side.

\para{Independence of $q$ and rational cohomology}
First, the formulas in Table A are independent of $q$ in some sense. Of course these point counts are not literally independent; we can check by examination that there are 18 squarefree cubic polynomials in $\F_3[T]$ (recalling our convention that squarefree polynomials are always taken to be monic):
\[
\begin{array}{llll}
&T^3+T&T^3-T\\
T^3+T+1&T^3+T-1&T^3-T+1&T^3-T-1\\
T^3+T^2-T&T^3-T^2-T&T^3+T^2+1&T^3-T^2+1\\
T^3-T^2-1&T^3+T^2+T+1&T^3-T^2+T+1&T^3+T^2-T+1\\
T^3+T^2+T-1&T^3+T^2-T-1&T^3-T^2+T-1&T^3-T^2-T-1
\end{array}
\] If we were to carry out the same count in $\F_{11}[T]$ we would find 1210  squarefree cubic polynomials, not 18. But once we notice that $18=3^3-3^2$ and $1210=11^3-11^2$, we see that these counts depend on $q$ in exactly the same way. In fact, formula (1) in Table A says that the number of  squarefree cubic polynomials in $\F_q[T]$ is always $q^3-q^2$. 

The same independence arises in many common point-counting situations: for example, the number of lines in $\F_q^3$ is $q^2+q+1$, no matter what $q$ is.
%\[|\mathbf{P}^{N-1}(\F_q)|=1+q+q^2+\cdots+q^{N-1}. \]
%In many such situations, including those described in this paper, 
The Grothendieck--Lefschetz theorem explains these coincidences as reflecting the underlying topology of the complex points of an algebraic variety.
In particular, we can match the terms occurring in each point-counting formula with those rational cohomology groups that are nonzero, providing a surprising bridge between topology and arithmetic. 

As a simple example, consider the problem of counting the number of lines in $\F_q^3$; that is, the number of points in the projective space $\PP^2(\F_q)$.  The corresponding variety is $\CP^2$, the topological space of complex lines in $\C^3$.  It is easy to compute by hand that 
\[H^0(\CP^2)=\Q,\ \  H^2(\CP^2)=\Q,\ \  H^4(\CP^2)=\Q,\]
and these three nonzero cohomology groups correspond to the three terms of the point-counting formula
\[|\PP^2(\F_q)|=q^2+q+1.\]

For the count of  squarefree cubic polynomials in $\F_q[T]$, the corresponding variety is the topological space of  squarefree cubic complex polynomials, which we denote by $\Conf_3(\C)$:
\[\Conf_3(\C)=\big\{f(z)=z^3+bz^2+cz+d\,\big\vert\,b,c,d\in \C,\ 
f(z)\text{ is squarefree}\big\}\]
By considering the coefficients $(b,c,d)\in \C^3$, we can identify $\Conf_3(\C)$ with the complement in $\C^3$ of the discriminant locus, where $b^2c^2-4c^3-4b^3d-27d^2+18bcd=0$. By a direct calculation we find that $H^0(\Conf_3(\C))=\Q$ and $H^1(\Conf_3(\C))=\Q$, but that all other cohomology groups vanish.
These two nonzero cohomology groups correspond respectively to the two terms of the formula (1):
\[|\Conf_3(\F_q)|=q^3-q^2\quad\longleftrightarrow\quad H^0(\Conf_3(\C))=\Q,\ \ H^1(\Conf_3(\C))=\Q\]

\para{Asymptotics of counts and homological stability}
A second form of stability  in the formulas in Table A is that  they  are in some sense independent of $n$. As before, we know that the counts cannot literally be   independent of $n$.
%\,---\,the number of squarefree cubic polynomials in $\F_3[T]$ is $18=3^3-3^2$, while the number of squarefree quartic polynomials is $54=3^4-3^3$\,---\,
Nevertheless, the single formula $q^n-q^{n-1}$ gives the number of all squarefree degree-$n$ polynomials in $\F_q[T]$ for \emph{all} $n\geq 2$.

The set of all squarefree, degree $n$ polynomials in $\C[T]$ is the complex algebraic variety 
\[\Conf_n(\C)=\big\{f(z)\in \C[T]\,\big\vert\,\deg f(z)=n,\ \ f(z)\text{ is squarefree}\big\}.\]
The stability of the formula $q^n-q^{n-1}$ as $n$ increases reflects  \emph{homological stability} for the topological spaces $\Conf_n(\C)$: Arnol'd  proved that for \emph{any} $n\geq 2$ the space $\Conf_n(\C)$ has the rational cohomology of a circle. Therefore for any $n\geq 2$ there are two nonzero cohomology groups of $\Conf_n(\C)$, which correspond to the two terms of the formula (1): \[\left|\Conf_n(\F_q)\right|=q^n-q^n\quad\longleftrightarrow\quad H^0(\Conf_n(\C))=\Q,\ \ H^1(\Conf_n(\C))=\Q\]

This situation is  simpler than most, since here we have not just \emph{stability} of cohomology, but actually \emph{vanishing} of cohomology: for $i\geq 2$ we have $H^i(\Conf_n(\C);\Q)=0$ for all $n$.  In general, homological stability for a sequence of spaces $X_n$ only means that $H^i(X_n)$ is independent of $n$ for $n\gg i$. A more representative example is given by the number of lines in $\F_q^{n+1}$ as $n$ varies. The corresponding topological space is $\CP^n$, the space of complex lines in $\C^{n+1}$. These projective spaces do exhibit homological stability, since $H^*(\CP^n)=\Q[x]/(x^{n+1})$ with $x\in H^2(\CP^n)$. Working degree-by-degree, this means that $H^{2k}(\CP^n)=\Q$ for all $n\geq k$, while $H^{2k+1}(\CP^n)=0$ for all $n\geq 0$. Therefore the $n+1$ nonzero cohomology groups of $\CP^n$ correspond to the $n+1$ terms of the point-counting formula. \[|\PP^n(\F_q)|=q^n+q^{n-1}+\cdots+q+1\quad\longleftrightarrow\quad
H^{2k}(\CP^n)=\Q,\ \ k=0,1,\ldots,n-1,n.\]
Under this correspondence, the stabilization of $H^{2k}(\CP^n)$ for $n\geq k$ corresponds to the stabilization of the $q^{n-k}$ term on the left side once $n\geq k$.

\para{Combinatorial statistics and representation stability}
Both of the previous types of stability are well-understood, both on the topological and combinatorial side. Our focus in this paper is on a new kind of combinatorial stability, whose topological reflection is the \emph{representation stability} of \cite{CF} and \cite{CEF}. The rest of the introduction will be spent explaining this connection.

The new feature here is that we are not just counting squarefree polynomials, but certain \emph{combinatorial statistics} associated to them.
Let us focus on  formula (2) on the left side in Table A. This formula says that if a squarefree polynomial of degree $n$ over $\F_q$ is chosen at random, we can expect that it will have slightly less than 1 linear factor on average. For example, for the squarefree cubic polynomials in $\F_3[T]$ that we listed above, the number of linear factors is:
\[
\begin{array}{rlllll}
3:&T^3-T\\
1:&T^3-T^2-T&T^3+T+1&T^3+T-1&T^3+T^2-T&T^3+T\\
&T^3+T^2+1&T^3-T^2-1&T^3+T^2+T+1&T^3-T^2+T-1\\
0:&T^3-T+1&T^3-T-1&T^3-T^2+1
&T^3-T^2+T+1\\&T^3+T^2-T+1
&T^3+T^2+T-1&T^3+T^2-T-1&T^3-T^2-T-1
\end{array}
\] 
Therefore if we randomly select from these 18 possibilities, the expected number of linear factors that our chosen polynomial will have is\[\frac{(3\cdot 1)+(2\cdot 0)+(1\cdot 9)+(0\cdot 8)}{18}=\frac{12}{18}=\frac{2}{3}=1-\frac{1}{3}.\phantom{+\frac{1}{9}\frac{1}{27}}\] For the 54 squarefree quartic polynomials in $\F_3[T]$, we would find that none has three linear factors, 9 have two linear factors, 24 have one linear factor, and the remaining 21 have no linear factors at all. (Of course, no squarefree polynomial in $\F_3[T]$ can have \emph{four} linear factors, since there are only three elements of $\F_3$ which could be its roots!) Thus the expected number of linear factors in this case is
\[\frac{(3\cdot 0)+(2\cdot 9)+(1\cdot 24)+(0\cdot 21)}{54}=\frac{42}{54}=\frac{7}{9}=1-\frac{1}{3}+\frac{1}{9}.\phantom{\frac{1}{27}}\]
For the 162 squarefree quintic polynomials the same computation of the expectation would take the form:
\[\frac{(3\cdot 3)+(2\cdot 24)+(1\cdot 63)+(0\cdot 72)}{162}=\frac{120}{162}=\frac{20}{27}=1-\frac{1}{3}+\frac{1}{9}-\frac{1}{27}.\]
For the 146,410 squarefree quintic polynomials in $\F_{11}[T]$, the computation is a good deal more complicated, now involving polynomials with up to five linear factors. Yet the formula (2) tells us that the expectation must work out to exactly
\[\frac{134200}{146410}=\frac{1220}{1331}=1-\frac{1}{11}+\frac{1}{121}-\frac{1}{1331}.\]

\para{Twisted cohomology} Topologically, these formulas are still explained by the cohomology of $\Conf_n(\C)$, but now with certain \emph{twisted} coefficients, which we now describe.

Given a squarefree degree-$n$ polynomial $f(T)\in \C[T]$, its set of roots 
\[R(f)\coloneq\{\lambda\in \C\, | f(\lambda)=0\}\] varies continuously as we vary $f(T)$. Therefore the space $V$ defined by \[V\coloneq \big\{\,\big(f(T)\in \Conf_n(\C),\ \ h\colon R(f)\to \Q\,\big)\big\}\] has a continuous map $V\to \Conf_n(\C)$ given by $(f(T),h)\mapsto f(T)$.
Since $|R(f)|=n$ for all $f(T)\in \Conf_n(\C)$, we can think of $V$ as a vector bundle $\Q^n\to V\to \Conf_n(\C)$, where the fiber over $f(T)\in \Conf_n(\C)$ is the $\Q$-vector space of functions on the set of roots $R(f)$. 

Note that there is no natural choice of ordering for the roots in $R(f)$, so we cannot find a global trivialization of the vector bundle $V$. However for small deformations $f_t(T)$ of the polynomial $f(T)$, the set of roots $R(f_t)$ is close to $R(f)$, and so we do have a canonical identification between $R(f)$ and $R(f_t)$ by which we can transfer $h\colon R(f)\to \R$ to $h_t\colon R(f_t)\to \R$. (For example, there is an ``obvious'' bijection between the roots of $(T-1)(T-2)(T-3)$ and $(T-2.01)(T-3.01)(T-1.01)$, even though we cannot talk about the ``first root'' of either polynomial.) This gives $V$ the structure of a flat vector bundle (also called a \emph{local system}) over $\Conf_n(\C)$.

We denote by $H^i(\Conf_n(\C);\Q^n)$  the twisted cohomology of $\Conf_n(\C)$ with coefficients in $V$; it is these cohomology groups that correspond to the counts of linear factors in the formula (2). For example, we will compute in Section~\ref{sec:standard} that $H^i(\Conf_5(\C);\Q^5)=$
\[\Q\text{ for }i=0,\quad \Q^2\text{ for }i=1,\quad \Q^2\text{ for }i=2,\quad \Q^2\text{ for }i=3,\quad \Q\text{ for }i=4,\] and $0$ for $i\geq 5$. %\[H^i(\Conf_5(\C);\Q^5)=\begin{cases}\Q&i=0\\\Q^2&i=1,2,3\\\Q&i=4\\0&i\geq 5\end{cases}\]
This corresponds to the fact that the total number of linear factors over all squarefree quintic polynomials in $\F_3[T]$ is
\[3^5\ -\ 2\cdot 3^4\ +\ 2\cdot 3^3\ -\ 2\cdot 3^2\ +\ 3
%=243-162+54-18+3
=120\]
while in $\F_{11}[T]$ the total number is
\[11^5\ -\ 2\cdot 11^4\ +\ 2\cdot 11^3\ -\ 2\cdot 11^2\ +\ 11=134200.\]
These are precisely the numerators of the fractions $\frac{120}{162}$ and $\frac{134200}{146410}$ that we computed above. (The denominators arise because the natural quantity to count is the \emph{expected} number, rather than the \emph{total} number, of linear factors.) 
These computations allow us to give another derivation of some recent results of Kupers--Miller \cite{KM}, in relation to a prediction made by Vakil--Wood \cite{VW}; see Section~\ref{sec:standard} for details.

\para{Combinatorial statistics and local systems} At this point, one should ask why the twisted coefficient system $\Q^n$  corresponds to the number of linear factors of a polynomial in $\F_q[T]$, rather than some other statistic. 

Let $g(T)\in \Conf_n(\F_q)$ be a squarefree polynomial with coefficients in $\F_q$.  
The Frobenius map $\Frob_q\colon \Fqbar\to \Fqbar$ defined by $x\mapsto x^q$ fixes exactly the elements of $\F_q\subset \Fqbar$. 
%Just as the roots of a real polynomial are invariant under complex conjugation, 
Since $\Frob_q$ fixes each of the coefficients of $g(T)$, it therefore must permute the set of roots $R(g)=\{\lambda\in \Fqbar\,|\,g(\lambda)=0\}$.   
%Moreover, a real polynomial factors into linear factors (for each real root fixed by conjugation) and irreducible quadratic factors (for each pair of complex roots exchanged by conjugation). 
%The same is true for $g(T)\in \F_q[T]$: 
If $\sigma_g$ is the permutation of the roots $R(g)$ induced by $\Frob_q$, each length-$k$ orbit of $R(g)$ under $\sigma_g$ corresponds to a single irreducible factor of $g(T)$ of degree $k$.
In the language of Theorem~\ref{theorem:general}, the number of degree-$i$ factors of $g(T)$ was denoted $d_i(g)$, while $c_i(\sigma)$ denoted the number of $i$-cycles in $\sigma$. We can summarize this discussion as
\[d_i(g)=c_i(\sigma_g).\]
In particular, the number of linear factors of $g(T)$ is $c_1(\sigma_g)$.

This permutation of the roots has a parallel in the topological picture: any loop $\gamma(t)=f_t(T)$ in $\Conf_n(\C)$ beginning and ending at $f(T)$ induces a permutation  $\sigma_\gamma$ of the roots $R(f)$, by continuing the identification $R(f)\simeq R(f_t)$ around the loop $\gamma(t)$.
Our construction of $V$ guarantees that the monodromy $\gamma_*\colon V_f\to V_f$ given by transporting the fiber $V_f$ along this loop is the matrix representation of the permutation $\sigma_\gamma$. In particular, the \emph{trace} $\chi_V(\gamma)=\tr\gamma_*$ is the number of fixed points $c_1(\sigma_\gamma)$ of the permutation $\sigma_\gamma$. This is why the coefficient system $V$ corresponds to counting linear factors, rather than some other statistic. %; see Section~\ref{section:???} for full details.
%As we will explain in Section~\ref{section:???}, this is the
%the permutation $\sigma_g$ the action of $\Frob_q$ on the roots of $g(T)$ corresponds to the the monodromy $\gamma_*$, and 

\para{Finding appropriate coefficient systems}
In general, say that we want to understand the polynomial statistic $P(g)=P(d_1(g),d_2(g),\ldots)$ of $g(T)\in \Conf_n(\F_q)$ for some polynomial $P\in \Q[X_1,X_2,\ldots]$. Then we need to find a coefficient system $W$ on $\Conf_n(\C)$ for which $\tr \gamma_*\colon V_f\to V_f$ is given by $\chi_P(\sigma_\gamma)=P(c_1(\sigma_\gamma),c_2(\sigma_\gamma),\ldots)$. Once we've found $W$, the expected value of the statistic $P(g)$ can be read off the twisted cohomology $H^*(\Conf_n(\C);W)$.
%In particular, we can explain other statistics of squarefree polynomials by constructing other twisted coefficient systems on $\Conf_n(\C)$ with prescribed characters.
%For example, an irreducible quadratic factor of $g(T)\in \Conf_n(\F_q)$ corresponds to a pair of roots $\{\lambda,\overline{\lambda}\}$ in $\F_{q^2}$ which are exchanged by $\Frob_q$. Therefore to count irreducible quadratic factors,  we would need to find a coefficient system $U$ on $\Conf_n(\C)$ for which the trace $\chi_U(\gamma)=\tr \gamma_*$ is the number of \emph{transpositions} in $\sigma_\gamma$.
Fortunately, we can do this for \emph{any} statistic! It is actually not possible to realize every polynomial statistic itself by a single coefficient system, but we can always express it as a \emph{linear combination} of statistics for which the necessary coefficient system can be constructed. %In fact, in this way we can explain via topology \emph{any} statistic for squarefree 
%degree $n$ polynomials that depends only on the factorization into irreducibles; see Theorem~\ref{theorem:general} below. 

\para{Irreducible versus reducible quadratic factors} In the formula (3) of Table A we compare the numbers of \emph{irreducible} versus \emph{reducible} quadratic factors; we'll refer to the difference of these statistics as the \emph{quadratic excess} of a polynomial. An irreducible quadratic factor of $g(T)$ corresponds to a pair of roots $\{\lambda,\overline{\lambda}\}$ which are exchanged by $\Frob_q$, or a 2-cycle of $\sigma_g$. In the same way, a reducible quadratic factor corresponds to a pair of roots which are each fixed by $\Frob_q$, or a pair of fixed points of $\sigma_g$. %This shows the quadratic excess of a degree-$n$ squarefree polynomial is bounded above by $\frac{n}{2}$, and  below by $-\binom{n}{2}$. The asymmetry in these bounds might lead us to guess that the quadratic excess would be negative on average, but we will see  that this is not the case.
Therefore the statistic we are looking for is $P(g)$ with $P=\binom{X_1}{2}-X_2$.

This is realized by the coefficient system
$W=\bwedge^2 V$. The fiber $W_f$ has basis $e_\lambda\wedge e_{\lambda'}$ for each pair $\lambda\neq \lambda'$ of roots in $R(f)$, and the monodromy $\gamma_*$ permutes these basis elements according to the action of $\sigma_f$ on $R(f)$. If $\lambda$ and $\lambda'$ are both fixed by $\sigma_f$ we have $e_\lambda\wedge e_{\lambda'}\mapsto e_\lambda\wedge e_{\lambda'}$, while if $\lambda$ and $\lambda'$ are exchanged by $\sigma_f$ we have $e_\lambda\wedge e_{\lambda'}\mapsto -e_\lambda\wedge e_{\lambda'}$. Therefore the trace $\chi_W(\gamma)=\tr\gamma_*\colon W_f\to W_f$ is given by
\[\chi_W(\gamma)=\binom{\text{\# fixed points of }\sigma_\gamma}{2}-\text{\# transpositions of }\sigma_\gamma=\binom{c_1(\sigma_\gamma)}{2}-c_2(\sigma_\gamma)=\chi_P(\sigma_\gamma)\]
as desired. Therefore topologically, the quadratic excess can be computed from the cohomology $H^i(\Conf_n(\C);\bwedge^2 \Q^n)$.

For a concrete example, we can compute that $H^i(\Conf_5(\C);\bwedge^2 \Q^5)=$
%\begin{equation}
%\label{eq:intro:wedge2mult}
\[0\text{ for }i=0,\quad \Q\text{ for }i=1,\quad \Q^4\text{ for }i=2,\quad \Q^5\text{ for }i=3,\quad \Q^2\text{ for }i=4,\]
%\end{equation}
and $0$ for $i\geq 5$.
This tells us that the total quadratic excess of squarefree quintics in $\F_q[T]$ will be %\begin{equation}
%\label{eq:intro:quad}
\[q^4-4q^3+5q^2-2q.\]
%\end{equation}
Dividing by $\left|\Conf_5(\F_q)\right|=q^5-q^4$, we find that the expected value of the quadratic excess is $\frac{1}{q}-\frac{3}{q^2}+\frac{2}{q^3}$.

\para{Twisted homological stability} Finally, we arrive at our real focus in this paper: the \emph{stabilization} of formulas such as (2) and (3) as $n\to\infty$. If we extended the above computations of quadratic excess to polynomials of higher degree, we would find:
\begin{align*}
&\text{total: }&&\text{expectation: }\\
n=5:\qquad&q^4-4q^3+5q^2-2q&&\frac{1}{q}-\frac{3}{q^2}+\frac{2}{q^3}\\
n=6:\qquad&q^5-4q^4+7q^3-7q^2+3q&&\frac{1}{q}-\frac{3}{q^2}+\frac{4}{q^3}-\frac{3}{q^4}\\
%
%n=6: 2H^2+5H^3+5H^4+2H^5
%+ 0 1 2 2 2 1 0
%= H^1+4H^4+7H^3+7H^4+3H^5
%
%
%n=7: 2H^2+5H^3+6H^4+6H^5+3H^6
%+ 0 1 2        2         2        2         1
%= H^1+4H^2+7H^3+8H^4+8H^5+4H^6
%
%n=8: 2H^2+5H^3+6H^4+7H^5+8H^6+4H^7
%+ 0 1   2       2         2        2         2        1
%= 1   4    7     8    9    10   5
n=7:\qquad&q^6-4q^5+7q^4-8q^3+8q^2-4q&&\frac{1}{q}-\frac{3}{q^2}+\frac{4}{q^3}-\frac{4}{q^4}+\frac{4}{q^5}\\
n=8:\qquad&q^7-4q^6+7q^5-8q^4+9q^3-10q^2+4q\quad&&\frac{1}{q}-\frac{3}{q^2}+\frac{4}{q^3}-\frac{4}{q^4}+\frac{5}{q^5}-\frac{5}{q^6}\\
%n=9:\qquad&q^8-4q^7+7q^6-8q^5+9q^4-12q^3+12q^2-5q&&\frac{1}{q}-\frac{3}{q^2}+\frac{4}{q^3}-\frac{4}{q^4}+\frac{5}{q^5}-\frac{7}{q^6}+\frac{5}{q^7}\\
\end{align*}
We see that  these formulas are converging term-by-term to 
\begin{equation}
\label{eq:intro:limitqe}
q^{n-1}-4q^{n-2}+7q^{n-3}-8q^{n-4}+\cdots\qquad\quad\text{and}\quad\qquad\frac{1}{q}-\frac{3}{q^2}+\frac{4}{q^3}-\frac{4}{q^4}+\cdots,
\end{equation} as claimed in (3). Just as the stabilization of simple point-counts was explained by homological stability, the term-by-term stabilization of these statistics corresponds to a \emph{stabilization of twisted cohomology}:
\begin{align*}
H^1(\Conf_n(\C);\bwedge^2 \Q^n)&\ =\ \Q \ \text{ for all }n\geq 4\\
H^2(\Conf_n(\C);\bwedge^2 \Q^n)&\ =\ \Q^4 \text{ for all }n\geq 5\\
H^3(\Conf_n(\C);\bwedge^2 \Q^n)&\ =\ \Q^7 \text{ for all }n\geq 6\\
H^4(\Conf_n(\C);\bwedge^2 \Q^n)&\ =\ \Q^8 \text{ for all }n\geq 7,\ \  \text{ and so on.}
\end{align*} 
We can also approach this connection from the other direction, as we do in Section~\ref{sec:Lfunctions}: the formula \eqref{eq:intro:limitqe}  can be proved directly via analytic number theory, which then yields a \emph{proof} that the stable twisted cohomology is as we've claimed here. 

\para{Representation stability} There remains one final layer to uncover. The quadratic excess and number of linear factors are not the only statistics that stabilize as $n\to \infty$. In fact, \emph{any} statistic built as a polynomial in the counts of the numbers of factors of various degrees will stabilize in the same way. On the topological side, this means that the twisted cohomology of $\Conf_n(\C)$ must stabilize not just for the coefficient systems $\Q^n$ and $\bwedge^2 \Q^n$, but for $\bwedge^k \Q^n$, $\Sym^k \Q^n$, and many other natural sequences of coefficient systems. What is the underlying explanation?

Given a squarefree polynomial $f(T)\in \Conf_n(\C)$, the $n$-element set of roots $R(f)=\{\lambda\in \C\, | f(\lambda)=0\}$ varies continuously, describing an $n$-sheeted cover of $\Conf_n(\C)$. This cover is not normal, and its Galois closure is an $S_n$-cover of $\Conf_n(\C)$. The resulting $S_n$-cover is the hyperplane complement
\[\PConf_n(\C)=\big\{(\lambda_1,\ldots,\lambda_n)\,\big\vert\, \lambda_i\in \C,\ \lambda_i\neq \lambda_j\big\}\]
covering $\Conf_n(\C)$ by sending $(\lambda_1,\ldots,\lambda_n)$ to the polynomial $f(T)=(T-\lambda_1)\cdots(T-\lambda_n)$ with those roots. Lifting $f(T)$ to $(\lambda_1,\ldots,\lambda_n)\in \PConf_n(\C)$ amounts to choosing an ordering of the roots, and the deck group $S_n$ acts on $\PConf_n(\C)$  by permuting the ordering.

\medskip
When pulled back to the cover $\PConf_n(\C)$, the twisted coefficient systems $\Q^n$ and $\bwedge^2 \Q^n$ become trivial vector bundles with a nontrivial action of the Galois group $S_n$, i.e.\ \emph{representations} of the group $S_n$.
The rational cohomology $H^i(\PConf_n(\C))$ is also a representation of $S_n$ via the action of the deck group, and the transfer map for the finite cover $\PConf_n(\C)\to \Conf_n(\C)$ gives natural isomorphisms

\begin{align*}
H^i(\Conf_n(\C);\Q^n)\ \ \ \ \ \ &\approx\ \  H^i(\PConf_n(\C))\ \otimes_{S_n}\  \Q^n\\
H^i(\Conf_n(\C);\bwedge^2\Q^n)\ \ &\approx\ \  H^i(\PConf_n(\C))\ \otimes_{S_n}\  \bwedge^2\Q^n\\
H^i(\Conf_n(\C);V_n)\ \ \ \ \,\ \ &\approx\ \  H^i(\PConf_n(\C))\ \otimes_{S_n}\  V_n
\end{align*} 
Every $S_n$-representation is self-dual (since every element $\sigma\in S_n$ is conjugate to its inverse), so the dimension of such a tensor product is the inner product of $S_n$-characters 
\[\dim(V\otimes_{S_n} W)=\dim(\Hom_{S_n}(V,W))=\langle \chi_V,\chi_W\rangle_{S_n}.\]
We think of this inner product as the ``multiplicity of $W$ in $V$'', as this is the case when $W$ is irreducible.
 Therefore the stabilization of  twisted cohomology of $\Conf_n(\C)$ that explains formula (1)  amounts to the statement that for each $i\geq 0$ the multiplicity of $\Q^n$ in the $S_n$-representation $H^i(\PConf_n(\C))$ is eventually constant. Similarly, the stabilization of formula (2) means the multiplicity of $\bwedge^2 \Q^n$ in $H^i(\PConf_n(\C))$ is eventually constant, and so on.
This property of $H^i(\PConf_n(\C))$, that the multiplicity of natural families of representations is eventually constant, is precisely the \emph{representation stability} introduced and proved in Church--Farb~\cite{CF}.

\para{Character polynomials and FI-modules} What makes a family of $S_n$-representations $W_n$ ``natural'' in this way? There are many possible answers, but for us we ask that their characters are given by a single polynomial $P$ simultaneously for all $n$. For example, we saw earlier that the character of $\Q^n$ is given by $\chi_{X_1}$ for all $n\geq 1$, and the character of $\bwedge^2 \Q^n$ is given by $\chi_{\binom{X_1}{2}-X_2}$ for all $n\geq 1$. Therefore the multiplicities that we are interested in will be inner products of the form $\langle \chi_P,H^i(\PConf_n(\C))\rangle_{S_n}$.

%Since the cycle decomposition of $\sigma\in S_n$ is invariant under conjugacy, the number of $i$-cycles in $\sigma$ defines a class function on $S_n$. For any $i\geq 1$, we let $X_i$ denote this class function. Extending linearly, any polynomial $P\in\Q[X_1,\ldots ,X_d]$ determines a class function $\chi_P$ on $S_n$. For example, \[P=X_1^2-X_2\quad \leadsto\quad 
%\chi_P(\sigma)=(\text{\# fixed points of }\sigma)^2-\text{\# transpositions of }\sigma\] We call $P$, or sometimes $\chi_P$, a {\em character polynomial} (see \S\ref{section:stabilization}). A key feature is that a single polynomial $P$ determines characters $\chi_P\colon S_n\to \Q$ for all $n$ simultaneously, and many natural families of $S_n$-representations have characters given by a single character polynomial. For example, the character of $\Q^n$ is given by $P=X_1$ for all $n\geq 1$, and the character of $\bwedge^2 \Q^n$ is given by $P=\binom{X_1}{2}-X_2$.
%
%If $W$ is an $S_n$-representation, let us write $\langle P,W\rangle_{S_n}$ for the inner product $\langle \chi_P,\chi_W\rangle_{S_n}$. The multiplicities of $\Q^n$ and $\bwedge^2 \Q^n$ that we are interested in are therefore given by inner products such as $\langle X_1, H^i(\PConf_n(\C))\rangle_{S_n}$ or $\langle \binom{X_1}{2}-X_2, H^i(\PConf_n(\C))\rangle_{S_n}$.

Analyzing how  multiplicities $\langle \chi_P,V_n\rangle_{S_n}$ change as $n\to \infty$ is one of the main reasons that we introduced \emph{FI-modules} in \cite{CEF}. An FI-module bundles a sequence of $S_n$-representations such as $V_n=H^i(\PConf_n(\C))$ into a single mathematical object $V=H^i(\PConf(\C))$, in such a way that representation stability for $V_n$ is equivalent to finite generation for $V$. One of the main theorems of \cite{CEF} states that for any fixed character polynomial $P$ and any finitely-generated FI-module $V$, the inner products $\langle \chi_P,V_n\rangle_{S_n}$ are eventually constant.

 We have finally found the common cause underlying the stabilization of the combinatorial formulas in (2) and (3): it is the fact \cite[Theorem~4.7]{CEF} that $H^i(\PConf(\C))$ is a finitely-generated FI-module. And just as the formulas in (2) and (3) converged to a fixed power series as $n\to\infty$, the same will be true for \emph{any} polynomial statistic.
 
\para{Error bounds and stable range} Theorem~\ref{theorem:general} states that  the normalized statistic $q^{-n}\sum_{f(T)\in \Conf_n(\F_q)} P(f)$ converges to a limit $L$, but  says nothing about how \emph{fast} this statistic converges to the limit. It turns out that bounding the error term of this convergence is closely related to the question of a \emph{stable range} for representation stability, as we briefly explain. We prove Theorem~\ref{theorem:general} by first establishing the exact formula on the left (Proposition~\ref{pr:exactcount}), and then proving that it converges to the limit $L$ on the right.
\begin{equation}
\label{eq:power}
q^{-n}\!\!\!\!\!\sum_{f(T)\in \Conf_n(\F_q)}\!\!\!\!\!\!\!\! P(f)=\sum_{i=0}^n \frac{\langle\chi_P,H^i(\PConf_n(\C))\rangle_{S_n}}{(-q)^{i}} \underset{n\to \infty}\longrightarrow \sum_{i=0}^\infty \frac{
\langle \chi_P,H^i(\PConf(\C))\rangle}{(-q)^{i}}=L
\end{equation} In Sections~\ref{s:polynomials} and \ref{s:purebraid} we will find two obstacles governing the speed of this convergence. First, we must eliminate the possibility that the inner products $\langle \chi_P,{H^i(\PConf_n(\C))}\rangle_{S_n}$ could grow exponentially in $i$; without this, the series defining $L$ may not even be convergent! For general hyperplane complements this is a real obstacle, but for $\PConf_n(\C)$ we will be able to bound these inner products using known results. Once this is dealt with, we still need to know how large $n$ must be take before $\langle \chi_P,H^i(\PConf_n(\C))\rangle_{S_n}$ stabilizes to the limiting value $\langle \chi_P,H^i(\PConf(\C))\rangle$.

In general, the range $n\geq N_i$ for which some cohomology group $H^i$ is equal to its limiting value is known as the stable range; if there exist $K$ and $C$ so that $H^i$ stabilizes for $n\geq K\cdot i+C$, we say the problem has a linear stable range. Define the \emph{degree} of $P\in \Q[X_1,X_2,\ldots]$ as usual, except that $\deg X_k=k$. From \cite[Theorem~4.8]{CEF} and Proposition~\ref{pr:charpolyexpect},  one can show that $\langle \chi_P,H^i(\PConf_n(\C))\rangle_{S_n}$ is constant for all $n >= 2i+\deg P$. One can then deduce from \eqref{eq:power} that \[q^{-n}\sum_{f(T)\in \Conf_n(\F_q)} P(f)=L+O(q^{(\deg P-n)/2})=L+O(q^{-n/2}).\] 
Such a bound on the error term, of the form $L+O(q^{-\epsilon n})$, is called a \emph{power-saving bound}. This discussion shows that any \emph{linear stable range} $n\geq Ki + C$ for $\langle \chi_P,H^i(\PConf_n(\C))\rangle_{S_n}$ implies a power-saving bound for the count $q^{-n}\sum P(f)$, with $\epsilon=\frac{1}{K}$.

\para{Acknowledgments}
The authors thank Brian Conrad, Andrew Granville, Michael Lugo, and Akshay Venkatesh for useful conversations on the subject matter of the paper.  The first author was supported by NSF grant DMS-1103807, the second author was supported by NSF grant DMS-1101267, and the third author was supported by NSF grant DMS-1105643.

\section{The twisted Grothendieck--Lefschetz Formula}
\label{section:GL}

The Grothendieck--Lefschetz formula is a device that relates the topology of algebraic varieties over the complex numbers to the number of points of varieties over finite fields, thus providing a surprising bridge between topology and arithmetic.  The goal of this section is to give an introduction to the Grothendieck--Lefschetz formula by working it out explicitly in a few basic examples.     

\subsection{Background on Grothendieck--Lefschetz}

\para{$X(\F_q)$ as set of fixed points} We begin with a variety $X$ defined over the finite field $\F_q$.  The main arithmetic invariant of $X$ is its number of $\F_q$-points, $|X(\F_q)|$.   It is a fundamental observation that one can realize $X(\F_q)$ as the fixed points of a dynamical system as follows. Since $X$ is defined over $\F_q$, we have the geometric Frobenius morphism $\Frob_q\colon X\to X$, which acts (in an affine chart) on the coordinates by $x\mapsto x^q$. If $\Fqbar$ is the algebraic closure of $\F_q$, the morphism $\Frob_q$ acts on the set $X(\Fqbar)$ of $\Fqbar$-points of $X$. A point $x\in X(\Fqbar)$ will be fixed by $\Frob_q$ if all of its coordinates are fixed by $x\mapsto x^q$, i.e.\ if all its coordinates lie in $\F_q$, i.e.\ if $x$ lies in $X(\F_q)$. Therefore we have
\[X(\F_q)=\Fix\big(\Frob_q\colon X(\Fqbar)\to X(\Fqbar)\big).\]

\para{Grothendieck--Lefschetz formula} For an endomorphism $f\colon Y\to Y$ of a compact topological space $Y$, the classical Lefschetz fixed point formula lets us count the fixed points of $f$ in terms of the induced action of $f^\ast$ on 
the cohomology $H^\ast(Y;\Q)$. Specifically, it says that for nice maps $f$,
\begin{equation}
\label{eq:compactLefschetz}
\#\Fix(f\colon Y\to Y)=\sum_{i\geq 0}(-1)^i \tr\big(\Frob_q:H^i(Y;\Q)\big).
\end{equation}
%We can even do this with twisted coefficients, in which case the Lefschetz formula relates a sum of local factors over the fixed points of $f$ with the global action of $f^*$ on $H^\ast(Y;\FF)$.

Grothendieck's great insight was that the fixed points of $\Frob_q$ could be analyzed in the same fashion.
Since $X(\F_q)$ is finite and $X(\Fqbar)$ is totally disconnected in the standard topology, it might seem strange to talk about the ``topology" of $X$. But Grothendieck showed that the variety $X$ in fact has a cohomology theory with many of the familiar properties of the classical theory, called the \emph{\'etale} cohomology $H^i_{\et}(X;\Q_\ell)$. (For experts, we write $H^i_{\et}(X;\Q_\ell)$ as an abbreviation for $H^i_{\et}(X_{/\Fqbar};\Q_\ell)$, the \'{e}tale cohomology of the base change $X_{\Fqbar}$ of $X$ to $\Fqbar$. As always, we fix $\ell$ prime to $q$.) For readers unfamiliar with this definition, it can be taken as a black-box; \'etale cohomology should be thought of as an analogue of singular cohomology which can be defined purely algebraically. 

The key consequence is the Grothendieck--Lefschetz fixed point theorem, which relates the fixed points of a morphism $f\colon X\to X$ with its action on the \'etale cohomology $H^i_{\et}(X;\Q_\ell)$, exactly in accordance with the usual Lefschetz fixed point formula. Applying this to the Frobenius morphism $\Frob_q$ gives the following fundamental formula, which holds for any smooth projective variety $X$ over $\F_q$:
\begin{equation}
\label{eq:GL1}
\left|X(\F_q)\right|=\#\Fix(\Frob_q) = \sum_{i\geq 0}(-1)^i\tr\big(\Frob_q:H^i_{\et}(X;\Q_\ell)\big)
\end{equation}

%Let $X$ be any algebraic variety defined over $\Z$.   Let $\Fqbar$ denote the algebraic closure of $\F_q$, and let $\Frob_q\colon \Fqbar \to\Fqbar$ denote the Frobenius.  The map $\Frob_q$ induces an endomorphism 
%\[\Frob_q\colon X(\Fqbar)\to X(\Fqbar).\]
%By Fermat's Little Theorem, $\Fix(\Frob_q)=X(\F_q)$, and more generally
%\[\Fix(\Frob_q^r)=X(\F_{q^r})\ \ \ \text{for each\ }r\geq 1\]

%Under suitable conditions on $X$, the {\em Grothendieck--Lefschetz formula} tells us that we can 
%apply the Lefschetz Fixed Point Theorem to $\Frob_q\colon X(\Fqbar)\to X(\Fqbar)$, replacing usual cohomology by \'{e}tale cohomology $H^i_{\et}(X;\Q_\ell)$ and with each fixed point having index $1$; in other words
%\begin{equation}
%\label{eq:GL1}
%\#\Fix(\Frob_q) = \sum_{i\geq 0}(-1)^i
%\tr[\Frob_q:H^i_{\et}(X;\Q_\ell)\to H^i_{\et}(X;\Q_\ell)]
%\end{equation}

Of course \eqref{eq:GL1} is only as good as our ability to compute $H^i_{\et}(X;\Q_\ell)$ and the 
trace of Frobenius on it. For general varieties $X$ this can be very difficult: indeed the last part of the Weil Conjectures to be proved was a bound on the eigenvalues of $\Frob_q$ when $X$ is smooth. Even after the formula \eqref{eq:GL1} was established, it was almost a decade before this bound was proved by Deligne, completing the proof of the Weil conjectures. Fortunately, the varieties considered in this paper are extremely special, and in particular $\Frob_q$ will always act on $H^i_{\et}(X;\Q_\ell)$ by a specific power of $q$. This reduces the computation of $\tr(\Frob_q:H^i_{\et}(X;\Q_\ell))$ to determining the dimension of $H^i_{\et}(X;\Q_\ell)$. We will do this by comparing it with the Betti numbers of a manifold, where we can apply tools of topology.

\para{Comparison between $\Fqbar$ and $\C$} If $X$ is defined over $\Z$ or $\Z_p$, we can reduce $X$ modulo $p$ to obtain a variety over $\F_p$; this puts us in the above situation, so we can study $X(\F_p)$ via the \'etale cohomology $H^i_{\et}(X;\Q_\ell)$. On the other hand, by extending the scalars from $\Z$ to $\C$, we can look at the complex points $X(\C)$, which is a compact complex manifold (possibly with singularities, if $X$ is not smooth). 
This leads (in a nontrivial way) to a \emph{comparison map}
\begin{equation}
\label{eq:comparisonmap}
c_X\colon H^i_{\et}(X_{/\overline{\F}_p};\Q_\ell)\to H^i(X(\C);\Q_\ell).
\end{equation}
%When $X$ is smooth and projective, \eqref{eq:comparisonmap} lets us compare the \'etale cohomology of $X$ over $\overline{\F}_p$ with the usual singular cohomology of the compact complex manifold $X(\C)$.
Artin's comparison theorem states that under favorable conditions, the comparison map $c_X$ is an isomorphism.

%For now all we will need is the following consequence:
%\bigskip
%
%\noindent
%{\bf Comparison Theorem (Artin): }$ \dim_{\Q_\ell}H^i_{\et}(X;\Q_\ell)=\dim_\C H^i(X(\C);\C)$.
%\bigskip
%
%The most difficult part of applying  the Grothendieck--Lefschetz fixed point theorem is the determination of the trace of Frobenius.    This was accomplished by Deligne for many $X$ in his proof of the Weil Conjectures.    In the cases of interest here, we will be able to quote the pertinent theorems that will give this value.  With these theorems in hand, the Grothendieck--Lefschetz formula \eqref{eq:GL1} can be used to compute.

\para{Non-compact varieties}
For a non-compact space $Y$, we know that the Lefschetz formula does not hold as stated in \eqref{eq:compactLefschetz}; we must look instead at the compactly-supported cohomology $H^i_c(Y;\Q)$. Similarly, if $X$ is not projective we should replace $H^i_{\et}(X;\Q_\ell)$ in \eqref{eq:GL1} by the ``compactly-supported \'etale cohomology''. But when $X$ is smooth we can skirt this issue by using Poincar\'{e} duality, obtaining the following formula valid for any smooth $X$:
\begin{equation}
\label{eq:GLcorrect}
\left|X(\F_q)\right|=\#\Fix(\Frob_q) = q^{\dim X}\sum_{i\geq 0}(-1)^i\tr\big(\Frob_q:H^i_{\et}(X;\Q_\ell)^\vee\big)
\end{equation}

\para{Example (Squarefree polynomials over \boldmath$\F_q$)} The number of monic squarefree degree-$n$ polynomials over $\F_q$ is well-known to be $q^n-q^{n-1}$. As a first warmup, we will describe how this computation can be derived from the Grothendieck--Lefschetz formula \eqref{eq:GLcorrect}.

In the introduction, we introduced the space $\Conf_n$ of monic squarefree degree-$n$ polynomials. Let us be more precise about what this means. We can identify the space $D_n$ of monic degree-$n$ polynomials with $\A^n$ by $T^n+a_1T^{n-1}+\cdots+a_n\leftrightarrow (a_1,\ldots,a_n)$. The condition for a polynomial $f(T)$ to be  squarefree is the non-vanishing of its \emph{discriminant} $\Delta(f)$. Since the discriminant $\Delta(f)$ is given by an integral polynomial in the coefficients $a_1,\ldots,a_n$, the complement $\Conf_n\coloneq \A^n-\{\Delta=0\}$ is a quasiprojective variety, and in fact a smooth scheme over $\Z$.

To count the number ${\#} \Conf_n(\F_q)$ of monic squarefree degree $n$ polynomials over $\F_q$, we will use the Grothendieck--Lefschetz formula \eqref{eq:GLcorrect}. We will show below at the end of the proof of Theorem~\ref{th:chi} that $\Frob_q$ acts on $H^i_{\et}(\Conf_n;\Q_\ell)$ by multiplication by $q^{i}$ (see also the discussion following Proposition~\ref{pr:kim}). Therefore it remains to compute the dimensions of $H^i_{\et}(\Conf_n;\Q_\ell)$.

Over an algebraically closed field such as $\C$, we can identify  $\Conf_n(\C)$ with the ``configuration space'' parametrizing $n$-element subsets of $\C$ (whence the name). This identification sends the squarefree polynomial $f(T)\in \Conf_n(\C)$ to its set of roots.
Using this description of $\Conf_n(\C)$, Arnol'd \cite{Arnold} proved that for $n\geq 2$: 
\[
H^i(\Conf_n(\C);\C)= 
\begin{cases}
%\left\{
%\begin{array}{ll}
\C & i=0\\
\C & i=1\\
0& i\geq 2
\end{cases}
%\end{array}\right.
\]
By Artin's Comparison Theorem, this implies that $H^0_{\et}(\Conf_n;\Q_\ell)=H^1_{\et}(\Conf_n;\Q_\ell)=\Q_\ell$, while $H^i_{\et}(\Conf_n;\Q_\ell)=0$ for $i\geq 2$. Since $\Frob_q$ acts on $H^i_{\et}(\Conf_n;\Q_\ell)$ by $q^i$, and thus on $H^i_{\et}(\Conf_n;\Q_\ell)^\vee$ by $q^{-i}$, we conclude that 
\[
\tr\big(\Frob_q: H^i_{\et}(\Conf_n;\Q_\ell)^\vee\big)
=
\begin{cases}
1 & i=0\\
q^{-1}&i=1\\
0& i\geq 2
\end{cases}
\]
$\Conf_n$ is $n$-dimensional, being an open subvariety of $\A^n$,  so the Grothendieck--Lefschetz formula \eqref{eq:GLcorrect} gives
\begin{align*}
\# \Conf_n(\F_q)&= \# \Fix[\Frob_q\colon \Conf_n(\Fqbar)\to\Conf_n(\Fqbar)]\\
&= q^n\sum_{i\geq 0}(-1)^i
\tr(\Frob_q:H^i_{\et}(\Conf_n;\Q_\ell)^\vee)\\
&= q^n\big((\tr(\Frob_q:H^0_{\et}(\Conf_n;\Q_\ell)^\vee) - \tr(\Frob_q:H^1_{\et}(\Conf_n;\Q_\ell)^\vee)\big)\\
&= q^n(1-q^{-1})=q^n-q^{n-1}
\end{align*}
This agrees with the well-known value of ${\#} \Conf_n(\F_q)$.

\subsection{Twisted coefficients} 
\label{subsec:twisted}
Much more subtle counts of $\F_q$-points can be obtained by using a version of Grothendieck--Lefschetz with twisted coefficients. For any appropriate system of coefficents  $\FF$ on a smooth projective variety $X$ defined over $\F_q$ (namely a so-called $\ell$-adic sheaf), we have a version of \eqref{eq:GL1} with coefficients in $\FF$:
\begin{equation}
\label{eq:GLtwisted}
\sum_{x\in X(\F_q)}\tr\big(\Frob_q | \FF_x)
=\sum_i(-1)^i\tr\big(\Frob_q: H^i_{\et}(X;\FF)\big)
\end{equation}
If $\Frob_q$ fixes a point $x\in X(\F_q)$, it acts on the stalk $\FF_x$ of $\FF$ at $x$, and the local contributions on the left side of \eqref{eq:GLtwisted} are the trace of $\Frob_q$ on each of these stalks. 
On the right side we have the \'etale cohomology of $X_{/\F_q}$ with coefficients in $\FF$, which may  again be taken as a black box. For non-projective $X$ we may again correct the formula \eqref{eq:GLtwisted}, either by considering the compactly-supported version of $H^i_{\et}(X;\FF)$, or via Poincar\'{e} duality.

In the remainder of this section we give an example of how the formula \eqref{eq:GLtwisted} may be applied to
$\Conf_n$. First, we describe the space $\PConf_n$, and how $\PConf_n$ arises as a covering space of $\Conf_n$.

\para{\boldmath$\PConf_n$ as a cover of \boldmath$\Conf_n$}
Recall that $D_n$ is the space of monic squarefree degree-$n$ polynomials, and consider the map 
\[\pi\colon \A^n\to D_n\]
defined by 
\[\pi\colon (x_1,\ldots,x_n)\mapsto f(T)= (T-x_1)\cdots(T-x_n).\]
Since $\pi$ is invariant under permutation of the coordinates $x_i$,  it factors through the quotient $\A^n/S_n$.  In fact $\pi$ induces an isomorphism $\A^n/S_n\to D_n$, as follows.  The $S_n$-invariant functions on $\A^n$ form the ring of symmetric polynomials $\Z[x_1,\ldots,x_n]^{S_n}$.  As a function of the roots $x_j$, the coefficient $a_i$ is $\pm$ the $i$th elementary symmetric function $e_i(x_1,\ldots,x_n)$. 
The fundamental theorem of symmetric polynomials states that \[\Z[x_1,\ldots,x_n]^{S_n}=\Z[e_1,\ldots,e_n]=\Z[a_1,\ldots,a_n],\] giving the isomorphism $\A^n/S_n\to D_n$.

Under this map $\pi$, what is the preimage of $\Conf_n=D_n-\{\Delta=0\}$? The discriminant $\Delta(f)$ vanishes exactly when $f$ has a repeated root, so the preimage of $\{\Delta=0\}$ will be the set of all $(x_1,\ldots,x_n)\in \A^n$ for which two coordinates $x_i$ and $x_j$ coincide. In other words, if we define $\PConf_n\coloneq \A^n-\{x_i=x_j\}$, then $\PConf_n$ is the preimage of $\Conf_n$ under $\pi$. Since $\PConf_n$ is defined in $\A^n$ by the nonvanishing of integral polynomials, $\PConf_n$ is a smooth $n$-dimensional scheme over $\Z$.

Since $S_n$ acts freely on $\PConf_n$ (by definition!), restricting $\pi$ to a map $\PConf_n\to \Conf_n$ gives an \'etale Galois cover with Galois group $S_n$ (acting on $\PConf_n$ by permuting the coordinates). On the topological side, $\PConf_n(\C)$ is the cover of $\Conf_n(\C)$ 
corresponding to the kernel of the representation $\pi_1(\Conf_n(\C))\to S_n$ sending a loop of configurations to the permutation it induces on the $n$ points.

\begin{numberedremarknotitalicized}
\label{rem:quotient}
The discussion above shows that $\Conf_n$ is the quotient $\Conf_n\approx \PConf_n/S_n$ as an algebraic variety. However for readers not familiar with scheme-theoretic quotients, we emphasize that the $k$-points $\Conf_n(k)$ are \emph{not} just the quotient $\PConf_n(k)/S_n$ of the $k$-points of $\PConf_n$.

For an algebraically closed field $\kbar$, there is no discrepancy: $\PConf_n(\kbar)$ parametrizes ordered $n$-tuples $(x_1,\ldots,x_n)$ of elements $x_i\in \kbar$ and $\Conf_n(\kbar)$ parametrizes unordered $n$-element sets $\{x_1,\ldots,x_n\}\subset \kbar$, exactly as one would expect.

The difference arises for fields $k$ that are not algebraically closed. For example, the polynomial $T^2+1\in \R[T]$ is squarefree, so it defines a point in $\Conf_2(\R)$. However, this polynomial is not in the image of any $(x_1,x_2)\in \PConf_2(\R)$, since we cannot write $T^2+1=(T-x_1)(T-x_2)$ for any $x_1,x_2\in \R$. For a more drastic example, the $\F_p$-points $\Conf_n(\F_p)$ parametrize monic squarefree degree $n$ polynomials in $\F_p[T]$, and this set is always nonempty. (For example, either $T^n-T\in \F_p[T]$ or $T^n-1\in \F_p[T]$ is always squarefree, depending on whether $p|n$ or not.) But when $n>p$ the set $\PConf_n(\F_p)$ is empty, since choosing more than $p$ distinct elements of $\F_p$ is impossible!

We can understand $\Conf_n(k)$ by relating it to $\Conf_n(\overline{k})$, as the subset of points that are defined over $k$. The catch is to describe what it means for a set $\{x_1,\ldots,x_n\}\subset \overline{k}$ to be ``defined over $k$'': it means not that \emph{each} $x_i$ lies in $k$ (i.e.\ is invariant under $\Gal(\overline{k}/k)$), but that the \emph{set} $\{x_1,\ldots,x_n\}$ is invariant under $\Gal(\overline{k}/k)$. For example, the set $\{i,-i\}$ is invariant under $\Gal(\C/\R)$ (i.e.\ by complex conjugation), so the corresponding point in $\Conf_2(\C)$ should give a point in $\Conf_2(\R)$\,---\,and indeed it does, namely the polynomial $T^2+1\in \R[T]$.

We can see this concretely for a squarefree polynomial $f(T)\in \Conf_n(\F_q)$. Since $f(T)$ is squarefree, it must have $n$ distinct roots $\lambda_1,\ldots,\lambda_n$ in $\Fqbar$. But since $f(T)$ has coefficients in $\F_q$, it is invariant under $\Frob_q$ (which topologically generates $\Gal(\Fqbar/\F_q)$), so the \emph{set} of roots $\{\lambda_1,\ldots,\lambda_n\}$ must be invariant under $\Frob_q$. Conversely, given any set $\{x_1,\ldots,x_n\}\subset \Fqbar$ which is taken to itself by $\Frob_q$, the polynomial $f(T)=(T-x_1)\cdots(T-x_n)$ is fixed by $\Frob_q$, and thus has coefficients in $\F_q$.

\end{numberedremarknotitalicized}

\para{Twisted statistics for $\Conf_n$} Using the surjection $\pi_1(\Conf_n(\C))\to S_n$, any finite-dimensional complex representation $V$ of $S_n$ lets us build a vector bundle (with flat connection) over $\Conf_n(C)$, so that its \emph{monodromy representation} is the composition $\pi_1(\Conf_n(\C))\to S_n\to \GL(V)$. Since the cover $\PConf_n(\C)\to \Conf_n(\C)$ corresponds to the kernel of this surjection, any bundle we build in this way will become trivial if we pull it back to $\PConf_n(\C)$.

The same constructions can be done in the algebraic setting: the Galois $S_n$-cover $\PConf_n\to \Conf_n$ gives a natural corrrespondence between finite-dimensional representations of $S_n$ and finite-dimensional local systems (locally constant sheaves) on $\Conf_n$ that become trivial when restricted to $\PConf_n$. Given a representation $V$ of $S_n$, let $\chi_V$ be its character, and let $\V$  denote the corresponding local system on $\Conf_n$. (Since every irreducible representation of $S_n$ can be defined over $\Z$, the local system $\V$ determines an
$\ell$--adic sheaf. We will not stress this point further, although it is important.)

%Now let $V$ be any finite-dimensional complex representation of $S_n$.  Such a representation gives a {\em local system}, that is a locally constant sheaf, on the scheme $\Conf_n$; by abuse of notation we denote this local system by $V$.   
%Attached to such a system is its {\em monodromy representation} 
%\[\rho:\pi_1(X(\C))\to S_n \to \GL(V).\] 
%
%Since every irreducible representation of $S_n$ is defined over $\Z$ (as the image of the Young
%symmetrizer $c_\lambda\in \Z S_n$), it can be viewed as an
%$\ell$--adic sheaf. We obtain the \'etale cohomology over $\Q_\ell$,
%which by a further abuse of notation we denote simply by
%$H^i(\Conf_n;V(\lambda))$.

If $f=f(T)\in \Conf_n(\F_q)$ is a fixed point for the action of $\Frob_q$ on $\Conf_n(\Fqbar)$, then $\Frob_q$ acts on the stalk $\V_f$ over $f$.  This action can be described as follows.  The roots of $f(T)$ are permuted by the action of Frobenius on $\Fqbar$, which determines a permutation $\sigma_f\in S_n$ (defined up to
conjugacy). The stalk $\V_f$ is isomorphic to $V$, and
in some basis for $V$, the automorphism $\Frob_q$ acts according to the action of $\sigma_f$. (This is explained in more detail in the latter part of the next section.) It follows that:
\[\tr(\Frob_q:\V_f)=\chi_V(\sigma_f)\]

The  Grothendieck--Lefschetz formula \eqref{eq:GLtwisted} becomes in this case (via Poincar\'e duality) the equality
\begin{equation}
\label{eq:GL2}
\sum_{f\in \Conf_n(\F_q)}\tr\big(\Frob_q | \V_f)
=q^n\sum_i(-1)^i\tr\big(\Frob_q: H^i_{\et}(\Conf_n;\V)^\vee)
\end{equation}
As before, we will see that $\Frob_q$ acts on $H_{\et}^i(\Conf_n;\V)$
by $q^{i}$. It thus suffices to know the
dimension of the cohomology group
$H^i_{\et}({\Conf_n};\V)$.  We will prove below that this dimension can be computed as
\[\dim_{\Q_\ell}H^i_{\et}({\Conf_n};\V)= \big\langle V,H^i(\PConf_n(\C);\C)\big\rangle_{S_n}\]
Here $\langle V,W\rangle_{S_n}$ is the usual inner product of $S_n$-representations $V$ and $W$: \[\langle V,W\rangle=\dim_\C\Hom_{S_n}(V,W)\] % for example, when $V$ is irreducible this is just the multiplicity of $V$ in $W$.  \

Combining all these
observations, the Grothendieck--Lefschetz formula becomes the
following fundamental formula:
\begin{equation}
\label{eq:GL3}
\sum_{f\in \Conf_n(\F_q)}\chi_V(\sigma_f)=\sum_i(-1)^i
q^{n-i}\big\langle V,H^i(\PConf_n(\C);\C)\big\rangle_{S_n}
\end{equation}
We will prove a generalization of this formula as Theorem~\ref{th:chi}.

Note that when $V$ is the trivial representation we have $\chi_V(\sigma_f)=1$ for all $f$, and so \eqref{eq:GL3} reduces to the previous untwisted Grothendieck--Lefschetz formula \eqref{eq:GL1}.  Formula \eqref{eq:GL3} converts various counting problems about polynomials over $\F_q$ to the problem of understanding the decomposition of  $H^i(\PConf_n(\C);\C)$ as $S_n$-representations.   

\para{Example (counting linear factors)} 
Let $W=\C^n$ be the standard permutation representation of $S_n$.  Then $\chi_W(\sigma)$ is the number of fixed points of $\sigma$.   Over each fixed point $f=f(T)\in \Conf_n(\F_q)$, the roots fixed by the permutation $\sigma_f\in S_n$ (i.e.\ fixed by $\Frob_q$) are those lying in $\F_q$.  The set of such roots corresponds precisely to the set of linear factors of $P$.  We thus have
\[\chi_W(\sigma_f)= \ \text{the number of linear factors of $f(T)$}.\]
In Proposition~\ref{pr:standard} we will prove that for each $i\geq 1$ we have
  \[\big\langle W,H^i(\PConf_n(\C);\C)\big\rangle_{S_n} =\begin{cases}0&\text{ for }n\leq i\\1&\text{ for }n=i+1\\2&\text{ for }n\geq i+2\end{cases}\]
%

%In \S\ref{sec:standard} we will prove for each $n\geq 3$ that 
%\[\langle V(0),H^i(\Conf_n(\C);\C)\rangle=
%\left\{
%\begin{array}{ll}
%1 & i=0,1\\
%0 & i>1
%\end{array}
%\right.
%\]
%
%and 
%
%\[\langle V(1),H^i(\Conf_n(\C);\C)\rangle =
%\left\{
%\begin{array}{ll}
%1 & i=1,n-1\\
%2& 2\leq i\leq n-2\\
%0 & i\geq n
%\end{array}
%\right.
%\]

Applying \eqref{eq:GL3} thus gives that the total number of linear factors of all monic squarefree degree-$n$ polynomials over $\F_q$ equals
\begin{align*}
\sum_{f\in \Conf_n(\F_q)}\chi_W(\sigma_f)&=\sum_i(-1)^i
q^{n-i}\big\langle W,H^i(\PConf_n(\C);\C)\big\rangle_{S_n}\\
&=q^n - 2q^{n-1}+2q^{n-2}-2q^{n-3}+\cdots\mp 2q^3\pm 2q^2\mp q
\end{align*}

To obtain the \emph{expected} number of linear factors, we simply divide by the cardinality of $\Conf_n(\F_q)$. Since this was determined above to be $q^n-q^{n-1}$, we conclude that the expected number of linear factors is 
\[1-\frac{1}{q}+\frac{1}{q^2}-\frac{1}{q^3}+\cdots \pm \frac{1}{q^{n-2}}\]
This can be widely generalized: in fact for \emph{any} statistic $s(f)$ on polynomials $f(T)$ only depending on the lengths of the irreducible factors of $f$, we can find a representation $V$ (or a difference of two representations) that allows us to calculate $\sum_{f\in \Conf_n(\F_q)}s(f)$ via topology; see \S\ref{s:purebraid}.

\section{Hyperplane arrangements, their cohomology, and combinatorics of squarefree polynomials}
\label{s:polynomials}

In \cite{CEF} the theory of various ``FI-objects'' (e.g. FI-spaces, FI-varieties, FI-modules) was developed in order to better understand infinite sequences of such objects.   In this chapter we introduce the notion of  ``$\FI$-complement of hyperplane arrangement'', or ``FI-CHA'' for short.  These are, roughly,  complements of  those 
hyperplane arrangements that can be generated in a uniform way by a finite set of ``generating hyperplanes''.  Applying the cohomology or \'{e}tale cohomology functor will then give an FI-module (in the sense of \cite{CEF}, and defined below).  We can then deduce strong constraints on these \'{e}tale cohomology groups from the results of  \cite{CEF}.

We then explain the direct connection of these cohomology groups to moduli spaces of monic, squarefree polynomials, proving a general theorem that converts the stability inherent in FI-modules to 
the stability of  combinatorial statistics for squarefree polynomials over $\F_q$.   In Section~\ref{s:purebraid} we will apply this general theorem to obtain precise answers to a variety of counting problems in $\F_q[T]$.

\subsection{FI-hyperplane arrangements}

We briefly recall the basic definitions and notation for FI-modules from \cite{CEF}.  We denote by $\FI$ the category of finite sets with inclusions, and by $\FI^{op}$ its opposite category.  We denote the set $\set{1,\ldots,n}$ by $[n]$.  A functor from $\FI$ (respectively $\FI^{op}$) to the category of modules over a ring $A$ is called an {\em $\FI$-module} (respectively $\FI^{op}$-module) over $A$.  If $V$ is an $\FI$-module, we denote by $V_n$ the $A$-module $V([n])$. Since $\End_{\FI}([n])\simeq S_n$, each $A$-module $V_n$ has a natural action of $S_n$.

Let $R$ be a ring and let $L = \{L_1, \ldots, L_m\}$ be a finite set of nontrivial linear forms over $R$ in variables $x_1, \ldots, x_d$ containing the form $x_1-x_2$.    For each $n$, each $L_i$, and each injection $f\colon [d] \hookrightarrow [n]$ we have a linear form $L_i^f$ in $x_1, \ldots, x_n$ defined by
\beq
L_i^f(x_1, \ldots, x_n) = L_i(x_{f(1)}, \ldots, x_{f(d)}).
\eeq
(We could relax the condition that the forms $L_i$ involve the same number of variables, in which case $f$ would range over inclusions $f\colon [d_i]\hookrightarrow [n]$, but for readability we will stick to the simpler situation.) Each such form determines a hyperplane $H_i^f$ in affine $n$-space $\A^n_R$ defined by the linear equation $L_i^f=0$.
We define the \emph{hyperplane complement} $\AA(L)_n$ as the complement of this hyperplane arrangement: \[
\AA(L)_n\coloneq \A^n_R-\bigcup_{f,i} H_i^f\] Not all hyperplane arrangments can be built in this way from a ``generating set'' $L$, but many of the most familiar hyperplane arrangements can.  Some simple examples:
\begin{itemize}
\item  $L=\{x_1-x_2\}$. In this case we have only one form $L_1=x_1-x_2$, so an inclusion $f\colon [2]\hookrightarrow [n]$ determines the form $L_1^f=x_{f(1)}-x_{f(2)}$. As $f$ ranges over all inclusions, we obtain  the arrangment of hyperplanes $x_i - x_j=0$ for $i\neq j\in \{1,\ldots,n\}$, usually known as the \emph{braid arrangement}. The hyperplane complement $\AA(L)_n$  is the space
\[\AA(L)_n=\big\{(x_1,\ldots,x_n)\in \A^n_R\,\big|\, x_i\neq x_j \text{ for }i\neq j\big\}\] of $n$-tuples of {\em distinct} points in $\A^1$, which was denoted by $\PConf_n$ in \S\ref{subsec:twisted}.
\item  $L = \set{x_1-x_2, x_1+x_2, x_1}$. In this case we obtain the \emph{braid arrangement of type $B_n$}, consisting of the hyperplanes $x_i-x_j=0$ and $x_i+ x_j=0$ for $i\neq j\in \{1,\ldots,n\}$ together with the coordinate hyperplanes $x_i=0$ for $i\in \{1,\ldots,n\}$.  The hyperplane complement \[\AA(L)_n=\big\{(x_1,\ldots,x_n)\in \A^n_R\,\big|\, x_i\neq 0, x_i\neq x_j, x_i\neq -x_j\text{ for }i\neq j\big\}\] parametrizes $n$-tuples of  nonzero points in $\A^1$ disjoint from each other and from their negatives.
\item $L = \set{x_1 - x_2, x_1- 2x_2 + x_3}$. In this case the hyperplane complement \[\AA(L)_n=\big\{(x_1,\ldots,x_n)\in \A^n_R\,\big|\, x_i\neq x_j, x_i+x_j\neq 2x_k \text{ for }i\neq j\neq k\big\}\] parametrizes $n$-tuples of  distinct numbers in $R$ no three of which form an arithmetic progression.
\end{itemize}

For any $L$, the action of $S_n$ on $\A^n_R$ by permuting the coordinates permutes the set of hyperplanes $H_i^f$, and thus induces an action of $S_n$ on $\AA(L)_n$. But there are also maps between the different $\AA(L)_n$ for different $n$, and in fact the ensemble of all these schemes and maps forms an $\FI^{op}$-scheme.

\begin{prop}  There exists an functor $\AA(L)$ from $\FI^{op}$ to the category of schemes over $R$ which sends $[n]$ to $\AA(L)_n$ and sends $g\colon [m] \hookrightarrow [n]$ to the surjection $(x_1, \ldots, x_n)\mapsto (x_{g(1)}, \ldots, x_{g(m)})$. 
\end{prop}

We call $\AA(L)$ the FI-CHA {\em determined by $L$}; here ``FI-CHA'' is an abbreviation for ``$\FI$-complement of hyperplane arrangement''.

\begin{proof}  An injection $g\colon [m] \hookrightarrow [n]$ determines a surjection $g^*\colon \A^n\to \A^m$ which sends $(x_1, \ldots, x_n)$ to $(x_{g(1)}, \ldots, x_{g(m)})$, and this assignment is obviously functorial. Moreover the form $L_i^f\colon \A^n\to \A^1$ from the definition of $\AA(L)_n$ is just the composition of $f^*\colon \A^n\to \A^d$ with the original linear form $L_i\colon \A^d\to \A^1$. It follows that $g^*(L_i^f)=L_i^{g\circ f}$, so the preimage of $H_i^f$ under $g^*$ is the hyperplane $H_i^{g\circ f}$, demonstrating that $g^*$ restricts to a map $g^*\colon \AA(L)_n\to\AA(L)_d$.
\end{proof}

Composing the functor $\AA(L)$ with the contravariant functor  ``\'{e}tale cohomology''  thus gives the following.

\begin{cor} Let $\AA(L)$ be an FI-CHA.  The \'{e}tale cohomology groups $H^i_{\et}(\AA(L)_n;\Q_\ell)$ fit together into an $\FI$-module over $\Q_\ell$.  We denote this $\FI$-module by $H^i_{\et}(\AA(L);\Q_\ell)$.
\end{cor}

%From now on, we restrict to the case $R=\bar{\F}_q$, the algebraic closure of a finite field.  The investigation of the \'{e}tale cohomology of hyperplane complements over $\Fqbar$ was initiated by Lehrer in ~\cite{lehrer:ladic}.  The main theme is that everything agrees beautifully with the classical description of the singular cohomology of complex hyperplane arrangement. 

%\para{\'Etale cohomology of hyperplane arrangements}
\subsection{\'Etale cohomology of hyperplane arrangements}

The following facts on the \'etale cohomology of the hyperplane complements $\AA(L)_n$,
 whose analogues for complex hyperplane arrangements are well-known, were proved over a general base field by Kim~\cite{kim:arrangements} (see also Lehrer~\cite{lehrer}). If $X$ is a variety over a field $k$, by $H^i_{\et}(X;\Q_\ell)$ we always mean the cohomology $H^i_{\et}(X_{/\overline{k}};\Q_\ell)$ of the base change $X_{\overline{k}}$.
  
Let $k$ be a field, and let $L\colon \A^n\to \A^1$ be a nontrivial $k$-linear form. If $H\subset \A^n$ is the hyperplane $L=0$, this form restricts to a map $L\colon \A^n-H\to \A^1-\{0\}$. The fibers of this map are $\A^{n-1}$, so on cohomology $L$ induces a isomorphism
\beq
L^*\colon H^1_{\et}(\A^1 - \{0\}; \Q_\ell)\stackrel{\approx}{\to}H^1_{\et}(\A^n -H).
\eeq

\begin{prop}  Let $k$ be a field, and fix a prime $\ell$ different from the characteristic of $k$. Given a finite set of hyperplanes $H_1,\ldots,H_m$ in $\A^n$ defined over $k$, let $\AA$ be the complement $\AA\coloneq \A^n-\bigcup H_j$.  Then:
\begin{enumerate}[(i)]
\item $H^1_{\et}(\AA;\Q_\ell)$ is spanned by the images of the $m$ maps
\beq
 H^1_{\et}(\A^n -H_j)\ra H^1_{\et}(\AA;\Q_\ell)
\eeq
induced by the inclusion of $\AA$ into $\A^n-H_j$ for $j=1,\ldots,m$.
\item $H^i_{\et}(\AA;\Q_\ell)$ is generated by $H^1(\AA;\Q_\ell)$ under cup product.
\end{enumerate}
\label{pr:kim}
\end{prop}

Let $k = \Fqbar$ and suppose that $L$ is defined over $\F_q$.  Let  $\Frob_q\colon \AA\to \AA$ be the Frobenius morphism. (In this paper we always mean the \emph{geometric} Frobenius morphism, which acts on the coordinates of an affine variety over $\Fqbar$ by raising the coordinates to the $q$th power.)  A consequence of Proposition~\ref{pr:kim} that will be crucial for us is:

\medskip
\emph{The induced action of $\Frob_q$ on $H^i_{\et}(\AA;\Q_\ell)$ is scalar multiplication by $q^i$.}
\medskip

For those familiar with the terms, this follows from Proposition~\ref{pr:kim} as follows.    
Let $\Q_\ell(1)$ denote the $1$-dimensional Galois representation given by the cyclotomic character, let $\Q_\ell(n)\coloneq \Q_\ell(1)^{\otimes n}$, and let $\Q_\ell(-1)$ denote the dual of $\Q_\ell(1)$. It is well-known that 
$H^1_{\et}(\A^1 - \{0\}; \Q_\ell)\approx \Q_\ell(-1)$, so Proposition~\ref{pr:kim}(i) implies that $H^1_{\et}(\AA;\Q_\ell)\approx \Q_\ell(-1)^{\oplus b_1}$ for some $b_1$.  By Proposition~\ref{pr:kim}(ii) it follows that $H^i_{\et}(\AA;\Q_\ell)\approx \Q_\ell(-i)^{\oplus b_i}$ for some $b_i$. The geometric Frobenius morphism $\Frob_q$ is known to act by $q^{-1}$ on $\Q_\ell(1)$, so it acts by $q^{i}$ on $\Q_\ell(-i)$, as claimed.

\para{Finitely-generated FI-modules} Suppose that $V$ is an FI-module with each $V_n$ finite dimensional.   Then 
$V$ is {\em finitely generated} \cite[Definition~2.16]{CEF} if there are finitely many elements $v_1,\ldots,v_m\in V_d$ such that each $V_n$ is spanned by the images $f_*v_j$ induced by all inclusions $f\colon [d]\hookrightarrow [n]$.

\begin{prop} Let $\AA(L)$ be an FI-CHA over a field of characteristic $\neq\ell$.  Then $H^i_{\et}(\AA(L);\Q_\ell)$ is a finitely generated FI-module for each $i\geq 0$.
\label{pr:hifg}
\end{prop}

\begin{proof}
For each form $L_j\colon \A^d\to \A^1$ let $\omega_j\in H^1_{\et}(\AA(L)_d;\Q_\ell)$ be the image of the map 
\[L_j^*\colon H^1_{\et}(\A^d -H_j)\ra H^1_{\et}(\AA(L)_d;\Q_\ell).\] 

We begin by showing that the FI-module $H^1_{\et}(\AA(L);\Q_\ell)$ is finitely generated by these classes $\omega_1,\ldots,\omega_m$.

For each $f\colon [d]\hookrightarrow [n]$, let $\omega_j^f\in H^1_{\et}(\AA(L)_n;\Q_\ell)$ be the image of the map 
\[{L_j^f}^*\colon H^1_{\et}(\A^n -H_j^f)\ra H^1_{\et}(\AA(L)_n;\Q_\ell).\]

Proposition~\ref{pr:kim} implies that $H^1_{\et}(\AA(L)_n;\Q_\ell)$ is spanned by the classes $\omega_j^f$ as $f$ ranges over all inclusions $f\colon [d]\hookrightarrow [n]$.  Since $L_j^f=f^*\circ L_j$, naturality implies that $\omega_j^f=f^*\omega_j$. This shows that $H^1_{\et}(\AA(L)_n;\Q_\ell)$ is spanned by the images $f^*\omega_j$ of the finite set $\omega_1,\ldots,\omega_m\in H^1_{\et}(\AA(L)_d;\Q_\ell)$, as desired.

Proposition~\ref{pr:kim}(ii) now implies that $H^i_{\et}(\AA(L);\Q_\ell)$ is a quotient of $\bwedge^i H^1_{\et}(\AA;\Q_\ell)$, which is finitely generated by \cite[Proposition 2.62]{CEF}.
\end{proof}

\begin{cor} Let $\AA(L)$ be an FI-CHA over a field of characteristic $\neq \ell$.  Then the sequence $H^i_{\et}(\AA(L)_n;\Q_\ell)$ of $S_n$-representations is uniformly representation stable in the sense of \cite{CF}.
\label{co:ficharepstab}
\end{cor}
\begin{proof}
It was proved in \cite[Proposition 2.58]{CEF} that if $V$ is a finitely generated FI-module over a field of characteristic 0, then the sequence $V_n$ of $S_n$-representations is uniformly stable.
\end{proof}

Applying the representation stability from Corollary~\ref{co:ficharepstab} to the trivial representation, we conclude that the dimension of the $S_n$-invariant subspace $H^i_{\et}(\AA(L)_n;\Q_\ell)^{S_n}$ becomes independent of $n$ for large enough $n$. Transfer implies that 
\[H^i_{\et}(\AA(L)_n/S_n;\Q_\ell)\approx H^i_{\et}(\AA(L)_n;\Q_\ell)^{S_n}\]

Therefore, to understand the consequences of Corollary~\ref{co:ficharepstab}, we first describe 
the quotient spaces $\AA(L)_n/S_n$.

%Our goal in this section is to explain how the theory of FI-modules can be used to study systematically the geometry of the spaces $\AA(L)_n$ and their quotients by the symmetric group action $\AA(L)_n / S_n$.
\begin{definition}
For any FI-CHA $\AA(L)$, the symmetric group $S_n$ acts on the scheme $\AA(L)_n\subset \A^n$. We denote by $\BB(L)_n$ the quotient scheme $\BB(L)_n\coloneq \AA(L)_n/S_n$. 
\end{definition}

For $L=\{x_1-x_2\}$, we saw in Section~\ref{subsec:twisted} that the quotient of the hyperplane complement $\PConf_n=\AA(x_1-x_2)_n$ by $S_n$ was the moduli space  $\Conf_n=\BB(x_1-x_2)_n$ of squarefree polynomials. For any FI-CHA $\AA(L)$, we have assumed that $L$ contains the form $x_1-x_2$; therefore all points $(x_1,\ldots,x_n)\in\AA(L)_n$ must have $x_i\neq x_j$, so $S_n$ acts freely on $\AA(L)_n$. Moreover we can restrict the covering map $\PConf_n\to \Conf_n$ to $\AA(L)_n\subset \PConf_n$. This identifies the quotient $\BB(L)_n$ with an open subspace of $\Conf_n$. 
%
%Let $D_n$ denote the space of monic degree-$n$ polynomials in $\Z[T]$.  Of course $D_n$ can be identified with $\A^n$ via the map sending the polynomial 
%$T^n+a_1T^{n-1}+\cdots+a_n$ to the $n$-tuple $(a_1,\ldots,a_n)$, but we keep the separate notation $D_n$ to avoid confusion. Consider the map 
%\[\psi:\A^n\to D_n\]
%defined by 
%\[(x_1,\ldots,x_n)\mapsto P(T)=(T-x_1)\cdots(T-x_n)\]
%
%Since $\psi$ is invariant under permutation of the coordinates $x_i$,  it factors through the quotient $\A^n/S_n$.  In fact $\psi$ induces an isomorphism $\A^n/S_n\to D_n$, as follows.  The $S_n$-invariant functions on $\A^n$ form the ring of symmetric polynomials $\Z[x_1,\ldots,x_n]^{S_n}$.  As a function of the roots $x_j$, the coefficient $a_i$ is $\pm$ the $i$th elementary symmetric function $e_i(x_1,\ldots,x_n)$. 
%The fundamental theorem of symmetric polynomials states that \[\Z[x_1,\ldots,x_n]^{S_n}=\Z[e_1,\ldots,e_n]=\Z[a_1,\ldots,a_n],\] giving the isomorphism $\A^n/S_n\to D_n$.
%
%Restricting $\psi$ to the subspace $\PConf_n=\AA(x_1-x_2)_n$ consisting of $n$-tuples of distinct points identifies $\BB(x_1-x_2)_n=\PConf_n/S_n$ with the space $\Conf_n$ of monic \emph{squarefree} polynomials. More generally, 
For any field $k$, the  points $\BB(L)_n(k)$ are in bijection with the set of monic squarefree  degree-$n$ polynomials $P(T)\in k[T]$ with the property that no subset of the roots of $P$  (taken in the algebraic closure $\overline{k}$) satisfies any of the linear relations in $L_i$.

For example, when $L=\{x_1-x_2,x_1+x_2,x\}$ the points of $\BB(L)_n(k)$ are precisely the monic, squarefree,  degree $n$ polynomials not divisible by $T$ or by $T^2-a$ for any $a\in \overline{k}$.  When $L = \set{x_1 - x_2, x_1- 2x_2 + x_3}$ the points of $\BB(L)_n(k)$ are the squarefree, 
degree-$n$ polynomials having no three roots in arithmetic progression.  

\subsection{Point-counting for FI-CHAs}
\label{sec:pointcountingandstabilization}
Let $L$ be a collection of linear forms defined over $\F_q$, so that the schemes $\AA(L)$ and $\BB(L)_n$ are defined over $\F_q$. The $\F_q$-points $\BB(L)_n(\F_q)$ form a finite set, consisting of the monic squarefree degree-$n$ polynomials in $\F_q[T]$ whose roots do not satisfy any of the relations in $L$.

\para{Counting squarefree polynomials} We can count the number $|\BB(L)_n(\F_q)|$ of such polynomials in terms of the \'etale cohomology of $\BB(L)_n$, via the Grothendieck--Lefschetz fixed point formula.
%If $\Frob_q\colon {\BB(L)_n}\to {\BB(L)_n}$ is the Frobenius map, this formula states that
%\begin{equation}
%\label{eq:GLone}
%|\BB(L)_n(\F_q)|=\sum (-1)^i \tr\big(\Frob_q: H^i_{c}(\BB(L)_n;\Q_\ell)\big),
%\end{equation}
%where $H^i_{c}$ denotes \'etale cohomology with compact support (as always, of the base change  ${\BB(L)_n}_{/\Fqbar}$).
%
%For general varieties the action of Frobenius on $H^i_{c}$ or $H^i_{\et}$ is very difficult to compute, but for $\BB(L)_n$ the situation is much simpler: $\Frob_q$ acts by a power of $q$. Combined with Poincar\'e duality, we will see below that the formula \eqref{eq:GLone} becomes
We will see below that in this case this becomes the formula 
\[|\BB(L)_n(\F_q)|=\sum (-1)^i q^{n-i} \dim H^i_{\et}(\BB(L)_n;\Q_\ell)\]
By transfer, the dimension of $H^i_{\et}(\BB(L)_n;\Q_\ell)$ is the dimension of the $S_n$-invariant subspace of $H^i_{\et}(\AA(L)_n;\Q_\ell)$, and Corollary~\ref{co:ficharepstab} states that this invariant subspace becomes independent of $n$ for large $n$. This seems to show that the right hand side of
\[\frac{|\BB(L)_n(\F_q)|}{q^n}=\sum (-1)^i q^{-i} \dim H^i_{\et}(\BB(L)_n;\Q_\ell)\]
converges to a fixed power series as $n\to \infty$. However, we need a bound on these dimensions (e.g.\ as in Definition~\ref{def:convergent}); otherwise the power series itself need not converge, for example!

\para{Counting other statistics}
If $\chi$ is a class function  $\chi\colon S_n\to \Q$, we can consider $\chi$ as a function on the finite set $\Conf_n(\F_q)$ as follows. Given $f(T)\in \Conf_n(\F_q)$, let $\sigma_f\in S_n$ be the permutation of the roots $R(f)=\{x\in \Fqbar|f(x)=0\}$ induced by $\Frob_q$. This depends on an ordering of the roots, so $\sigma_f$ is only well-defined up to conjugation; nevertheless, we can define \[\chi(f)\coloneq \chi(\sigma_f)\] since $\chi$ is conjugacy-invariant.
Note that each $k$-cycle in the cycle decomposition of $\sigma_f$ corresponds to a degree $k$  irreducible factor of $f(T)$.

If $V$ is any $S_n$-representation over a field of characteristic 0, we denote by $\langle \chi,V\rangle$ the standard inner product of $\chi$ with the character of $V$; we sometimes refer to $\langle \chi,V\rangle$ as the \emph{multiplicity} of $\chi$ in $V$ since this is true when both are irreducible, by Schur's lemma. We write $\langle \chi, H^i(\AA(L)_n)\rangle$ as an abbreviation for  the inner product $\langle \chi, H^i_{\et}(\AA(L)_n;\Q_\ell)\rangle$. We remark that this inner product lies in $\Q$, since every representation of $S_n$ is defined over $\Q$; furthermore, at least for the FI-CHA $\AA(x_1-x_2)$ considered in \S\ref{s:purebraid}, this inner product is independent of $\ell$.

Our main general tool for studying the statistics of various sets of squarefree polynomials over $\F_q$ is the following.  

\begin{theorem}[{\bf Point counts for hyperplane arrangements}]
\label{th:chi}
Let $L$ be a collection of linear forms defined over $\F_q$, and let $\chi$ be any class function on $S_n$. Then  for each $n\geq 1$, 
\begin{equation}
\label{eq:chi}
\sum_{f(T)\in \BB(L)_n(\F_q)}\chi(f)=\sum_i (-1)^i q^{n-i}\langle \chi,H^i(\AA(L)_n)\rangle
\end{equation}
\end{theorem}
\begin{proof}
Since both sides of \eqref{eq:chi} are linear in $\chi$, and since the irreducible characters give a basis for class functions on $S_n$, it is enough to prove that \eqref{eq:chi} holds when $\chi$ is an irreducible character of $S_n$ (i.e.\ the character of an irreducible representation). 

The Galois $S_n$-cover $\AA(L)_n\to \BB(L)_n$ yields a natural corrrespondence between the set of (conjugacy classes of) finite-dimensional representations of $S_n$ and the set of (isomorphism classes of) those finite-dimensional local systems on $\BB(L)_n$ that become trivial when restricted to $\AA(L)_n$. Given an irreducible character $\chi$, let $V$ denote the corresponding irreducible representation of $S_n$, and let $\V$  denote the corresponding local system on $\BB(L)_n$.

Applying the Grothendieck--Lefschetz formula to the local system $\V$ relates the action of $\Frob_q$ on the stalks $\V_f$ with its action on the \'etale cohomology with coefficients in $\V$, via the following formula:
\begin{equation}
\label{eq:firstV}
\sum_{f\in \BB(L)_n(\F_q)} \tr\big(\Frob_q : \V_f\big)= \sum_j (-1)^j \tr\big( \Frob_q : H^j_{c}(\BB(L)_n; \V)\big)
\end{equation}
The left side is easy to analyze: each stalk $\V_f$ is isomorphic to the representation $V$, and  $\Frob_q$ acts according to the action of $\sigma_f\in S_n$. Since $\chi$ is the character of $V$, we see that $\tr(\Frob_q:\V_f)=\tr(\sigma_f:V)=\chi(f)$, showing that the left sides of \eqref{eq:chi} and \eqref{eq:firstV} coincide. 

It remains to simplify the right side of \eqref{eq:firstV}. Let $\widetilde{\V}$ denote the pullback of $\V$ to $\AA(L)_n$.  Transfer gives an isomorphism $H^j_{c}(\BB(L)_n;\V)\approx H^j_{c}(\AA(L)_n;\widetilde{\V})^{S_n}$. But $\widetilde{\V}$ is trivial on $\AA(L)_n$, so we have 
\[H^j_{c}(\AA(L)_n;\widetilde{\V})\approx H^j_{c}(\AA(L)_n;\Q_\ell)\otimes V\]
 as $S_n$-representations. Combining these gives
\beq
H^j_{c}(\BB(L)_n;\V) \approx \big(H^j_{c}(\AA(L)_n; \Q_\ell) \tensor V\big)^{S_n} \approx H^j_{c}(\AA(L)_n;\Q_\ell) \tensor_{\Q[S_n]} V.
\eeq
In particular, 
\[\dim (H^j_{c}(\BB(L)_n; \V))=\dim (H^j_{c}(\AA(L)_n;\Q_\ell) \tensor_{\Q[S_n]} V).\]
%%%WORKING HERE 4.18pm
Since $V$ is self-dual as an $S_n$-representation, $H^j_{c}(\AA(L)_n;\Q_\ell) \tensor_{\Q[S_n]} V$ is isomorphic to $\Hom_{\Q[S_n]}\big(V,H^j_{c}(\AA(L)_n)\big)$, whose dimension is given by the inner product $\langle \chi,H^j_{c}(\AA(L)_n)\rangle$.

%Because the hyperplanes $x_i - x_j$ are among those excluded from $\AA(L)_n$, we know that $S_n$ acts without fixed points on $\AA(L)_n$, so $\BB(L)_n$ is smooth.  
Since $\AA(L)_n$ is smooth of dimension $n$, Poincar\'{e} duality provides an identification
\beq
H^{2n-i}_{c}(\AA(L)_n; \Q_\ell) \approx \Hom\big(H^{i}_{\et}(\AA(L)_n;\Q_\ell); \Q_\ell(-n)\big).
\eeq
Since the $S_n$-action on $\Q_{\ell}(-n)$ is trivial, this implies that $\langle \chi,H^{2n-i}_{c}(\AA(L)_n)\rangle=\langle \chi,H^i_{\et}(\AA(L)_n)\rangle$.
%It follows that
%\beq
%\dim H^{2n-i}_{\et}(\BB(L)_n;V_n) = \langle V_n, H^{i}_{\et}(\AA(L)_n;\Q_\ell) \rangle = \langle P, H^{i}_{\et}(\AA(L)_n;\Q_\ell) \rangle
%\eeq
Moreover, by Proposition~\ref{pr:kim}(ii) the action of $\Frob_q$ on $H^{i}_{\et}(\AA(L)_n;\Q_\ell)$ is multiplication by $q^i$, and the action on $\Q_\ell(-n)$ is multiplication by $q^n$, so it follows that $\Frob_q$ acts on $H^{2n-i}_{c}(\BB(L)_n; V_n)$ by multiplication by $q^{n-i}$.
% and moreover that
%\beq
%\dim H^{2n-i}_{\et}(\BB(L)_n; \Q_\ell) = \dim H^{i}_{\et}(\BB(L)_n;\Q_\ell) = \dim H^{i}_{\et}(\AA(L)_n;\Q_\ell)^{S_n}
%\eeq
Putting this all together, we find that \[\tr\big(\Frob_q:H^{2n-i}_{c}(\BB(L)_n;\V)\big)=q^{n-i}\langle \chi,H^i(\AA(L)_n)\rangle.\] Substituting this into \eqref{eq:firstV} yields the desired formula. \end{proof}

\subsection{Stabilization of point counts} 
\label{section:stabilization}

We now consider the stabilization of the point count $q^{-n}|\BB(L)_n(\F_q)|$ as $n\to \infty$. We also describe certain families of statistics $\chi_n\colon S_n\to \Q$ for which the  sum $q^{-n}\sum_{f\in \BB(L)_n(\F_q)}\chi_n(f)$ will similarly stabilize. These sequences will always stabilize ``formally'', but to obtain actual convergence we require a condition on the growth of the representations $H^i(\AA(L)_n)$; see Definition~\ref{def:convergent} for the precise definition.

Recall from the introduction that a \emph{character polynomial} is a polynomial $P\in \Q[X_1,X_2,\ldots]$. Such a character polynomial $P$ simultaneously determines a class function  $P_n\colon S_n\to \Q$ for all $n$ in the following way.
For each $k\geq 1$, define the class function $c_k\colon S_n\to \Q$ by setting $c_k(\sigma)$ equal to the number of $k$-cycles of $\sigma$. The assignment $X_k\mapsto c_k$ extends to a ring homomorphism from $\Q[X_1,X_2,\ldots]$ to the ring of class functions on $S_n$, under which $P\in \Q[X_1,X_2,\ldots]$ is sent to the class function $P_n\colon S_n\to \Q$ defined by
\beq
P_n(\sigma) = P(c_1(\sigma), c_2(\sigma), \ldots, c_k(\sigma)).
\eeq

\begin{definition}[{\bf Sequence $\chi_n$ given by character polynomial}]
A sequence of characters $\chi_n\colon S_n\to \Q$ is \emph{given by a character polynomial} (resp.\ \emph{eventually given by a character polynomial}) if there exists $P\in \Q[X_1,X_2,\ldots]$ such that $\chi_n=P_n$ for all $n$ (resp.\ for all $n\geq N$ for some $N$). 
\end{definition}

If it exists, this character polynomial $P$ is uniquely determined.
We will sometimes write $\langle P,Q\rangle_n$ as shorthand for the inner product $\langle P_n,Q_n\rangle$ of $S_n$-characters.

The expectations of character polynomials \[\E_{\sigma\in S_n} P_n(\sigma)=\frac{1}{n!}\sum_{\sigma\in S_n}P_n(\sigma)=\langle P_n,1\rangle\]  compute the averages of natural combinatorial statistics with respect to the uniform distribution on $S_n$. For example, the well-known fact that a randomly chosen permutation $\sigma\in S_n$ has 1 fixed point on average says that $\langle X_1,1\rangle_n=1$ for all $n\geq 1$. Similarly, the fact that a random permutation in $S_n$ has $\frac{1}{k}$ $k$-cycles on average says that $\langle X_k,1\rangle_n=\frac{1}{k}$ for all $n\geq k$ (for $n<k$, of course, there are no $k$-cycles at all). More complicated character polynomials still express natural statistics: for example, if $P=(X_2-\frac{1}{2})^2$, then $\langle P,1\rangle_n$ is the \emph{variance} of the number of transpositions $X_2$, and it is not hard to calculate that this is equal to $\frac{1}{4}$ for $1\leq n<4$ and $\frac{1}{2}$ for all $n\geq 4$.

In these examples the expectation is independent of $n$ except for some finite initial segment. In fact this is a general property of character polynomials, as the following proposition shows. The \emph{degree} $\deg P$ of a character polynomial $P\in \Q[X_1,X_2,\ldots]$ is defined in the usual way, except that $X_i$ is defined to have degree $i$.
\begin{proposition}
\label{pr:charpolyexpect}
Given two character polynomials $P,Q\in \Q[X_1,X_2,\ldots]$, the inner product $\langle P_n,Q_n\rangle$ of $S_n$-characters is independent of $n$ once $n\geq \deg P+\deg Q$.
\end{proposition}
\begin{proof}
Since $\langle P_n,Q_n\rangle = \langle P_n\cdot Q_n,1\rangle$, it suffices to prove the proposition in the case when $Q=1$. In this case the inner product $\langle P_n,1\rangle$ is just the average \[\langle P_n,1\rangle=\E_{\sigma\in S_n} P_n(\sigma)=\frac{1}{n!}\sum_{\sigma \in S_n} P_n(\sigma).\]
Given a sequence $\mu=(\mu_1,\mu_2,\ldots,\mu_k)$ of non-negative integers with $\mu_k>0$, we define the character polynomial $\binom{X}{\mu}$ as the product of binomials \[\binom{X}{\mu}\coloneq \binom{X_1}{\mu_1}\binom{X_2}{\mu_2}\cdots\binom{X_k}{\mu_k}.\] The character polynomial $\binom{X}{\mu}$ has degree $|\mu|$, where $|\mu|\coloneq \sum i\cdot \mu_i$, and in fact a basis for the character polynomials of degree $\leq n$ is given by all the $\binom{X}{\mu}$ with $|\mu|\leq n$. By the linearity of expectation, it suffices to prove the proposition for $P=\binom{X}{\mu}$.

When $n=|\mu|$, the class function $P=\binom{X}{\mu}$ takes only the values 0 and 1, since the only way $P_n(\sigma)$ can be nonzero is if $\sigma$ has cycle type $\mu$ (i.e.\ has exactly $\mu_i$ $i$-cycles). Thus the expectation is just the proportion of permutations lying in this conjugacy class. It is easy to check that this conjugacy class has size $\frac{|\mu|!}{z_\mu}$, where $z_\mu=\prod_i i^{\mu_i}\cdot \mu_i!$. Therefore $n=|\mu|$ implies $\langle \binom{X}{\mu},1\rangle_n=\frac{1}{z_\mu}$.

Now consider arbitrary $n\geq |\mu|$. For each $\sigma\in S_n$ the value of $P_n(\sigma)$ is the number of ways to choose $\mu_1$ 1-cycles from $\sigma$, $\mu_2$ 2-cycles, etc. For each such choice, the union of the supports of these cycles determines a subset $S\subset [n]$ with $|S|=|\mu|$, and conversely such a choice is determined by a subset $|S|=|\mu|$ for which the restriction $\sigma|_S$ has cycle type $\mu$. Thus by linearity of expectation we can write  
\[\langle P_n,1\rangle=\frac{1}{n!}\sum_{\sigma \in S_n}\sum_{|S|=|\mu|}\delta_\mu(\sigma,S)\] 
where $\delta_\mu(\sigma,S)$ is equal to 1 if $\sigma(S)=S$ and $\sigma|_S$ has cycle type $\mu$, and equal to 0 otherwise. Exchanging the order of summation yields $\frac{1}{n!}\sum_{|S|=|\mu|}\sum_{\sigma\in S_n}\delta_\mu(\sigma,S)$. The inner sum simply counts the number of permutations $\sigma$ with $\sigma|_S$ of cycle type $\mu$. Such a permutation is determined independently by $\sigma|_S$, for which there are $\frac{|\mu|!}{z_\mu}$ possibilities by the previous paragraph, and by $\sigma|_{[n]-S}$, for which there are $(n-|\mu|)!$ possibilities. Since the inner sum does not depend on $S$, the outer sum reduces to multiplication by $\binom{n}{|\mu|}$, the number of subsets $S$. We conclude that
\[\langle {\textstyle\binom{X}{\mu}},1\rangle_n=\frac{1}{n!}\cdot \frac{n!}{|\mu|!\cdot (n-|\mu|)!}\cdot \frac{|\mu|!}{z_\mu}\cdot (n-|\mu|)!=\frac{1}{z_\mu}\] for all $n\geq |\mu|$, as desired.
\end{proof}

\begin{numberedremark}
\label{rem:poisson}
We saw above that the number of $i$-cycles $X_i$ has mean $1/i$ for $n\geq i$. The formula $\langle \binom{X}{\mu},1\rangle_n=\frac{1}{z_\mu}$ from the proof above allows us to compute the higher moments of this statistic. Working equivalently with the factorial moments, we have \[\E_{S_n} \binom{X_i}{k}=\frac{1}{i^k\cdot k!}=\frac{(1/i)^k}{k!}\quad\text{ for all }n\geq k\cdot i.\] It is well-known (and easy to check) that the factorial moments of a Poisson distribution with mean $\theta$ are equal to $\frac{\theta^k}{k!}$. This means that with respect to the uniform distribution on $S_n$, as $n\to \infty$ the random variable $X_i$ converges in some sense to a Poisson distribution with mean $1/i$. Moreover, the fact that \[\E_{\sigma\in S_n}\binom{X}{\mu}=\frac{1}{z_\mu}=\prod_i \frac{(1/i)^{\mu_i}}{\mu_i!}=\prod_i \E_{\sigma\in S_n} \binom{X_i}{\mu_i}\quad\text{ for }n\geq |\mu|\] means that the random variables $X_i$ become \emph{independent} Poisson in the limit. This picture is well-known to probabilists; see \cite{Takacs} for an extensive history going back over 300 years.
\end{numberedremark}

\bigskip
It was proved in \cite[Theorem~1.6]{CEF} that if $V$ is a finitely-generated FI-module over a field of characteristic 0, then the characters $\chi_{V_n}$ are eventually given by a character polynomial. Therefore Proposition~\ref{pr:charpolyexpect} implies that for any fixed character polynomial $P$, the inner products $\langle P,V_n\rangle$ are eventually independent of $n$. Applying this to $H^i_{\et}(\AA(L);\Q_\ell)$ using Proposition~\ref{pr:hifg} yields the following corollary.

\begin{corollary}
For any character polynomial $P$, the inner products $\langle P,H^i_{\et}(\AA(L)_n;\Q_\ell)$ are eventually independent of $n$.
\end{corollary}

\begin{definition}[{\bf Convergent FI-CHA}]  We say an FI-CHA $\AA(L)$ over $k$ is {\em convergent} if it satisfies the following two equivalent conditions:
\begin{enumerate}[1.]
\item for each $a\geq 0$ there is a function $F_a(i)$, subexponential in $i$ and not depending on $n$, which bounds the dimension of the $S_{n-a}$-invariant subspace
\beq
\dim H^i_{\et}(\AA(L)_n;\Q_\ell)^{S_{n-a}}\leq F_a(i) \text{ for all }n\text{ and }i.
\eeq
\item for each character polynomial $P\in \Q[X_1,X_2,\ldots]$ there exists a function $F_P(i)$, subexponential in $i$ and not depending on $n$, such that:
\[ \big|\langle P,H^i_{\et}(\AA(L)_n;\Q_\ell)\rangle\big|\leq F_P(i) \text{ for all }n\text{ and }i.\]
\end{enumerate}
\label{def:convergent}
\end{definition}
We can see that these conditions are equivalent as follows. The character of the permutation representation $\Q[S_n/S_{n-a}]$ is given for all $n$ by $a!\cdot \binom{X_1}{a}$, so from $V^{S_{n-a}}\approx \Hom_{S_n}(\Q[S_n/S_{n-a}],V)$ we see that
\[\dim (V^{S_{n-a}})=\langle a! {\textstyle\binom{X_1}{a}},V\rangle.\] This shows that the second condition implies the first.

For the converse, for each partition $\lambda\vdash d$, let $\chi_\lambda\colon S_d\to \Q$ be the corresponding irreducible character of $S_d$. Define the character polynomial $P_\lambda\coloneq \sum_{|\mu|=d}\chi_\lambda(c_\mu)\binom{X}{\mu}$, where $c_\mu\in S_d$ is a permutation with cycle type $\mu$. The character polynomial $P_\lambda$ is the unique character polynomial $P$ of degree $\leq d$ for which $P_d=\chi_\lambda$ and $P_n=0$ for $n<d$. It also has the property that $\langle P_\lambda,V\rangle_n$ is a nonnegative integer for any $S_n$-representation $V$, because $\langle P_\lambda,V\rangle_n$ is the multiplicity of the irreducible $S_d$-representation $V_\lambda$ inside $V^{S_{n-d}}$. In particular, if $d_\lambda=\dim V_\lambda$ we have
\[\sum_{\lambda \vdash d}d_\lambda\langle P_\lambda,V\rangle_n=\dim V^{S_{n-d}}.\] This corresponds to the fact that the sum $\sum d_\lambda\chi_\lambda$ is the character of the regular representation, so that $\sum d_\lambda P_\lambda=d!\cdot \binom{X_1}{d}$. 
This formula implies that $\langle P_\lambda,V\rangle\leq \dim V^{S_{n-d}}$ for any $S_n$-representation $V$, so the first condition implies the second for the $P_\lambda$. The second condition for arbitrary $P$ follows, since any $P$ is a finite linear combination of the $P_\lambda$.

\begin{theorem}[{\bf Character polynomials for convergent arrangements}]
 Let $L$ be an FI-CHA over $\F_q$, and denote by $\langle P,H^i(\AA(L))\rangle\in \Q$ the limiting multiplicity
\[\langle P,H^i(\AA(L))\rangle \coloneq \lim_{n\to \infty} \langle P,H^i_{\et}(\AA(L)_n;\Q_\ell).\]
If $\AA(L)$ is convergent, then for any character polynomial  $P$:
%\beq
%\gamma(L,P) = \sum_{i=0}^\infty (-1)^i q^{-i} a(L,P,i)
%\eeq
%converges, where $a(L,P,i)$ are the stable multiplicities of \eqref{alpi}.
\begin{equation}
\label{eq:convergentlimit}
\lim_{n \ra \infty} q^{-n} \sum_{f \in \BB(L)_n(\F_q)} P(f) = \sum_{i=0}^\infty (-1)^i
\frac{\langle P,H^i(\AA(L))\rangle}{q^i}
%%%%%%\frac{a(L,P,i)}{(-q)^i}%(-1)^i q^{-i} a(L,P,i).
\end{equation}
In particular, both the limit on the left-hand side and the infinite sum on the right exist.
%\eeq
%approaches
%\beq
%\sum_{i=0}^\infty (-1)^i q^{-i} a(L,P,i)
%\eeq
%as $n \ra \infty$. 
\label{th:limit}
\end{theorem}
  
The special case of Theorem~\ref{th:limit} when $P=1$ demonstrates the existence of the limit
\beq
\lim_{n\to\infty}q^{-n} |\BB(L)_n(\F_q)|.
\eeq

\begin{proof}
%By Proposition~\ref{pr:hifg}, the cohomology  $H^i_{\et}(\AA(L)_n; \Q_\ell)$ is a finitely generated FI-module.  As mentioned above, it follows from \cite[Theorem 1.6]{CEF} that the character of $S_n$ acting on $H^i_{\et}(\AA(L)_n; \Q_\ell)$ agrees with some fixed character polynomial $Q$ for $n$ large enough, and thus that $\langle  P, H^i_{\et}(\AA(L)_n; \Q_\ell) \rangle$ stabilizes to $a(L,P,i)$ as $n \ra \infty$.  For concision's sake we refer to these stable coefficients simply as $a(i)$ hereafter. 
%It remains to show that the contributions of $\langle P, H^i_{\et}(\AA(L)_n;\Q_\ell) \rangle$ for  ``large" values of $i$ can be neglected.  This is where we use the convergence of $\AA$. 
Given $P$, let $F_P(i)$ be the subexponential function guaranteed by Definition~\ref{def:convergent}. Since $\big|\langle P,H^i_{\et}(\AA(L)_n;\Q_\ell)\rangle\big|\leq F_P(i)$ for all $n$, the stable multiplicity is also bounded: $\big|\langle P,H^i(\AA(L))\rangle\big|\leq F_P(i)$. Since  $F_P(i)$ is subexponential, the sum $\sum_{i=0}^\infty(-q)^{-i}\langle P,H^i(\AA(L))\rangle$ on the right side of \eqref{eq:convergentlimit} is absolutely convergent. 

By Theorem~\ref{th:chi}, the terms appearing in the limit on the left side of \eqref{eq:convergentlimit} are given by
\[
q^{-n}\sum_{f\in \BB(L)_n(\F_q)}P(f)=\sum (-q)^i\langle P,H^i_{\et}(\AA(L)_n;\Q_\ell)\rangle
\]
Given a threshold $I\geq 0$, choose $N$ large enough so that for all $n\geq N$ and all $i\leq I$ we have
\[ \langle P, H^i_{\et}(\AA(L)_n; \Q_\ell) \rangle= \langle P, H^i(\AA(L)) \rangle \]
Therefore for any $n\geq N$, the difference between the $n$th term on the left of \eqref{eq:convergentlimit} and the claimed limit is equal to
\[\sum_{i=I+1}^\infty (-q)^{-i}\big( \langle P, H^i(\AA(L))-\langle P,H^i_{\et}(\AA(L)_n;\Q_\ell)\rangle\big)\]
In absolute value, this sum is bounded by 
$\sum_{i=I+1}^\infty (F_P(i) + F_P(i)) q^{-i}$. The convergence of the series $\sum F_P(i)/q^i$ means that this bound goes to zero as $I\to \infty$.Therefore by taking $n$ sufficiently large, the left side of \eqref{eq:convergentlimit} approaches the right side, as desired. \end{proof}

\medskip

%\begin{cor}  For every finite field $\F_q$, the probability that a random monic squarefree polynomial of degree $n$ has no three roots in arithmetic progression approaches a limit as $n \ra \infty$.
%\label{cor:rootsap}
%\end{cor}

The problem, of course, is to determine to what extent the FI-CHAs of interest satisfy the convergence condition in Definition~\ref{def:convergent}.    For example, is this condition satisfied for the FI-CHA $\AA(x_1-x_2, x_1-2x_2+x_3)$ parametrizing squarefree polynomials with no three roots in arithmetic progression?  Here our knowledge is somewhat lacking.  Below will give many examples that are convergent.  It is possible that all FI-CHAs defined over finite fields are convergent, but we emphasize that a generic FI-CHA over $\C$ should \emph{not} be convergent, so if true this property is a special feature of FI-CHAs over finite fields.

\para{Stable multiplicities may depend on $p$} We conclude this section with a warning. In the next section, we consider the FI-CHA $\AA(x_1-x_2)$ over $\F_p$, and find that the stable multiplicities $\langle P,H^i(\AA(x_1-x_2))\rangle$ do not depend on the characteristic $p$. In general, if we start with any FI-CHA defined over the integers, we can reduce it modulo any prime $p$ to get an FI-CHA defined over $\F_p$. However, in general one should \emph{not} expect the stable coefficients $\langle P, H^i(\AA(L))$ to be independent of the prime $p$, as we now explain.

The cohomology of the complement of a hyperplane arrangement is determined by the combinatorics of the intersection lattice of the hyperplanes. This lattice consists of all the subspaces arising as intersections of various hyperplanes, under the relation of containment. If this intersection lattice is different in different characteristics, one should not expect the cohomology to remain the same. This situation can and does occur, even for quite natural FI-CHAs, as the following two examples illustrate.

\begin{xample} {\rm For any finite field of characteristic $p\neq 2$, the proportion of monic, squarefree polynomials  with no three roots in arithmetic progression is strictly less than $1$. But over a field of characteristic 2 this proportion is 1, since three distinct numbers cannot be in arithmetic progression modulo 2: if $x_2=x_1+y$ and $x_3=x_2+y$, then $x_3=x_1+2y=x_1$. In terms of the FI-CHA $\AA(x_1-x_2,x_1-2x_2+x_3)$, this can be seen from the fact that the form $x_i - 2x_j + x_k$ \emph{coincides} with $x_i - x_k$ in characteristic $2$, so the intersection lattice is obviously not the same as  in characteristic $0$.}
\end{xample}

\begin{xample}
{\rm We emphasize that the differences of intersection lattices need not be merely a matter of some finite set of primes being ``bad''. Consider the FI-CHA $\AA(x_1-2x_2)$.  The hyperplanes associated to $\AA(x_1-2x_2)_n$ include
\beq
x_1 - 2x_2=0,\  x_2 - 2x_3=0,\, \ldots,\  x_n - 2x_1=0
\eeq
Over $\C$, these $n$ hyperplanes intersect transversely: indeed their intersection is the subspace where \[x_1=2x_2=4x_3=\cdots=2^{n-1}x_n=2^nx_1,\] which is clearly trivial. The same argument shows that these hyperplanes remain transverse  whenever $2^n$ is not congruent to $1$ mod $p$, which is the case for all but finitely many $p$.  But when $p | (2^n - 1)$, the last linear equation is redundant with the first $n-1$, and so the intersection of the $n$ hyperplanes has codimension at most $n-1$.  Of course, every prime $p$ is a divisor of $2^n - 1$ for {\em some} choice of $n$, namely the multiplicative order of 2 in $\F_p^\times$. Since the FI-CHA $\AA(x_1-2x_2)$ involves all $n$ simultaneously, one should not expect the intersection lattice attached to $\AA(x_1-2x_2)$ modulo $p$ to be the same as that of $\AA(x_1-2x_2)$ modulo $p'$ unless $2$ has the same multiplicative order mod $p$ and mod $p'$, a very unusual occurrence.}
\end{xample}

\section{Statistics of squarefree polynomials and the cohomology of the pure braid group}

\label{s:purebraid}

In this section we explain how to go back and forth between the answers to statistical questions about squarefree polynomials over a finite field, on the one hand, and the FI-modules arising from the cohomology of the pure braid group, on the other.  This line of reasoning  (without the connection to FI-modules) was initiated by Lehrer;  see e.g.\ \cite{lehrer:survey,kisinlehrer,lehrer:discriminantvarieties}. The counting theorems presented here are for the most part not new.  The main point here is to elucidate 
the close connection between the asymptotics of arithmetic statistics over function fields and the FI-structure on the cohomology of the pure braid group. % This viewpoint lets us prove general theorems on these asymptotics in a unified way.

\subsection{Cohomology of the pure braid group}
Recall that $\PConf=\AA(x_1-x_2)$ is the $\FI^\op$-scheme consisting of hyperplane complements $\PConf_n=\A^n-\bigcup\{x_i=x_j\}$, whose points parametrize ordered tuples of distinct points in $\A^1$. The quotient $\Conf_n=\PConf_n/S_n=\BB(x_1-x_2)_n$ is the space of monic, squarefree,  degree $n$ polynomials.

In this specific case, the intersection lattice of the hyperplane arrangement $\PConf_n$ can be computed directly. The subspaces of $\A^n$ arising as intersections of the hyperplanes $x_i=x_j$ are in bijection with set partitions of the set $[n]=\{1,\ldots,n\}$. Given a partition $[n]=S_1\sqcup \cdots \sqcup S_m$, the corresponding subspace is defined by the equations $x_i=x_j$ whenever $i$ and $j$ lie in the same block $S_k$. For example, the partition $[5]=\{1,4\}\sqcup \{2,3,5\}$ corresponds to the subspace where $x_1=x_4$ and $x_2=x_3=x_5$.

This description shows in particular that the intersection lattice of $\PConf_n$ is the same whether the hyperplanes are considered over $\C$ or over a finite field $\F_q$.  It follows from the results of Lehrer in \cite{lehrer} that the comparison map
\beq
H^i_{\et}({\PConf_n}_{/\Fqbar};\Q_\ell)\to H^i(\PConf_n(\C);\Q_\ell)
\eeq
is an isomorphism of $S_n$-representations. In fact, the naturality of the FI-maps $\PConf_n\to \PConf_m$ implies that there is an isomorphism of FI-modules \beq
H^i_{\et}(\PConf_{/\Fqbar};\Q_\ell) \approx H^i(\PConf(\C);\Q_\ell).
\eeq

%\para{Cohomology of pure braid group}
Over $\C$, a monic squarefree degree-$n$ polynomial is determined by its set of roots, an unordered set of $n$ distinct points in the complex plane. Therefore the complex points $\Conf_n(\C)$ can be identified with the configuration space of $n$ distinct points in the plane. This space is well-known to be a $K(\pi,1)$ (Eilenberg-Mac Lane) space with fundamental group the \emph{braid group} $B_n$. In particular, this means that $H^i(\Conf_n(\C))\approx H^i(B_n)$. The finite cover $\PConf_n(\C)$ is also a $K(\pi,1)$ space, with fundamental group $\pi_1(\PConf_n(\C))=P_n$ the \emph{pure braid group}, which sits inside the braid group as an index-$n!$ subgroup:
\begin{equation}
\label{eq:Pnses}
1\to P_n\to B_n\to S_n\to 1
\end{equation} In the same way, the cohomology $H^i(\PConf_n(\C);\Q_\ell)$ can identified with the group cohomology of the pure braid group $P_n$. This identification is $S_n$-equivariant with respect to the action of $S_n$ on $H^i(P_n)$ coming from \eqref{eq:Pnses}. Therefore Theorem~\ref{th:chi} takes the following form.

\begin{prop}[{\bf Twisted Grothendieck--Lefschetz for \boldmath$\Conf_n$}]
For each prime power $q$, each positive integer $n$, and each character polynomial $P$, we have 
\beq
\sum_{f \in \Conf_n(\F_q)} P(f) = \sum_{i=0}^n (-1)^i q^{n-i} \langle P, H^i(P_n;\Q) \rangle.
\eeq
\label{pr:exactcount}
\end{prop}

For example, when $P = 1$, the inner product $\langle P, H^i(P_n;\Q) \rangle$ is the  multiplicity of the trivial $S_n$-representation in $H^i(\PConf_n(\C);\Q)$, which by transfer is the  dimension of $H^i(\Conf_n(\C);\Q)$. Arnol'd proved that for $n\geq 2$ this dimension is $1$ for $i=0,1$ and $0$ for $i > 1$.  So one recovers from Proposition \ref{pr:exactcount} the well-known fact that for all $n \geq 2$,
\beq
\left|\Conf_n(\F_q)\right|= \sum_{f \in \Conf_n(\F_q)} 1 = q^n - q^{n-1}.
\eeq
In other words, the number of squarefree monic polynomials of degree $n$ equals  $q^n - q^{n-1}$. We give another exact computation of this sort in Section~\ref{sec:standard} below.

We can also describe the limits of such statistics as $n\to \infty$ as in Theorem~\ref{th:limit}. To do this, we must first verify that the FI-CHA $\PConf_n$ is convergent.

\begin{prop}[\boldmath{$\PConf_n$} {\bf is a convergent FI-CHA}]  
For any field $k$, the FI-CHA $\PConf_n$ over $k$ is convergent in the sense of Definition~\ref{def:convergent}.
\label{pr:PConfconvergent}
\end{prop}

\begin{proof} By the discussion above, it suffices to prove that for each $a$ the invariant cohomology $H^i(P_n;\C)^{S_{n-a}}$ is bounded uniformly in $n$ and subexponentially in $i$. Lehrer-Solomon in \cite{lehrersolomon} provide an explicit description of $H^i(P_n;\C)$ as a sum
of induced representations
\[H^i(P_n;\C)=\bigoplus_\mu \Ind_{Z(c_\mu)}^{S_n}(\xi_\mu)\]
where $\mu$ runs over the set of conjugacy classes in $S_n$ of permutations having $n-i$ cycles, $c_\mu$ is any element of the conjugacy class $\mu$, and $\xi_\mu$ is a one-dimensional character of the centralizer  $Z(c_\mu)$ of $c_\mu$ in $S_n$, described explicitly below. 

We will discuss this description in great detail in the next section, but for now it suffices to remark that a permutation $c_\mu$ decomposing into $n-i$ cycles must have at least $n-2i$ fixed points. This implies that the centralizer $Z(c_\mu)$ contains the subgroup $S_{n-2i}$. Therefore  the dimension of the $S_{n-a}$-invariants in the induced representation $\Ind_{Z(c_\mu)}^{S_n}(\xi_\mu)$
is bounded above by the number of double cosets in $S_{n-a} \bs S_n / S_{n-2i}$, which is polynomial in $i$.  Indeed, it is equal to the number of maps $f\colon \{1,\ldots,a\}\to \{1,\ldots,2i,\star\}$ such that $|f^{-1}(j)|\leq 1$ and $|f^{-1}(\star)|\leq n-2i$; for fixed $a$ this is bounded by a constant times the number of subsets of $\{1,\ldots,2i,\star\}$ of size $\leq a$, which is $O(i^a)$.

The summands contributing to $H_i$ correspond to conjugacy classes $c_\mu$ in $S_n$ decomposing into $n-i$ cycles, which are in bijection with partitions on $i$ (by recording 
$\text{length}-1$ for each cycle). Since the number of partitions of $i$ is subexponential in $i$, and the contribution of each summand to $H^i(P_n;\C)^{S_{n-a}}$ is polynomial in $i$, this completes the proof.
\end{proof}

%In particular, the left-hand side is a {\em polynomial} in $q$.  This immediately implies a ``uniform" version of Proposition~\ref{pr:limitexists}:
We proved in \cite[Theorem~4.7]{CEF} that $H^i(\PConf_n(\C);\Q)$ is given for all $n\geq 0$ by a single character polynomial $Q$ of degree $\leq 2i$. Therefore 
Proposition~\ref{pr:charpolyexpect} yields the first claim of the following proposition. 
Since $\PConf_n$ is a convergent FI-CHA, Theorem~\ref{th:limit} gives the second claim, which relates the limiting statistics of squarefree polynomials with the representation-stable cohomology of the pure braid group.
\begin{prop}  For any character polynomial $P$, the inner product $\langle P,H^i(\PConf_n(\C);\Q)\rangle$ is independent of $n$ for $n\geq 2i+\deg P$. Furthermore, if we let  \[\langle P,H^i(\PConf(\C))\rangle \coloneq\lim_{n\to \infty}\langle P, H^i(\PConf_n(\C);\Q)\rangle=\lim_{n\to \infty}\langle P, H^i(P_n;\Q)\rangle,\] then for each prime power $q$, we have:
\beq
\lim_{n \ra \infty} q^{-n} \sum_{f \in \Conf_n(\F_q)} P(f) = \sum_{i=0}^\infty (-1)^i \frac{\langle P,\ H^i(\PConf(\C))\rangle}{q^{i}}
\eeq
In particular, both the limit on the left and the series on the right converge.\label{pr:purebraidlimit}
\end{prop}
In  Section~\ref{sec:Lfunctions} we compute the limiting values of such statistics for some explicit character polynomials $P$. Moreover, using Proposition~\ref{pr:purebraidlimit} it is also possible to \emph{deduce} results about the representation-stable cohomology from these computations.

\subsection{The standard representation}
\label{sec:standard}
%In this section, we give an exact computation of the multiplicity $\langle X_1,H^i(P_n)\rangle$ for all $i$ and all $n$. By Proposition~\ref{pr:exactcount}, this will let us compute
%\beq
%\sum_{f \in \Conf_n(\F_q)} X_1(f),
%\eeq
%the total number of linear factors among all squarefree polynomials of degree $n$. Moreover Proposition~\ref{pr:purebraidlimit} tells us that the ration between this total number and $q^n$ converges to a limit as $n\to \infty$. Since the number of such polynomials is $q^n-q^{n-1}$, this lets us compute the {\em average} number of linear factors of a squarefree polynomial over $\F_q$, as the degree of the polynomial goes to $\infty$, as
%\beq
%\lim_{n \ra \infty} \frac{ \sum_{f \in \Conf_n(\F_q)} X_1(f)} {\sum_{f \in \Conf_n(\F_q)}1 } = 
%\frac{1}{1-q^{-1}}\cdot \lim_{n \ra \infty}  q^{-n}  \sum_{f \in \Conf_n(\F_q)} X_1(f).
%\eeq

The number of linear factors of a polynomial $f(T)\in\F_q[T]$ is the number of fixed points of the permutation $\sigma_f$  induced by Frobenius on the roots of $f(T)$.  If one thinks of this permutation as something like a ``random permutation", one would expect the average number of fixed points to be $1$. 

This expectation might be supported by the fact that an \emph{arbitrary} polynomial has precisely 1 linear factor on average. To see this, note that for fixed $x\in \F_q$, the number of polynomials $f(T)$ with $f(x)=y$ is the same for every $y\in \F_q$ (consider the family $f(T)+z$ for $z\in \F_q$). Therefore for each $x$, the chance that $f(x)=0$ is $1/q$; summing over the $q$ possible roots $x$ shows that the average number of roots overall is $1$. 

However, we shall see that the average number of linear factors of a \emph{squarefree} polynomial is in fact not $1$, but approaches
\begin{equation}
\label{eq:X1limit}
\frac{1}{1+\frac{1}{q}} = 1 - \frac{1}{q} + \frac{1}{q^2} - \frac{1}{q^3} + \ldots
\end{equation}
as $n\to \infty$. In other words, a squarefree polynomial has slightly {\em fewer} linear factors on average than do arbitrary polynomials.  On reflection, one can see why:  the squarefree, degree $n$ polynomials that are multiples of a linear polynomial $L$ can all be written as $Lg$ with $\deg g = n-1$, but there is a further condition on $g$ beyond the requirement that it be squarefree; it must be coprime to $L$.  We refer to the work of Arratia, Barbour, and Tavare~\cite{abt} for a much more refined analysis of the distribution on permutations coming from random polynomials over $\F_q$; in short, one has that these permutations are ``equidistributed with respect to long cycles."  Statistics like $X_1$, on the other hand, which are sensitive to (very!) short cycles, may diverge from the corresponding statistics for the uniform distribution on permutations, as we see in the present case.

We begin by giving a precise computation of the average number of linear factors. This establishes the formula (2) from Table A in the introduction. As we will see, this computation is quite involved; in the next section  we will see that an answer only as $n\to \infty$ can be obtained more quickly.

\begin{proposition}[{\bf Expected number of linear factors}]
The expected number of linear factors for a monic, square free, degree $n$ polynomial $f(T)\in \F_q[T]$ is $1-\frac{1}{q}+\frac{1}{q^2}-\cdots\pm\frac{1}{q^{n-2}}$.
\label{pr:exactlinear}
\end{proposition}

The proof of Proposition~\ref{pr:exactlinear} rests on the following computation of $\langle X_1,H^i(P_n;\Q)\rangle$.\begin{proposition}
\label{pr:standard}
  For each $i\geq 1$,
  \[\big\langle X_1,H^i(P_n;\Q)\big\rangle =\begin{cases}0&\text{ for }n\leq i\\1&\text{ for }n=i+1\\2&\text{ for }n\geq i+2\end{cases}\]
\end{proposition}

We will derive Proposition~\ref{pr:standard} from the description by Lehrer-Solomon \cite{lehrersolomon} of $H^*(P_n;\C)$ as a
sum of induced representations, one for each conjugacy class $c_\mu$
in $S_n$. The conjugacy classes $c_\mu$ contributing to $H^i(P_n;\C)$ are those decomposing into $n-i$ cycles. Lehrer-Solomon \cite{lehrersolomon} prove that 
\begin{equation}
\label{eq:hi}
H^i(P_n;\C)=\bigoplus_\mu \Ind_{Z(c_\mu)}^{S_n}(\xi_\mu)
\end{equation}
where the one-dimensional characters $\xi_\mu\colon Z(c_\mu)\to \C^\times$ are described as follows.  Let $\mu_j$ be the number of $j$-cycles in $c_\mu$, so $n=\sum j\cdot \mu_j$.  The centralizer $Z(c_\mu)$ is the product  of wreath
products $\Z/j\Z\wr S_{\mu_j}=(\Z/j\Z)^{\mu_j}\rtimes S_{\mu_j}$, where the
$S_{\mu_j}$ factor acts by permuting the $j$--cycles in the
decomposition.

On each $\Z/j\Z$ factor, the character $\xi_\mu$ sends a generator to $\eta_j=(-1)^{j+1}e^{2\pi
  i/j}$; we will need only that $\eta_j\neq 1$ except when $j=2$, so $\xi_\mu$ is nontrivial on all $\Z/k\Z$ factors with $k\geq 3$. 
  %This is nontrivial except when $j=2$, when it is the trivial representation $\Z/2\Z\to \C^\times$; it is faithful except when $j\equiv 2\bmod{4}$, in which case the kernel has
%order 2.
For $j$ odd the character $\xi_\mu$ is trivial on the subgroup $S_{\mu_j}$, while for $j$
even, $\xi_\mu$ restricts to $S_{\mu_j}$ as the sign representation $S_{\mu_j}\to \{\pm 1\}\subset \C^\times$. As long as $\mu_2>1$, this makes the representation $\xi_\mu$ nontrivial on the $\Z/2\Z \wr
S_{\mu_2}$ factor as well.  We remark that although every representation of $S_n$ can be defined over $\Q$, the characters $\xi_\mu$ cannot be.

Proposition~\ref{pr:standard} is immediate from the following lemma, which shows
that the only two summands of the right-hand side of Equation \eqref{eq:hi} that contribute to $\langle X_1,H^i(P_n;\C)\rangle$ are
$c_\mu=(1\ \cdots i+1)$ and $c_\mu=(1\ \cdots\ i)(i+1\ i+2)$, which contribute when $n\geq i+1$ 
and when $n\geq i+2$, respectively.

\begin{lemma}
\label{lem:X1multiplicities}
  For all $n\geq 1$ and all conjugacy classes $c_\mu$, the inner product $\langle X_1, \Ind_{Z(c_\mu)}^{S_n}(\xi_\mu)\rangle$ equals 0 except in the following cases (the last two entries apply to $k\geq 3$): 
 
 \bigskip 
\begin{centering}
\begin{tabular}{llll}
&$c_\mu$&$\Z(c_\mu)$&\!\!\!\!$\langle X_1,\Ind_{Z(c_\mu)}^{S_n}(\xi_\mu)\rangle_n$\\
\hline
$H^0$&\ $\id$&\ $S_n$&\qquad$=1\qquad n\geq 1$\\
$H^1$&\ $(1\ 2)$&\ $\Z/2\Z\times S_{n-2}$&\qquad$=1\qquad n=2$\\
&&&\qquad$=2\qquad n\geq 3$\\
$H^2$&\ $(1\ 2)(3\ 4)$&\ $\Z/2\Z\wr S_2\times S_{n-4}$&\qquad$=1\qquad n\geq 4$\\
$H^{k-1}$&\ $(1\ \cdots\ k)$&\ $\Z/k\Z\times S_{n-k}$&\qquad$=1\qquad n\geq k$\\
$H^{k}$&\ $(1\ \cdots\ k)(k+1\ k+2)$&\ $\Z/k\Z\times \Z/2\Z\times S_{n-k-2}$&\qquad$=1\qquad n\geq k+2$
\end{tabular}
\end{centering}
%c_\mu&=\id&\langle X_1,\Ind_{S_n}^{S_n}\C\rangle_n=\langle X_1,1\rangle_n&=1\text { for all }n\geq 1\\
%c_\mu&=(1\ 2)&\langle X_1,\Ind_{S_2\times S_{n-2}}^{S_n}\C\rangle_n&=2\text { for all }n\geq 3 (=1 \text{ for }n=2)\\
%c_\mu&=(1\ 2)(3\ 4)&\langle X_1,\Ind_{\Z/2\Z\wr S_2 \times S_{n-4}}^{S_n}\C^{\epsilon_2}\rangle_n&=1\text { for all }n\geq 4\\
%c_\mu&=(1\ \cdots\ k) \text{ for }k\geq 3&\langle X_1,\Ind_{\Z/k\Z\times S_{n-k}}^{S_n}\C^{\eta_k}\rangle_n&=1\text { for all }n\geq k\\
%c_\mu&=(1\ \cdots\ k)(k+2) \text{ for }k\geq 3&\langle X_1,\Ind_{\Z/k\Z\times \Z/2\Z\times S_{n-k-2}}^{S_n}\C^{\eta_k}\rangle_n&=1\text { for all }n\geq k+2\\
\end{lemma}

\begin{proof}
Since $X_1$ is the character of the permutation $S_n$-representation $V=\C^n$, the inner product $\langle X_1, \Ind_{Z(c_\mu)}^{S_n}(\xi_\mu)\rangle$ computes the dimension of $\Hom_{S_n}(\Ind_{Z(c_\mu)}^{S_n}(\xi_\mu),\,\C^n)$.  The defining property of $\Ind$ implies:  
\[\Hom_{S_n}(\Ind_{Z(c_\mu)}^{S_n}(\xi_\mu),\,\C^n)\approx \Hom_{Z(c_\mu)}(\xi_\mu,\C^n).\]  
In other words, we seek to compute the dimension of the $\xi_\mu$-isotypic component 
\[V^{\xi_\mu}\coloneq \{v\in \C^n\ |\ \sigma\cdot v=\xi_\mu(\sigma) v\quad \forall \sigma\in Z(c_\mu)\}.\]
Let $e_1,\ldots,e_n$ be the standard basis for $\C^n$. Consider a factor $\Z/k\Z<Z(c_\mu)$ generated by a $k$-cycle. The representation $V=\C^n$ restricts to this subgroup as $\C^k\oplus \C^{\oplus n-k}$, where $\C^k$ denotes the regular representation of $\Z/k\Z$. 
In particular, in the case $k\geq 3$ when the character $\xi_\mu$ is nontrivial on $\Z/k\Z$, the $\xi_\mu|_{\Z/k\Z}$-isotypic component is one-dimensional. Explicitly, if $\Z/k\Z$ is generated by the $k$-cycle $(1 \cdots k)$ and $\xi_\mu$ sends this generator to the root of unity $\eta_k\neq 1$, then $V^{\xi_\mu|_{\Z/k\Z}}$ is spanned by $v_k=\eta_k e_1 + \eta_k^2 e_2 + \cdots + \eta_k^k e_k$.

Any other cycle in $c_\mu$ will fix the vector $v_k$. If $c_\mu$ contains another $l$-cycle $\Z/l\Z$ for $l\geq 3$ (whether $k=l$ or not), the factor $\Z/l\Z$ cannot act on $v_k$ by $\eta_l\neq 1$, and so $V^{\xi_\mu}=0$. Similarly, if $c_\mu$ contains more than one 2-cycle, the factor $\Z/2\Z\wr S_{\mu_2}$ cannot act on $v_k$ by the sign representation of $S_{\mu_2}$, and so again $V^{\xi_\mu}=0$.
This rules out all conjugacy classes containing a $k$-cycle with $k\geq 3$ except those of the form $(1\ \cdots\ k)$ and $(1\ \cdots\ k)(k+1\ k+2)$.

Furthermore, if $c_\mu$ contains more than three 2-cycles then $Z(c_\mu)$ has a subgroup of the form $\Z/2\Z\wr S_3$. The character $\xi_\mu$ restricts to the subgroup $S_3$ as the sign representation. The representation $V=\C^n$ restricts to this subgroup as $\C^3\oplus \C^3\oplus \C^{n-6}$, where $\C^3$ denotes the permutation representation of $S_3$. Since the sign representation of $S_3$ does not appear in $\C^3$ or in the trivial representation $\C$, we conclude that $V^{\xi_\mu}=0$ in this case. This rules out all conjugacy classes containing only 2-cycles except $(1\ 2)$ and $(1\ 2)(3\ 4)$.

It is now easy to verify the claimed multiplicities in the remaining cases. For $c_\mu=(1\ \cdots\ k)$, we already saw that $V^{\xi_\mu}$ is 1-dimensional, spanned by $v_k=\eta_k e_1 + \cdots + \eta_k^k e_k$. This vector is also fixed by $S_2\times S_{n-k-2}<S_{n-k}$, and so the same vector $v_k$ spans $V^{\xi_\mu}$
for $c_\mu=(1\ \cdots\ k)(k+1\ k+2)$. For $c_\mu=(1\ 2)(3\ 4)$, one can check by hand that $V^{\xi_\mu}$ is 1-dimensional and spanned by $e_1 + e_2 - e_3 - e_4$. Finally, for $c_\mu=(1\ 2)$ we find that a basis for $V^{\xi_\mu}=V^{S_2\times S_{n-2}}$ is given by $e_1+e_2$ and $e_3+\cdots+e_n$. The latter only occurs when $n\geq 3$, so for $n=2$ we have $\dim V^{S_2\times S_{n-2}}=\dim V^{S_2}=1$, while for $n\geq 3$ we have $\dim V^{S_2\times S_{n-2}}=2$ as claimed.
\end{proof}

\begin{proof}[Proof of Proposition~\ref{pr:exactlinear}]
We  apply Proposition~\ref{pr:exactcount} to the character polynomial $X_1$. For $i=0$ we have $\langle X_1,H^0(P_n;\Q)\rangle=\langle X_1,1\rangle_n=1$ for all $n\geq 1$; for $1\leq i\leq n-2$ we have $\langle X_1,H^i(P_n;\Q)\rangle=2$ by Proposition~\ref{pr:standard}; and for $i=n-1$ we have $\langle X_1,H^i(P_n;\Q)\rangle=1$ by Proposition~\ref{pr:standard}. Finally $H^i(P_n;\Q)=0$ 
for $i\geq n$.  Proposition~\ref{pr:exactcount}  thus gives
\begin{equation} 
\label{eq:X1total}
\sum_{f \in \Conf_n(\F_q)} X_1(f) = q^n - 2q^{n-1}+2q^{n-2}-2q^{n-3}+\cdots\mp 2q^3\pm 2q^2\mp q,
\end{equation}
where $\pm=(-1)^n$.
This formula is equivalent to Proposition~\ref{pr:exactlinear}, as can be seen in multiple ways. For example,
multiplying \eqref{eq:X1total} by $1+q^{-1}$ gives $q^n-q^{n-1}\pm q\mp 1$. Factoring this as $(q^n-q^{n-1})(1\pm q^{-(n-1)})$, this shows that \eqref{eq:X1total} is equal to \[(q^n-q^{n-1})\frac{1\pm q^{-(n-1)}}{1+q^{-1}}.\]
Dividing \eqref{eq:X1total} by $\left|\Conf_n(\F_q)\right|=q^n-q^{n-1}$ gives: 
\[\frac{\sum X_1(f)}{\left|\Conf_n(\F_q)\right|}=\frac{1\pm q^{-(n-1)}}{1+q^{-1}}=1-q^{-1}+q^{-2}-\cdots\pm q^{-(n-2)}\]
as claimed  in Proposition~\ref{pr:exactlinear}. 
\qedhere
\end{proof}

\para{Proposition~\ref{pr:exactlinear} and Occam's razor} 
Before moving on, we point out that Proposition~\ref{pr:exactlinear} allows us to give another perspective on the recent results of Kupers--Miller \cite{KM}. In that paper they consider the space $\Conf'_n(\C^d)$ parametrizing configurations of $n$ distinct points in $\C^d$, where one point is labeled and the other points are indistinguishable. Their results concern the stable cohomology $H^i(\Conf'(\C^d);\Q)\coloneq \lim_{n\to\infty} H^i(\Conf'_n(\C^d);\Q)$ for $i>0$. Verifying Vakil--Wood's Conjecture H from \cite{VW}, they prove that the stable dimension $\dim H^i(\Conf'(\C^d);\Q)$ is  periodic in $i$. Moreover, Kupers--Miller prove that this dimension is 2 when $i=(2d-1)k$, and 0 otherwise.

We can give another proof of Kupers--Miller's result in the case $d=1$ using Proposition~\ref{pr:exactlinear}, and in fact compute the unstable cohomology of $\Conf'_n(\C)$. In this case their result says just that $\dim H^i(\Conf'(\C);\Q)=2$ for all $i>0$. We can identify the space $\Conf'_n(\C)$ with the space of degree-$(n+1)$ polynomials of the form $f(T)=(T-x)^2\cdot g(T)$, where $g(T)$ is a squarefree polynomial coprime to $T-x$. (The double root $x$ is the labeled point, while the $n-1$ roots of $g(T)$ are the indistinguishable points.) This shows that $\Conf'_n(\C)$ is the quotient of $\PConf_n(\C)$ by the subgroup $S_{n-1}<S_n$. Therefore by transfer, we have an isomorphism
\[H^i(\Conf'_n(\C);\Q)\approx H^i(\PConf_n(\C);\Q)^{S_{n-1}}.\]
As we used in the proof of Lemma~\ref{lem:X1multiplicities}, $\dim V^{S_{n-1}}=\langle X_1,V\rangle$. Therefore Proposition~\ref{pr:standard} shows that 
\[\dim H^i(\Conf'_n(\C);\Q)=\langle X_1,H^i(\PConf_n(\C);\Q)\rangle=2\quad\text{ for all }n\geq i+2.\] We remark that the results of \cite{KM} violate a tentative prediction of the stable Betti numbers made in \cite[Eq. 1.50]{VW}; for example, for $d=1$ this ``motivic Occam's Razor'' predicted that $\dim H^i(\Conf'(\C);\Q)=2$ when $i\equiv 0,1\bmod{4}$, rather than for all $i>0$.
%
%(This their motivic ``Occam's Razor'' predicts that this stable dimension is given by $\dim H^i(\Conf'(\C);\Q)=2$ for $i\equiv 0,1\bmod{4}$ and $\dim H^i(\Conf'(\C);\Q)=0$ for $i\equiv 2,3\bmod{4}$ \cite[Eq. 1.50]{VW}.

%
%
%contradicts a prediction made by Vakil--Wood in \cite{VW}. Their predictions concern the stable cohomology $H^i(\Conf'(\C);\Q)\coloneq \lim_{n\to\infty} H^i(\Conf'_n(\C);\Q)$ for $i>0$. Specifically,  Vakil--Wood's Conjecture H states that $\dim H^i(\Conf'(\C);\Q)$ should be periodic in $i$. Moreover, 
%
%Although Conjecture H does hold in this case, the specific prediction of \cite[Eq. 1.50]{VW} is not correct. Indeed, This is the quotient of $\PConf_n(\C)$ by the subgroup $S_{n-1}<S_n$. Therefore by transfer, we have an isomorphism
%\[H^i(\Conf'_n(\C);\Q)\approx H^i(\PConf_n(\C);\Q)^{S_{n-1}}.\] But as we used in the proof of Lemma~\ref{lem:X1multiplicities}, $\dim V^{S_{n-1}}=\langle X_1,V\rangle$. Proposition~\ref{pr:standard} shows that the  multiplicity $\langle X_1,H^i(\PConf_n(\C);\Q)\rangle=2$ for all $n\geq i+2$. Therefore we see that the stable cohomology is 2-dimensional for \emph{all} $i>0$, rather than only for $i\equiv 0,1\bmod{4}$ as predicted in \cite[Eq. 1.50]{VW}:
%\[\dim H^i(\Conf'(\C);\Q)= \langle X_1,H^i(\PConf(\C);\Q)\rangle=2\quad\text{ for all }i>0\]

%
We point out that this application only required the computation of the \emph{stable} multiplicities $\langle X_1,H^i(\PConf(\C);\Q)\rangle$, not the exact computations of Proposition~\ref{pr:standard}. Therefore the appeal to Proposition~\ref{pr:standard} could be replaced by the  $L$-function argument given in Section~\ref{sec:Lfunctions} below, which also yields the stable multiplicities.

\subsection{Using $L$-functions to compute representation-stable cohomology}
\label{sec:Lfunctions}
The average number of linear factors of a squarefree polynomial, in the limit as the degree goes to $\infty$, can also be computed by a direct counting argument in the style of analytic number theory.  
In this section we sketch this computation.

\para{The zeta function of $\F_q[T]$}
The chief actor in the story is the zeta function $\zeta(s)=\zeta_{\F_q[T]}(s)$ of the ring $\F_q[T]$. This is an analytic function of a complex variable $s$, which for $\Re s > 1$ is defined as the Euler product
\beq
\zeta(s) = \prod_{P} \frac{1}{1- q^{-s \deg P}}
\eeq
as $P=P(T)$ ranges over monic irreducible polynomials in $\F_q[T]$. Since every monic polynomial factors uniquely as a product of monic irreducible polynomials, we can expand
\[(1-q^{-s\deg P})^{-1}=1+q^{-s\deg P}+q^{-2s\deg P}+\cdots\]
Multiplying out gives 
the equivalent formula
\beq
\zeta(s) = \sum_f q^{-s \deg f}
\eeq
where here $f=f(T)$ ranges over all monic polynomials in $\F_q[T]$.

The definition of the zeta function $\zeta_{\F_q[T]}(s)$ parallels the classical definition of the Riemann zeta function $\zeta_\Z(s)$ from analytic number theory, with the monic irreducible polynomials in $\F_q[T]$ naturally standing in for the prime numbers in $\Z$. In the classical case, $\zeta_\Z(s)$ is defined by a Euler product or sum which converges for $\Re s>1$; it then extends to a meromorphic function on the whole complex plane by analytic continuation.  The same thing is true for $\zeta_{\F_q[T]}(s)$, except that in this case we can describe the resulting meromorphic function directly: it is $\zeta_{\F_q[T]}(s)=\frac{1}{1-q^{1-s}}$. Indeed this is true almost by definition; there are $q^n$ monic polynomials of degree $n$, so 
\beq
\zeta(s) =  \sum_f q^{-s \deg f} =\sum_{n=0}^\infty q^n\cdot q^{-sn} = \sum_{n=0}^\infty q^{n(1-s)} = \frac{1}{1-q^{1-s}}.
\eeq

\para{The $L$-function of $\Conf_n(\F_q)$} We define the $L$-function $L(s)$ as a weighted version of the zeta function, where we only count those monic polynomials $f(T)$ that are squarefree.
\beq
L(s)\coloneq \sum_{f\text{ squarefree}}q^{-s\deg f}= \sum_n \sum_{f \in \Conf_n(\F_q)} q^{-ns}
\eeq
Every squarefree monic polynomial factors uniquely as a product of irreducible polynomials, but now with the condition that no factor appears more than once. Therefore $L(s)$ can be broken up as an Euler product 
\begin{equation}
L(s)=\prod_P (1 + q^{-s \deg P})
\label{eq:eulerproduct}
\end{equation}
over monic irreducible polynomials $P=P(T)$ in $\F_q[T]$.  Since $1+q^k=\frac{1-q^{2k}}{1-q^k}$ we can rewrite this as:
\beq
L(s)=\prod_P \frac{1- q^{-2s \deg P}}{1-q^{-s \deg P}}  = \frac{\zeta(s)}{\zeta(2s)} = \frac{1-q^{1-2s}}{1-q^{1-s}}
\eeq

%where $\zeta(s) = \prod_P (1- q^{-s \deg P})$ is the zeta function of the affine line over $\F_q$.  But it is easy to see that
%\beq
%\zeta(s) = \sum_f q^{-s \deg f}
%\eeq
%over {\em all} monic polynomials $f$; since there are $q^n$ such polynomials of degree $n$, this yields
%\beq
%\zeta(s) = (1-q^{1-s})^{-1}
%\eeq
%and so
%\beq
%L(s) 
%\eeq
Finally, we define the {\em weighted $L$-function} $L(X_1,s)$ by:
\beq
L(X_1,s)= \sum_{f\text{ squarefree}}X_1(f)q^{-s\deg f} = \sum_n \sum_{f \in \Conf_n(\F_q)}X_1(f) q^{-ns}
\eeq
This is a ``weighted" version of $L(s)$, with each squarefree polynomial weighted by its number of linear factors. By standard analytic number theory techniques, the average value of $X_1(f)$ on squarefree polynomials $f$ is given by the ratio of the residue of $L(X_1,s)$ at $s=1$ to that of $L(s)$. Since this ``average'' is over polynomials of all degrees at once, in the notation of the previous section this will correspond not to the average for any finite $n$, but  to the limiting statistic \[\lim_{n\to \infty}\frac{\sum_{f\in \Conf_n(\F_q)} X_1(f)}{\left|\Conf_n(\F_q)\right|}.\]

The statistic $X_1$ breaks up as a sum $X_1=\sum_{x\in \F_q} X_1^{(x)}$, where $X_1^{(x)}(f)$ takes the value 1 or 0 depending on whether $T-x$ divides $f(T)$ or not. This lets us write $L(X_1,s)=\sum_{x\in \F_q} L(X_1^{(x)},s)$, where \beq
L(X_1^{(x)},s) \coloneq \sum_{f\text{ squarefree}}X_1^{(x)}(f)q^{-s\deg f}.
\eeq
The presence of the 0-1 variable $X_1^{(x)}(f)$ has the effect that this sum includes only those terms for which $T-x$ divides $f(T)$:
\[L(X_1^{(x)},s) =\sum_{\substack{f\text{ squarefree}\\(T-x)|f(T)}}q^{-s\deg f}\]
But this is very close to the definition of $L(s)$, differing only in one local factor of the Euler product~\eqref{eq:eulerproduct}. Specifically, the Euler factor of \eqref{eq:eulerproduct} at $P=T-x$ is $1+q^{-s}$, where the first term corresponds to polynomials $f(T)$ with $P\not| f$ and the second term to those polynomials with $P|f$. Therefore the difference between $L(s)$ and $L(X_1^{(x)},s)$ is just to replace $1+q^{-s}$ by $q^{-s}$ in the Euler product; in other words
\beq
L(X_1^{(x)},s) = q^{-s} \prod_{P \neq T-x} (1 + q^{-s \deg P}) = \frac{q^{-s}}{1+q^{-s}}L(s)
\eeq
Since the $L$-function $L(X_1^{(x)},s)$ does not depend on $x\in \F_q$, this gives\beq
L(X_1,s) = \sum_{x\in \F_q}L(X_1^{(x)},s) =\sum_{x\in \F_q} \frac{q^{-s}}{1+q^{-s}}L(s) =q \frac{q^{-s}}{1+q^{-s}} L(s)
\eeq
Therefore the desired ratio is $L(X_1,s)/L(s) = q^{1-s}/(1+q^{-s})$. In particular, 
  the residue at $s=1$ is just the limit as $s\to 1$, which as claimed in \eqref{eq:X1limit} is  \[\lim_{s\to 1}\  \frac{q^{1-s}}{1+q^{-s}}
=\frac{1}{1+q^{-1}}=1-\frac{1}{q}+\frac{1}{q^2}-\cdots.\]

%\para{$L$-functions for other statistics} 
\subsection{$L$-functions for other statistics} 
We saw in the previous subsection  that a computation of the cohomology of the pure braid group, as in Proposition~\ref{pr:standard}, yields as a corollary information about the cardinality of $\PConf_n(\F_q)$ for every $q$, as in Proposition~\ref{pr:exactlinear}.  On the other hand, for a fixed character polynomial $P$, the coefficients $\langle P, H^i(\PConf(\C))\rangle$ in Proposition~\ref{pr:purebraidlimit} are determined if we know the value of $\sum (-q)^{-i} \langle P, H^i(\PConf(\C))\rangle$ for every $q$ (or even infinitely many $q$).  It follows that we can go in the other direction, computing the dimensions of cohomology groups by means of counting points over finite fields.  In this context, this observation is due to Lehrer~\cite{lehrer}.

To see how this works, consider the character polynomial $P = \binom{X_1}{2} - X_2$.  This gives the character of the $S_n$-representation $\wedge^2 \Q^n$, where $\Q^n$ is the permutation representation of $S_n$ with character $X_1$. Just as above, we can study the $L$-function 
\beq
L(P,s) = \sum_{f\text{ squarefree}}P(f)q^{-s\deg f}= \sum_n \sum_{f \in \Conf_n(\F_q)}P(f) q^{-ns}.
\eeq
The weighting factor $P(f)$ here is the difference between the number of reducible quadratic factors of $f$ and the number of irreducible quadratic factors. 

We can compute an explicit closed form for $L(P,s)$ along the same lines as we did above for $L(X_1,s)$;  namely, we can break up $L(X_2,s)$ as a sum over irreducible quadratic polynomials $g(T)$ of $L(X_2^{(g)},s)$, the $L$-function counting squarefree polynomials divisible by a fixed irreducible quadratic $g(T)$. By analyzing local Euler factors, we find that
\beq
L(X_2^{(g)},s) = \frac{q^{-2s}}{1+q^{-2s}}L(s).
\eeq
Since there are $\binom{q}{2}=\frac{1}{2}(q^2-q)$ irreducible quadratics in $\F_q[T]$, this gives
\beq
L(X_2,s) = \tfrac{1}{2}(q^2-q)\frac{q^{-2s}}{1+q^{-2s}}L(s).
\eeq
Thus we have
\beq
\lim_{s \ra 1}\ \frac{L(X_2,s)}{L(s)} =\tfrac{1}{2}(q^2-q)\frac{q^{-2}}{1+q^{-2}}=\frac{q-1}{2(q+q^{-1})}=\frac{q^2-q}{2(q^2+1)}
\eeq
The quantity on the right hand side is the average number of irreducible quadratic factors of an squarefree polynomial over $\F_q$.  (As $q\to \infty$ this average approaches $1/2$, agreeing with the number of length-$2$ cycles of a random permutation.)

The same computation can be carried out for $\binom{X_1}{2}$; now we sum over squarefree {\em reducible} quadratic polynomials $(T-x)(T-y)$.  But the $L$-function counting squarefree polynomials divisible by $(T-x)(T-y)$ is obtained from $L(s)$ by changing \emph{two} local factors from $1+q^{-s}$ to $q^{-s}$ and therefore is equal to $\frac{q^{-s}}{1+q^{-s}}\cdot \frac{q^{-s}}{1+q^{-s}} L(s)$.
Since there are again $\frac{1}{2}(q^2-q)$ squarefree polynomials $(T-x)(T-y)$ in $\F_q[T]$, this gives
\[L({\textstyle \binom{X_1}{2}},s)=\tfrac{1}{2}(q^2-q)\frac{q^{-2s}}{(1+q^{-s})^2}L(s).\]
Therefore we find that the residue at $s=1$ is
\beq
\lim_{s \ra 1} \frac{L({\textstyle \binom{X_1}{2}},s)}{L(s)} =\tfrac{1}{2}(q^2-q)\frac{q^{-2}}{(1+q^{-1})^2}=\frac{q^2-q}{2(q+1)^2}.
\eeq
Putting these together, we find that the average value of $P(f) = \binom{X_1(f)}{2}  - X_2(f)$ converges to
\beq
\tfrac{1}{2}(q^2-q)\left(\frac{1}{(q+1)^2}-\frac{1}{q^2+1}\right)
%-(q-1)q^{-2}(1+q^{-1})^{-2}(1+q^{-2})^{-1}
\eeq
as $\deg f\to \infty$. Note that this expression is negative, with leading term $-1/q$;  that is, a squarefree polynomial tends to have slightly more irreducible quadratic factors than reducible ones, and this bias decreases as $q$ grows.

We can transfer this counting statement to a computation of stable cohomology.  From the computations above, we have
\beq
L(P,s) = \tfrac{1}{2}(q^2-q)q^{-2s}\left(\frac{1}{(1+q^{-s})^{2}} - \frac{1}{1+q^{-2s}}\right) L(s)
\eeq
By definition of $L(P,s)$, the sum of $P(f)$ over the squarefree polynomials $f(T)\in\Conf^n(\F_q)$ is  the coefficient of $q^{-ns}$ when the above expression is expanded in $q^{-s}$.  One can expand by hand to check directly that this coefficient is
\beq
q^n(-q^{-1} + 4q^{-2} - 7q^{-3} + 8q^{-4} - 9q^{-5} + 12q^{-6} - 15q^{-7} + 16 q^{-8} + 17q^{-9} - 20q^{-10} + \ldots)
\eeq
In other words, these numbers give the stable value of the multiplicity of $\bwedge^2 \Q^n$ in $H^i(P_n)$:
\beq
\big\langle\, \textstyle{\binom{X_1}{2}} - X_2,\ H^i(\PConf(\C))\,\big\rangle = \begin{cases}
2i&\text{ if }i\equiv 0\bmod{4}\\
2i-1&\text{ if }i\equiv 1\bmod{4}\\
2i&\text{ if }i\equiv 2\bmod{4}\\
2i+1&\text{ if }i\equiv 3\bmod{4}
\end{cases}
\eeq
Dividing the power series  above by $\left|\Conf_n(\F_q)\right|=q^n-q^{n-1}$ gives the expression given in (3) from Table A in the introduction.
 The same computation can be carried out from the Lehrer--Solomon description of the cohomology of the pure braid group, as we did for $P=X_1$ in Proposition~\ref{pr:standard} (and even the unstable multiplicities for finite $n$ can be obtained in this way), but the necessary analysis is more complicated.

%The number of squarefree monic polynomials of degree $n$ is $q^n - q^{n-1}$, and so the sum of $P(f)$ over all such polynomials is approximately
%\beq
%q^n(1-q^{-1})[-(q-1)q^{-2}(1+q^{-1})^{-2}(1+q^{-2})^{-1}]
%\eeq
%which evaluates to
%\beq
%q^n(1-q^{-1} + 4q^{-2} - 7q^{-3} + 8q^{-4} - 9q^{-5} + 12/q^{-6} - 15q^{-7} \ldots 
%\eeq

\subsection{Beyond character polynomials}

Methods similar to the above can be used to study the statistics of arithmetic functions on the space of squarefree polynomials that are not character polynomials.

\para{The M\"{o}bius function}
For example, let $\mu$ be the M\"{o}bius function on squarefree monic polynomials over $\F_q$; that is, $\mu(f)$ is $(-1)^d$ where $d$ is the number of irreducible factors in $f(T)$.  When $q$ is odd, $\mu(f)$ can also be expressed in terms of a Legendre symbol:
\beq
(-1)^{\deg f} \mu(f) =  \left(\frac{\Delta_f}{q}\right)
\eeq
where $\Delta_f$ is the discriminant of $f(T)$, which is necessarily nonzero because $f(T)$ is squarefree.  In other words, the M\"{o}bius function keeps track of whether the discriminant of $f$ is a quadratic residue.

The M\"{o}bius function is directly related to the action of Frobenius on the roots of $f(T)$. This action determines a permutation $\sigma_f$ in $S_n$, where $n=\deg f$, defined up to conjugacy. The sign $\epsilon(\sigma_f)$  of this permutation is $(-1)^a$ where $a$ is the number of {\em even-length} cycles in $\sigma_f$. The cycles in $\sigma_f$ correspond bijectively to the irreducible factors of $f$.  So if $b$ is the number of odd-length cycles in $\sigma_f$, then $(-1)^{a+b} = \mu(f)$ by definition.  On the other hand, since $b$ is congruent to $n$ mod $2$, this implies that $\epsilon(\sigma_f)\cdot (-1)^n=\mu(f)$. In other words
\begin{equation}
\label{eq:Moebius}
(-1)^{\deg f} \mu(f) = \epsilon(\sigma_f).
\end{equation}

From the perspective that the action of Frobenius on the roots is something like a ``random permutation" in $S_n$, one might expect $\mu(f)$ to take the value $1$ about half the time, and indeed this is the case.  As in the cases above, this can be proven either by an $L$-function argument or by a computation in stable cohomology of the pure braid group. We adopt the latter approach, which has the additional benefit of proving that $\mu(f)$ is 1 \emph{exactly} half the time, not just in the limit.  
We will need the following result of Lehrer-Solomon \cite[Proposition 4.7]{lehrersolomon}.
%; we include their proof for completeness.

\begin{lemma}%[{\bf sign rep doesn't occur in \boldmath$H^\ast(P_n;\Q)$}]
\label{lem:signrep}
The sign representation $\epsilon$ does not appear as an irreducible constituent of 
  $H^*(P_n;\Q)$ for any $i$ or any $n\geq 2$.
\end{lemma}
In fact, the total cohomology $H^*(P_n;\Q)$ is known to be isomorphic to two copies of $\Ind_{S_2}^{S_n}\Q$, so by Frobenius reciprocity we have $\langle \epsilon,H^*(P_n;\Q)\rangle_{S_n}=2\cdot \langle \epsilon,1\rangle_{S_2}=0$.
%\begin{proof}
%  By Lehrer-Solomon \cite{lehrersolomon}, each $H^i(P_n;\Q)$ is the sum of representations
%  induced up from certain linear characters $\xi_\mu$ on
%  centralizers $Z(c_\mu)$ over certain conjugacy classes $c_\mu$ in $S_n$. By Frobenius
%  reciprocity, it suffices to pair the restriction of $\epsilon$ to
%  $Z(c_\mu)$ with $\xi_\mu$; since these are both one-dimensional,
%  it suffices to check that $\xi_\mu$ is never equal to the sign
%  representation $\epsilon$. But this is clear from its definition; on
%  the $k$--cycles making up $c_\mu$, $\xi_\mu$ is not real except
%  on $2$--cycles, where it is trivial. This excludes all  $c_\mu$ except $c_\mu=\id$, when $Z(c_\mu)=Z(\id)=S_n$; but in this case $\xi_\mu$ is trivial, not the sign
%  representation.
%\end{proof}

Given Lemma~\ref{lem:signrep}, it is then immediate from Proposition~\ref{pr:exactcount} that for all $n\geq 2$,
\begin{equation}
\sum_{f \in \Conf_n(\F_q)} \mu(f) = 0.
\label{eq:muaverage}
\end{equation}
When $q$ is odd, this means that the discriminants $\Delta(f)$ of  the degree-$n$ squarefree polynomials $f(T)\in \F_q[T]$ are exactly evenly distributed between quadratic residues and non-residues, verifying (4) from Table A in the introduction.

\para{Irreducible and almost-irreducible polynomials}
How many of the monic squarefree polynomials of degree $n$ are irreducible?  Since irreducible polynomials in $\F_q[T]$ are the analogues of prime numbers in $\Z$, this question is the $\F_q[T]$ version of the {\em prime number theorem}.   More generally:  for a given $k$, how many of the monic squarefree polynomials of degree $n$ have no irreducible factor of degree less than $n/k$?  When $k=1$, this reduces to counting irreducible polynomials.

The answer to this question is known, by work of Panario and Richmond~\cite{panariorichmondbenor,thorne}.   Here we will explain how to prove results of this kind fairly simply, with explicit error terms, using the mechanisms of representation-stable cohomology for the pure braid group described here.

Let $\chi_k\colon S_n\to \{0,1\}\subset \Q$ be the class function for which $\chi_k(\sigma)=0$ if $\sigma$ contains a cycle of length $< n/k$, and $\chi_k(\sigma)=1$ if every cycle of $\sigma$ has length $\geq n/k$.  Our aim is then to estimate
\beq
\sum_{f \in \Conf_n(\F_q)} \chi_k(f)=
\sum_{i=0} (-1)^i q^{n-i}\langle \chi_k, H^i(P_n;\Q) \rangle
\eeq
% (Note that, even though $\chi_k$ is not a character polynomial, there is certainly a character polynomial on $S_n$ with which its values agree!)
and compare it to $\left|\Conf_n(\F_q)\right|=q^n-q^{n-1}$. We first note that the contribution of $H^i$ to the above sum is zero for all small positive $i$.

\begin{lemma} When $0<i<n/2k$ we have
\beq
\langle \chi_k, H^i(P_n;\Q) \rangle = 0.
\eeq
\label{le:vanishingchi}
\end{lemma}

\begin{proof}  It is shown in \cite[Theorem~4.7]{CEF} that the FI-module $H^i(P_\bullet;\Q)$ is in fact an $\FIsharp$-module. This implies by \cite[Theorem~2.67]{CEF} that the character of $H^i(P_n;\Q)$ is given by a single character polynomial $Q_i$ for \emph{all} $n\geq 0$ (not just for large enough $n$).  For instance, $Q_0 = 1$, $Q_1=\binom{X_1}{2}+X_2$, and we computed in \cite[Eq. 2]{CEF} that \beq
Q_2 = 2 \binom{X_1}{3} + 3 \binom{X_1}{4} + \binom{X_1}{2} X_2 - \binom{X_2}{2} -  X_3 - X_4.
\eeq
The degree of the character polynomial $Q_i$ coincides with the {\em weight} of the FI-module $H^i(P_n;\Q)$ as defined in \cite[Definition~2.50]{CEF}. The FI-module $H^1(P_n;\Q)$ is finitely generated by $H^1(P_2;\Q)\approx \Q$, so by \cite[Proposition~2.51]{CEF} the weight of $H^1(P_n;\Q)$ is at most 2. Since $H^i(P_n;\Q)$ is a quotient of the $i$th exterior power of $H^1(P_n;\Q)$, its weight is at most $2i$ by \cite[Proposition~2.62]{CEF}, so $\deg Q_i\leq 2i$.

The fact that $Q_i$ gives the character of $H^i(P_n;\Q)$ for \emph{all} $n\geq 0$ implies that for $i>0$ the polynomial $Q_i$ has no constant term (meaning that $Q_i(0,0,\ldots)=0$). This is easiest to see in two steps: first, note that the dimension of $H^i(P_n;\Q)$ is given by
\beq
\dim H^i(P_n;\Q) = Q_i(\id) = Q_i(n,0,0,\ldots).
\eeq

But for $n=0$ the group $P_n$ is trivial, so $\dim H^i(P_n;\Q)=0$ for $i>0$, verifying the claim.

Now let $\sigma\in S_n$ be a permutation with $\chi_k(\sigma)\neq 0$, so that no cycle of $\sigma$  has length shorter than $n/k$.  We have that $X_1(\sigma) = X_2(\sigma) = \ldots = X_{(n/k)-1}(\sigma) = 0$.  It follows that $Q(\sigma) = 0$ for any character polynomial of degree less than $n/k$ with no constant term. When computing the inner product $\langle \chi_k,Q\rangle$ for such $Q$, in every term $\chi_k(\sigma)Q(\sigma)$ one or both of the factors is 0, so we have $\langle \chi_k,Q\rangle=0$. When $0<i<n/2k$ we saw above that $Q_i$ has degree $\leq 2i<n/k$ and has no constant term, so $\langle \chi_k,Q_i\rangle=0$ as claimed.
\end{proof}

We now need to bound the contribution of the larger values of $i$.  Our main tool is the following lemma.

\begin{lemma}
Let $\chi$ be a class function on $S_n$ such that $|\chi(\sigma)| \leq 1$ for all $\sigma \in S_n$.  Then
\beq
\big|\langle \chi, H^i(P_n;\Q) \rangle\big| \leq p(2i)
\eeq
where $p(m)$ is the partition function.  Also,
\beq
\sum_i \big|\langle \chi,  H^i(P_n;\Q) \rangle\big| \leq p(n).
\eeq
\label{le:boundbigi}
\end{lemma}

\begin{proof}
We will need the explicit Lehrer-Solomon description of the $S_n$-action on $H^i(P_n;\Q)$.  Recall from \eqref{eq:hi} that $H^i(P_n;\C)$ decomposes as a sum over conjugacy classes $c_\mu$ with $n-i$ cycles of $\Ind_{Z(c_\mu)}^{S_n}(\xi_\mu)$, where $\xi_\mu$ is a $1$-dimensional representation.

%The character of  $\Ind_{Z_{S_n}(c_\mu)}^{S_n}(\xi_\mu)$ evaluated at a permutation $\sigma$ is bounded above in absolute value by the number of elements of $c_\mu$ fixed by conjugation by $\sigma$, i.e. the size of the intersection of $c_\mu$ with the centralizer $Z_{S_n}(\sigma)$ of $\sigma$.

%Now let $\Sigma$ be a subset of $S_n$ and let $\chi_\Sigma$ be its characteristic function.
 
By Frobenius reciprocity, the inner product $
 \langle \chi, \Ind_{Z(c_\mu)}^{S_n}(\xi_\mu) \rangle_{S_n}
$
is equal to
\beq
\langle \chi |_{Z(c_\mu)}, \xi_\mu \rangle_{Z(c_\mu)} = \frac{1}{|Z(c_\mu)|} \sum_{\sigma \in Z(c_\mu)} \chi(\sigma) \xi_\mu(\sigma),
\eeq
which is evidently bounded in absolute value by $1$.  Thus the value of $\langle \chi, H^i(P_n;\Q) \rangle$ is bounded above by the number of conjugacy classes $c_\mu$ with $n-i$ cycles.  Any such conjugacy class has at least $n-2i$ fixed points, which is to say the sum of its nontrivial cycle lengths is at most $2i$. Therefore the number of such $c_\mu$ is bounded by  $p(2i)$, the number of partitions of $2i$.  Of course, the total number of conjugacy classes $c_\mu$ is $p(n)$.
\end{proof}

Combining Lemmas~\ref{le:vanishingchi} and \ref{le:boundbigi} gives the following.

\begin{prop}[{\bf No small factors vs.\ no small cycles}]
Let $k$ be an integer.  Let $\Phi(n,k)$ be the number of monic, squarefree, degree $n$ polynomials over $\F_q$ with no prime factor of degree less than $n/k$.  Let $\pi(n,k)$ be the proportion of permutations in $S_n$ with no cycle length shorter than $n/k$.  Then
\beq
\Phi(n,k)  = \pi(n,k) q^n +  O(q^{n-\lceil n/2k \rceil } p(n))
\eeq
where $p(n)$ is the partition function and the implied constant is absolute.  In particular, holding $k$ fixed,
\beq
\lim_{n \ra \infty} (n q^{-n} \Phi(n,k)) = \lim_{n \ra \infty} n \cdot\pi(n,k). 
\eeq
\label{pr:smallirred}
\end{prop}

Implicit in the second part of the proposition is the fact that $\lim_{n \ra \infty} n\cdot \pi(n,k)$ exists. In fact, this limit is known to converge to $k\cdot\omega(k)$, where $\omega(k)$ is the {\em Buchstab function}, which approaches $e^{-\gamma}$ as $k \ra \infty$.  So one can also write
\beq
\Phi(n,k) = (k/n)\omega(k)q^n + o(q^n/n)
\eeq
as $n \ra \infty$ with $k$ fixed, as Panario and Richmond do in \cite[Theorem 3.4]{panariorichmondbenor}; however, $\pi(n,m)$ and $q^{-n} \Phi(n,k)$ converge to each other more quickly than either one does to $k\omega(k)/n$, so the formulation used here gives a better error term.

\begin{proof}
By Proposition~\ref{pr:exactcount} we know that
\beq
\Phi(n,k) = \sum_{f \in \Conf_n(\F_q)} \chi_k(f) = 
\sum_{i=0} (-1)^i q^{n-i} \langle \chi_k, H^i(P_n;\Q) \rangle.
\eeq
The contribution of $H^0(P_n)$ to this alternating sum is precisely $q^n$ times  $\langle \chi_k, 1 \rangle = \pi(n,k)$.  By Lemma~\ref{le:vanishingchi}, each $H^i(P_n)$ with $0<i<n/2k$ contributes $0$.  This leaves the values of $i$ greater than or equal to $n/2k$, for which Lemma~\ref{le:boundbigi} gives 
\beq
\sum_i |\langle \chi_k, H^i(P_n;\Q) \rangle| \leq p(n).
\eeq
This immediately gives the first claim.  The limit in the second claim then follows from the fact that $p(n)$ grows subexponentially with $n$.
\end{proof}

\bigskip

In case $k=1$, Proposition~\ref{pr:smallirred} says that the number of irreducible monic polynomials of degree $n$ is approximately $q^n/n$ with an error term at most on order of $q^{n/2}$ (note that if we set $N=q^n$, this approximation is $\frac{N}{\log N}$, just as in the usual Prime Number Theorem over $\Z$).  In fact, there is a well-known exact formula for the number of such polynomials:

\begin{equation}
\label{eq:exactirred}
\sum_{\ell|n}\frac{\mu(n/\ell)}{n}q^\ell
\end{equation}

One can reproduce the formula \eqref{eq:exactirred}, which appeared as (5) in Table A in the introduction, by computing the inner products  $\langle \chi_1, H^i(P_n;\Q) \rangle$ using the Lehrer-Solomon description, as we now sketch. Most summands of \eqref{eq:hi} will not contribute, since most centralizers $Z(c_\mu)$ do not contain an $n$-cycle. The only conjugacy classes  which do contribute are those contained in the centralizer of an $n$-cycle; since this centralizer is generated by the $n$-cycle itself, the conjugacy classes it contains are precisely the products $c_{(\ell)}$  of $\ell$ disjoint $\frac{n}{\ell}$-cycles. The summand of \eqref{eq:hi} for $c_{(\ell)}$ contributes to $H^{n-\ell}(P_n;\Q)$, and thus its contribution to \eqref{eq:exactirred} is weighted by $q^\ell$. All that remains is to verify that \[\langle \chi_1, \Ind_{Z(c_{(\ell)})}^{S_n} \xi_{(\ell)}\rangle = (-1)^\ell\frac{\mu(n/\ell)}{n}.\] This is straightforward but requires a case-by-case analysis of the specific characters $\xi_\mu$, so we do not carry out the full computation here.

The exact formula for $\Phi(n,1)$ can be used, with some care,   to reproduce Proposition~\ref{pr:smallirred}: note that both $\pi(n,k)$ and $\Phi(n,k)$ can be expressed as sums over partitions of $n$ into parts of size no smaller than $n/k$, and it follows from \eqref{eq:exactirred} that the proportion of polynomials with irreducible factor degrees $n_1, \ldots, n_r$ is very close to the proportion of permutations in $S_n$ with cycle lengths $n_1, \ldots, n_r$.

\para{Polynomials with factors of distinct degrees}
Another interesting case is the enumeration of monic squarefree polynomials in which the degrees of all irreducible factors are distinct.  Let $D_q(n)$ be the number of degree $n$ polynomials in $\F_q[T]$ with all irreducible factors of distinct degree. Then the asymptotics of $D_q(n)$ can be studied just as above, with $\chi_k$ replaced by the characteristic function of the subset $\Sigma=\Sigma_n$ of $S_n$ consisting of permutations with distinct cycle lengths.  We will be able to bound the contribution of $\langle \chi_\Sigma, H^i(P_n;\Q) \rangle$  large $i$ just as above, but in this setting  there is no vanishing statement like Lemma~\ref{le:vanishingchi}.  Still, the arguments above show the following:

\begin{prop}[{\bf Degree \boldmath$n$ polynomials with distinct irreducible factors}]  
There are real constants $a_0,a_1, a_2,\ldots$ such that, for each $q$, we have
\beq
\lim_{n \ra \infty} \frac{D_q(n)}{q^n} = a_0 + \frac{a_1}{q}  + \frac{a_2}{q^2} + \ldots
\eeq
\label{pr:distinct}
\end{prop}

\begin{proof}

We define \[a_i=(-1)^i \lim_{n\to \infty} \langle \chi_\Sigma,\,H^i(P_n;\Q)\rangle.\] To show that this limit exists requires some combinatorial argument, since $\chi_\Sigma$ is definitely not given by a character polynomial; however, once the limit is known to exist, Lemma~\ref{le:boundbigi} implies that $|a_i|\leq p(2i)$. For example, $a_0=\lim \langle \chi_\Sigma,1\rangle = \lim \frac{|\Sigma_n|}{n!}$ is the probability that a random permutation has distinct cycle lengths, which known to converge to $e^{-\gamma}$.

By Proposition~\ref{pr:exactcount} we have
\[q^{-n}D_q(n)=\sum_{i=0}^{\infty}(-1)\frac{ \langle \chi_n, H^i(P_n;\Q) \rangle }{q^i}.
\]
For each fixed $k$, the truncated sum $q^{-n} \sum_{i=0}^{k-1}(-q)^i \langle \chi_n, H^i(P_n;\Q) \rangle $
approaches $\sum_{i=0}^{k-1} a_i q^{-i}$ as $n \ra \infty$.  Moreover, the contribution of the cohomology of larger degree is
\beq
 \sum_{i={k}}^\infty (-1)^i (-q)^{-i} \langle \chi_n, H^i(P_n;\Q) \rangle,
\eeq
which is bounded in absolute value by $\sum_{i={k}}^\infty p(2i) q^{-i}$ by Lemma~\ref{le:boundbigi}. Thus 
\beq
\lim_{n \ra \infty}\  \Big|q^{-n} D_q(n) - \sum_{i=0}^\infty a_i q^{-i}\Big|\leq 2\sum_{i={k}}^\infty p(2i) q^{-i}.
\eeq
Since the quantity on the right approaches $0$ as $k$ grows, and $k$ was chosen arbitrarily, this completes the proof.
\end{proof}

This conforms with \cite[Theorem~6]{flajoletgourdonpanario}, which gives an infinite product formula for $\lim_{n \ra \infty} q^{-n} D_q(n)$ and shows that this limit converges to $e^{-\gamma}$ as $q \ra \infty$.

\para{Statistics uncorrelated with characteristic polynomials}
We have seen so far that, with respect to many natural statistics, the distribution of degrees of irreducible factors of random squarefree polynomials behave like the cycle lengths of random permutations ``up to $O(\frac{1}{q})$" ; for example, the average number of linear factors of a squarefree polynomial is $1-\frac{1}{q}+\frac{1}{q^2}+\cdots$, while the average number of fixed points of a permutation is exactly $1$.  

On the other hand, there are some statistics whose limiting asymptotics for polynomials behave {\em exactly} like the limiting asymptotics for the corresponding functions on permutations.  For example, the probability that a permutation has an even number of cycles is $1/2$, and we proved in Lemma~\ref{lem:signrep} and \eqref{eq:muaverage} that the probability that a random squarefree polynomial has an even number of prime factors is $1/2$ as well.

What distinguishes the two kinds of statistics?  The following  gives a partial answer.

\begin{definition}
\label{def:uncorrelated}
Let $(\chi_n)_{n\in \N}$ be a sequence of class functions $\chi_n$ on $S_n$ satisfying $|\chi_n(\sigma)| \leq 1$ for all $\sigma \in S_n$. We say that the sequence of class functions $\chi_n$ is \emph{uncorrelated with all character polynomials} if the limit \[x\coloneq \lim_{n\to \infty}\langle \chi_n, 1\rangle_n\]
exists, and furthermore for every character polynomial $P$ we have
\beq
 \lim_{n \ra \infty} \langle \chi_n, P\rangle_n=\lim_{n\ra \infty} \langle \chi_n,1\rangle_n \langle 1,P\rangle_n=\lim_{n\to \infty}\langle x,P\rangle_n
\eeq
This condition on $\chi_n$ can be thought of as saying that with respect to all finite moments, $\chi_n$ behaves like $x$ times the uniform distribution.
\end{definition}

\begin{prop} Assume that $\chi_n$ is uncorrelated with all character polynomials, with average value $x=\lim_{n\to \infty}\langle \chi_n,1\rangle$.
Then for every $q$, the average of $\chi_n(f)$ over all monic squarefree degree-$n$ polynomials $f(T)$ in $\F_q[T]$ approaches the same limit $x$ as $n \ra \infty$.
\label{pr:permpoly}
\end{prop}

An natural example of a sequence $\chi_n$ uncorrelated with all character polynomials is the characteristic function $\chi_n=\chi_{A_n}$ of $A_n$, in which case we of course have $x=1/2$, since $\langle \chi_{A_n},1\rangle=1/2$ at each finite limit. In this case Proposition~\ref{pr:permpoly} reproduces in the limit the fact demonstrated in \eqref{eq:muaverage}, that half of all squarefree polynomials have an even number of irreducible factors.  An elementary but slightly more involved argument shows if $S_n^{(4)}$ denotes the set of permutations whose number of cycles is divisible by 4, then  the characteristic function $\chi_n=\chi_{S_n^{(4)}}$ is uncorrelated with all character polynomials. In this case the inner products $\langle \chi_{S_n^{(4)}},1\rangle$ vary with $n$, but as $n\to \infty$ they converge to $x=1/4$.  Therefore the proportion of  squarefree polynomials whose number of irreducible factors is a multiple of $4$ approaches $1/4$ as $n \ra \infty$.

The function $X_1$, by contrast, fails to satisfy the conditions of Definition~\ref{def:uncorrelated}.  For one thing, it is not uniformly bounded, but, more importantly, it is clearly not uncorrelated with $P = X_1$ itself: we have $\langle X_1,X_1\rangle=2$ for all $n\geq 2$, which is not equal to $\langle X_1,1\rangle\langle 1,X_1\rangle=1\cdot 1=1$.  And, indeed, we have seen that the average value of $X_1(f)$ depends on $q$, though it approaches the corresponding random permutation statistic as $q \ra \infty$. Note that by Remark~\ref{rem:poisson}, $X_1$ \emph{is} uncorrelated in the limit from any character polynomial involving only $X_2,X_3,\ldots$; this shows that Definition~\ref{def:uncorrelated} really must be satisfied for \emph{all} $P$.

\section{Maximal tori in $\GL_n(\F_q)$}
\label{s:maximaltori}

Our goal in this section is to present in some depth another example of how representation stability for the cohomology of a complex variety is reflected in the combinatorial stability of associated counting problems over a finite field.  Here we will have cohomology of flag varieties on the one hand, and counting problems for maximal tori in the 
finite group $\GL_n(\F_q)$ on the other.  The results of this section are in large measure already proved in \cite{lehrer:rationaltori}; our goal here is to explain the relationship between the results and representation-stable cohomology, and to emphasize the analogy between the questions here and those about squarefree polynomials.

\subsection{Parameterizing the set of maximal tori in $\GL_n(\F_q)$}

For any variety $X$ one can define 
\[\PConf_n(X)\coloneq \{(x_1,\ldots,x_n)\in X\,|\, x_i\neq x_j\}\]
and its quotient $\Conf_n(X)\coloneq \PConf_n(X)/S_n$. Many of the results from the previous sections can be extended  in some form to this situation. However, the \'etale cohomology of $\PConf_n(X)$ will be much more complicated in general than it was for $\PConf_n=\PConf_n(\A^1)$, thanks to the contribution of $H^*_{\et}(X;\Q_\ell)$. In particular, for most varieties over $\F_q$ the action of $\Frob_q$ on $H^*_{\et}(X;\Q_\ell)$ is much more complicated than just multiplication by a power of $q$ (and in fact is quite difficult to compute, even when $X$ is 1-dimensional), so no simple formula like Theorem~\ref{th:chi} will be possible.

However, for projective space $\PP^m$ it is true that $\Frob_q$ acts on each $H^{2i}_{\et}(\PP^m;\Q_\ell)$ by $q^i$, so we could hope for a complete answer in this case. In this section we will consider a variant of $\PConf_n(\PP^m)$, where we require that points be not just distinct, but in \emph{general position}.

\begin{definition}
Let $\PP^{n-1}$ be the $(n-1)$-dimensional projective space as a scheme over $\Z$.  For any field $k$  the $k$-points $\PP^{n-1}(k)$ can be identified with the set of lines in the $n$-dimensional vector space $k^n$.
Inside the $n$-fold product $(\PP^{n-1})^n$, we define:
\[\tT_n\coloneq \big\{\,(L_1,\ldots,L_n)\,\big|\,L_1,\ldots,L_n\in \PP^{n-1}\text{ are lines in general position}\,\big\}\] For lines $L_1,\ldots,L_n\in \PP^{n-1}(k)$ to be in \emph{general position} means that the corresponding lines in $k^n$ are linearly independent, and thus give an internal direct sum $k^n=L_1\oplus\cdots\oplus L_n$. We may consider $\tT_n$ as a smooth scheme over $\Z$ (this is not obvious, but can be deduced from \cite[Proposition 9.1.1]{Fu}.) The natural action of $S_n$ on $(\PP^{n-1})^n$ by permuting the factors preserves $\tT_n$ (and in fact restricts to a free action on $\tT_n$). We define $\T_n$ to be the quotient $\tT_n/S_n$.
\end{definition}

\begin{remark}
Just as we saw in Remark~\ref{rem:quotient}, the $k$-points $\T_n(k)$ are not just the quotient of $\tT_n(k)$ by $S_n$. Instead the $k$-points $\T_n(k)$ correspond to sets $\{L_1,\ldots,L_n\}$ of lines in general position in $\overline{k}^n$ for which the \emph{set} of lines is invariant under $\Gal(\overline{k}/k)$, not each line itself. For example, the lines $L_1=\langle (1,\ i)\rangle$ and $L_2=\langle (1,\ -i)\rangle$ in $\C^2$  are in general position, and the set $\{L_1,L_2\}$ is invariant under complex conjugation, so it corresponds to a point of $\T_2(\R)$. As we now explain, such $\Gal(\kbar/k)$-invariant sets correspond naturally to \emph{maximal tori} defined over $k$.%,  just as points in $\Conf_n(k)$ corresponded to polynomials with coefficients in $k$.
\end{remark}

\para{The variety of maximal tori}
Given a line $L$ in $\overline{k}\,{}^n$, let $G_L$ be the group of automorphisms of $\overline{k}\,{}^n$ that preserve $L$; this is an algebraic subgroup of ${\GL_n}$  defined over $\overline{k}$. For example, for $L=\langle (1,\ 0)\rangle$ in $\C^2$, we have
\[G_L=\left\{\left.\begin{pmatrix}a&b\\c&d\end{pmatrix}\in \GL_2\,\right|\,c=0\right\}.\]
Given a set $\LL=\{L_1,\ldots,L_n\}$ of $n$ lines in general position in $\overline{k}\,{}^n$, let \[G_{\LL}=G_{L_1}\cap\cdots\cap G_{L_n}\] be the subgroup of $\GL_n$ preserving each line in $\LL$.

The key property is that if the set $\LL$ is preserved by $\Gal(\overline{k}/k)$ for some subfield $k\subset \overline{k}$, then the group $G_{\LL}$ will be invariant under $\Gal(\overline{k}/k)$.  By Galois descent, $G_{\LL}$ is thus defined over $k$. For example, consider the lines $L_1=\langle (1,\ i)\rangle$ and $L_2=\langle (1,\ -i)\rangle$ in $\C^2$. The individual subgroups $G_{L_1}$ and $G_{L_2}$ are not defined over $\R$; indeed we have
\begin{align*}G_{L_1}&=\left\{\left.\begin{pmatrix}a&b\\c&d\end{pmatrix}\in \GL_2\,\right|\,b+c=\ \ (a-d)i\right\}\\G_{L_2}&=\left\{\left.\begin{pmatrix}a&b\\c&d\end{pmatrix}\in \GL_2\,\right|\,b+c=-(a-d)i\right\}
\end{align*}
But their intersection $G_{\LL}$ is equal to
\[G_{\LL}=G_{L_1}\cap G_{L_2}=\left\{\left.\begin{pmatrix}a&b\\c&d\end{pmatrix}\in \GL_2\,\right|\,\begin{matrix}a=d,\\b=-c\end{matrix}\ \right\}=\left\{\begin{pmatrix}a&b\\-b&a\end{pmatrix}\in \GL_2\right\}\]
and thus is defined over $\R$.

In general, a \emph{torus} in $\GL_n$ over $k$ is an algebraic subgroup of $\GL_n$ defined over $k$ which becomes diagonalizable over $\overline{k}$. A torus is \emph{maximal} if it is not contained in any larger torus.  Each maximal torus in $\GL_n$ over $k$ becomes isomorphic to ${\G_m}^{\oplus n}$ over $\overline{k}$. The groups $G_{\LL}$ above are all maximal tori, since with respect to a basis $x_1\in L_1,\ldots,x_n\in L_n$ they consist just of diagonal matrices. Conversely, if $T$ is a maximal torus in $\GL_n$ over $k$, then its $n$ eigenvectors (which are obviously in general position) define a set $\LL_T=\{L_1,\ldots,L_n\}$ in $\overline{k}\,{}^n$. Since $T$ is defined over $k$, the property of being an eigenvector of $T$ is preserved by $\Gal(\overline{k}/k)$, so the set $\LL_T$ is preserved by $\Gal(\overline{k}/k)$. This gives the following description, analogous to the identification of $\Conf_n(k)$ as the space of squarefree polynomials in $k[T]$.

\begin{observation}
\label{ob:tori}
The $k$-points $\T_n(k)$ parametrize maximal tori over $k$ in $\GL_n$. 
\end{observation}

We recall some well-known facts about tori (see e.g.\ \cite[III.8]{Bo2}).
A torus $T$ is {\em $k$-split} if it is isomorphic over $k$ to a product of copies of $\G_m$.  %This is equivalent to the space $X(Y)_k$ of $k$-characters of $T$ spanning $k[T]$.
%In contrast, a torus $T$ is {\em $k$-anisotropic} if there is no homomorphism $T\to \G_m$ of algebraic groups.
For any torus $T$ over $k$, there exists a finite Galois extension of $k$ over which $T$ becomes split.
All maximal tori are conjugate in $\GL_n$ over $\overline{k}$, and all $k$-split maximal tori are conjugate in $\GL_n$ over $k$. A torus over $k$ is \emph{irreducible} if it is not isomorphic over $k$ to a product of tori. Every torus $T$ over $k$ factors uniquely (up to reordering) as a product of irreducible tori over $k$.

\subsection{Twisted Grothendieck--Lefschetz on ${\cal T}_n$}  Observation~\ref{ob:tori} tells us that the $\F_q$-points $\T_n(\F_q)$ parametrize the set of maximal tori $T$ defined over $F_q$ in $\GL_n$. Such a torus $T$ determines a subgroup $T(\F_q)$ of the finite group $\GL_n(\F_q)$, which is why the space $\T_n(\F_q)$ has been of interest to finite group theorists. 

Consider a maximal torus $T$ in $\GL_n$ defined over $\F_q$.
%If $T$ is irreducible, it splits over $\F_{q^n}$. In general, 
Since $T$ is defined over $\F_q$, the Frobenius map $\Frob_q$ preserves $T$ and thus permutes the eigenvectors $\LL_T=\{L_1,\ldots,L_n\}$. This defines a permutation $\sigma_T\in S_n$, defined up to conjugacy, and the cycle type of $\sigma_T$ corresponds to the factorization of $T$ into irreducible factors. For example, if $T$ is $\F_q$-split then $\sigma_T=\id$; if $T$ splits as a product of an $\F_q$-split torus with two 2-dimensional irreducible tori %(which split over $\F_{q^2}$)
and one 3-dimensional irreducible torus, %(which splits over $\F_{q^3}$),
then $\sigma_T=(1\ 2)(3\ 4)(5\ 6\ 7)$.

We can count the number $\left|\T_n(\F_q)\right|$ of maximal $\F_q$-tori in $\GL_n(\F_q)$ via the Grothendieck--Lefschetz formula, relating this to the cohomology $H^*_{\et}(\T_n;\Q_\ell)$. Moreover just as we did in Section~\ref{sec:pointcountingandstabilization}, we can count more interesting statistics for maximal $\F_q$-tori via the action of $S_n$ on the cohomology $H^*_{\et}(\tT_n;\Q_\ell)$ of the cover $\tT_n$. To understand $H^*_{\et}(\tT_n;\Q_\ell)$, we will relate it to the 
singular cohomology of $\tT_n(\C)$, which is well-understood.

Given a class function $\chi$ on $S_n$ and a maximal torus $T\in \T_n(\F_q)$, we write $\chi(T)$ for $\chi(\sigma_T)$. 
The \emph{co-invariant algebra} $R[x_1,\ldots,x_n]$ is the quotient   \[R[x_1,\ldots,x_n]\coloneq \Q[x_1,\ldots,x_n]/I_n,\] where  $I_n$ be the ideal of
  $\Q[x_1,\ldots,x_n]$ generated by all symmetric polynomials with
  zero constant term. Since $I_n$ is a homogeneous ideal, the natural grading on $\Q[x_1,\ldots,x_n]$ descends to a  grading
  \[R[x_1,\ldots,x_n]=\bigoplus_iR_i[x_1,\ldots ,x_n].\]
  
  The main theorem of this section is the following analogue of Theorem~\ref{th:chi}.   This result was first proved by Lehrer as Corollary $1.10'$ in \cite{lehrer:rationaltori}.  Unlike the argument we give here, Lehrer's proof did not invoke the Grothendieck--Lefschetz theorem.  
\begin{theorem}
\label{th:chitori}
Let $\chi$ be any class function on $S_n$. Then the sum of $\chi(T)$ over all maximal tori $T\in \T_n(\F_q)$ is equal to
\begin{equation}
\label{eq:chitori}
\sum_{T\in \T_n(\F_q)}\chi(T)=\sum q^{n^2-n-i}\langle \chi,R_i[x_1,\ldots,x_n]\rangle
\end{equation}
\end{theorem}

Our proof of Theorem~\ref{th:chitori} depends on a number of lemmas connecting $\T_n(\F_q)$ with the flag variety $\FF_n(\C)=\GL_n(\C)/B$.

\begin{lemma}[{\bf Borel}]
\label{lem:ttnBorel}
The cohomology $H^{2i}(\tT_n(\C);\Q)$ is concentrated in even degrees, and there is an $S_n$-equivariant isomorphism \[H^{2i}(\tT_n(\C);\Q)\approx R_i[x_1,\ldots,x_n].\]
\end{lemma}
\begin{proof}
For each $i=1,\ldots,n$ there is a natural line bundle $\cL_i$ over $\tT_n$ whose fiber over $\LL=(L_1,\ldots,L_n)$ is $L_i$. Specializing to $\tT_n(\C)$ this yields a complex line bundle $\cL_i\to \tT_n(\C)$, whose first Chern class is an element $c_1(\cL_i)\in H^2(\tT_n(\C);\Q)$.

Sending $x_i\mapsto c_1(\cL_i)$ determines a map $\Q[x_1,\ldots,x_n]\mapsto H^*(\tT_n(\C);\Q)$. Borel  \cite{Bo1}  proved that this map is surjective with kernel $I_n$. In other words, it gives an isomorphism $R[x_1,\ldots,x_n]\approx H^*(\tT_n(\C);\Q)$, which clearly takes $R_i[x_1,\ldots,x_n]$ to $H^{2i}(\tT_n(\C);\Q)$. See  \cite{Bo1} or \cite[Proposition
  10.3]{Fu} for a complete proof.
  \end{proof}

\begin{lemma}[{\bf action of Frobenius}]
\label{lem:ttncomparison}
There is an $S_n$-equivariant isomorphism \[H^i_{\et}(\tT_n{}_{/\Fqbar};\Q_\ell)\approx H^i(\tT_n(\C);\Q_\ell).\]  The Frobenius morphism $\Frob_q$ acts on $H^{2i}_{\et}(\tT_n{}_{/\Fqbar};\Q_\ell)$ by multiplication by $q^i$.
\end{lemma}

\begin{proof}
There is always a comparison map $c_{\tT}\colon H^i_{\et}(\tT_n{}_{/\Fqbar};\Q_\ell)\to H^i(\tT_n(\C);\Q_\ell)$, which is $S_n$-equivariant because since the action of $S_n$ on $\tT_n$ is algebraic.  (We are using here that $\tT_n$ has a suitable model over $\Spec \Z_p$.)   It would be immediate that $c_{\tT}$ is an isomorphism if $\tT_n$ were smooth and projective, but it is not, so we use the following argument taken from Srinivasan \cite[Th 5.13]{Sr}.  Let $\FF_n$ be the flag variety whose $k$-points $\FF_n(k)$ are in bijection with complete flags $(0\lneq V_1\lneq \cdots\lneq V_n=k^n)$; this is a smooth projective variety. There is a natural map $\pi\colon \tT_n\to \FF_n$ defined by
\[\pi\colon \tT_n\to \FF_n\qquad(L_1,\ldots,L_n)\mapsto (L_1\lneq L_1\oplus L_2 \cdots\lneq L_1\oplus \cdots \oplus L_n)\]
The fibers of $\pi$ are isomorphic to $\A^{\binom{n}{2}}$, so the induced map $\pi^*\colon H^*_{\et}({\FF_n}_{/\Fqbar};\Q_\ell)\to H^*_{\et}(\tT_n{}_{/\Fqbar};\Q_\ell)$ is a Galois-equivariant isomorphism. As for singular cohomology, the fibers of the induced map $\tT_n(\C)\to \FF_n(\C)$ are  isomorphic to $\C^{\binom{n}{2}}$ and thus contractible, so $\pi^*\colon H^*(\FF_n(\C);\Q_\ell)\to H^*(\tT_n(\C);\Q_\ell)$ is also an isomorphism. Since $\FF_n$ is smooth and projective (even over $\Spec \Z_p$) the comparison map $c_{\FF}\colon H^i_{\et}(\FF_n{}_{/\Fqbar};\Q_\ell)\to H^i(\FF_n(\C);\Q_\ell)$ is an isomorphism. Therefore we have:
\[\xymatrix{
H^i_{\et}(\tT_n{}_{/\Fqbar};\Q_\ell)\ar^{c}[r] &H^i_{\et}(\tT_n{}_{/\Fqbar};\Q_\ell)\\
H^i_{\et}(\FF_n{}_{/\Fqbar};\Q_\ell)\ar^{\pi^*}_{\approx}[u] \ar_{c_{\FF}}^{\approx}[r]
& H^i(\FF_n(\C);\Q_\ell)\ar_{\pi^*}^{\approx}[u] 
}\]
%\beq
%H^i_{\et}({\FF_n}_{/ \Fqbar};\Q_\ell) \approx H^i_{\et}({\FF_n}_{/ \C};\Q_\ell)\approx H^i(\FF_n(\C);\Q_\ell).
%\eeq
This demonstrates that $c_{\tT}$ is an isomorphism of vector spaces. (The action of $S_n$ on $\tT_n$ does not descend to $\FF_n$, so it is important that we already know $c_{\tT}$ to be $S_n$-equivariant.) Finally,  the existence of the Schubert cell decomposition of $\FF_n$ implies that $\Frob_q$ acts on $H^{2i}_{\et}(\FF_n{}_{/\Fqbar};\Q_\ell)$ by multiplication by $q^i$ (see \cite[Th 5.13]{Sr}); since $\pi^*$ is Galois-equivariant, the same claim for $H^{2i}_{\et}(\tT_n{}_{/\Fqbar};\Q_\ell)$ follows.
\end{proof}

%%for TC 9-3: add warning about defining TT_n over Z, move discussion of "defined over k" earlier

\begin{proof}[Proof of Theorem~\ref{th:chitori}]
Just as in the proof of Theorem~\ref{th:chi}, it suffices to prove \eqref{eq:chitori} for the character $\chi_V$ of an irreducible $S_n$-reprentation $V$. Let $\V$ denote the corresponding local system on $\T_n$ which becomes trivial when pulled back along the Galois $S_n$-cover $\tT_n\to \T_n$. The fixed points of the Frobenius morphism $\Frob_q\colon \T_n(\Fqbar)\to \T_n(\Fqbar)$ are precisely the $\F_q$-points $\T_n(\F_q)$. Moreover, each stalk $\V_T$ is isomorphic to $V$, and for $T\in \T_n(\F_q)$ the action of $\Frob_q$ on $\V_T$ is by the permutation $\sigma_T$, so its trace is \[\tr\big(\Frob_q\colon \V_T\to \V_T)=\chi_V(\sigma_T)=\chi_V(T).\]
%Therefore  the Grothendieck--Lefschetz formula gives
%\[\sum_{T\in \T_n(\F_q)} \chi_V(T)=\sum_{T\in \T_n(\F_q)} \tr\big(\Frob_q : \V_T\big)= \sum_j (-1)^j \tr\big( \Frob_q : H^j_{c}(\T_n; \V)\big)\]
Since $\tT_n$ is an open subvariety of $(\PP^{n-1})^n$, it has dimension $n(n-1)=n^2-n$, as does its quotient $\T_n$. Since $\T_n$ is smooth of dimension $n^2 - n$, we can use Poincare duality as in \eqref{eq:GLcorrect} to write the Grothendieck--Lefschetz formula as \begin{equation}
\sum_{T\in {\cal T}_n(\F_q)} \chi_V(T) = q^{n^2 - n} \sum_j(-1)^j\tr\big(\Frob_q : H^j_{\et}(\T_n;V)^\vee \big).
\label{eq:GLtorus15}
\end{equation}
As in the proof of Theorem~\ref{th:chi}, transfer gives an isomorphism
\beq
H^j_{\et}(\T_n;\V) %\approx \big(H^j(\tT_n; \Q_\ell) \tensor V\big)^{S_n}
\approx H^j_{\et}(\tT_n;\Q_\ell) \tensor_{\Q[S_n]} V.
\eeq
By Lemma~\ref{lem:ttnBorel} and Lemma~\ref{lem:ttncomparison}, this is only nonzero for $j=2i$; in this case $H^{2i}_{\et}(\T_n;\V)$  is acted on by $\Frob_q$ by multiplication by $q^i$, and its dimension is \[\dim H^{2i}_{\et}(\T_n;\V)=\langle \chi_V,H^{2i}_{\et}(\tT_n;\Q_\ell)\rangle
=\langle \chi_V,R_i[x_1,\ldots,x_n]\rangle.\] Therefore \eqref{eq:GLtorus15} becomes
\[\sum_{T\in {\cal T}_n(\F_q)} \chi_V(T) = q^{n^2 - n} \sum_i q^{-i} \langle \chi_V,R_i[x_1,\ldots,x_n]\rangle\]
as claimed. 
\qedhere

\end{proof}

%WORKING HERE: add version of Prop whatever, also add Prop whatever to intro
We also have the following analogue of Proposition~\ref{pr:purebraidlimit}. It was proved in \cite[Theorem 3.4]{CEF} that the coinvariant algebras $R[x_1,\ldots,x_n]$ can be bundled together into a graded FI-module $R=\bigoplus_i R_i$ such that that each graded piece $R_i$ is finitely generated as an FI-module. Therefore the results of \cite{CEF} imply that for any fixed character polynomial $P$ and any $i\geq 0$, the inner product $\langle P,R_i[x_1,\ldots,x_n]\rangle$ is eventually independent of $n$. We denote this stable multiplicity by
\[\langle P,R_i\rangle\coloneq \lim_{n\to \infty}\langle P, R_i[x_1,\ldots,x_n]\rangle.\]

\begin{theorem}  For any character polynomial $P$ and any prime power $q$, we have:
\beq
\lim_{n \ra \infty} q^{-(n^2-n)} \sum_{T \in \T_n(\F_q)} P(T) = \sum_{i=0}^\infty (-1)^i\frac{\langle P,R_i\rangle}{q^{i}}
\eeq
In particular, both the limit on the left and the series on the right converge.
\label{th:torilimit}
\end{theorem}
This theorem is proved in exactly the same way as Theorem~\ref{th:limit}, using Theorem~\ref{th:chitori} in place of Theorem~\ref{th:chi}.
%In fact, it is not hard to show, exactly as in Proposition~\ref{pr:purebraidlimit}, that the expected value of $P(\sigma_T)$ for any character polynomial $P$ will approach a limit in $n$:
%\beq
%\lim_{n \ra \infty} q^{-(n^2-n)} \sum_{T\in {\cal T}_n(\F_q)} P(\sigma_T) = \sum_i\frac{b(P,2i)}{q^{i}}
%\eeq
%where the numerator is the stable value $b(P,2i)\coloneq \lim_{n \ra \infty} \langle P, H^{2i}(\widetilde{\cal F}_n;\C) \rangle$.
One does need the analogue of the convergence condition in Definition~\ref{def:convergent}. But the results of \cite{CEF} imply that there is a constant $\alpha$ such that $\langle P, R_i[x_1,\ldots,x_n] \rangle =\langle P, R_i\rangle $ for all $i < \alpha n$.  Since $R[x_1,\ldots,x_n]$ is a quotient of $\C[x_1,\ldots,x_n]$, it is enough to observe that the degree-$i$ piece of $\C[x_1,\ldots,x_n]^{S_{n-a}}$ grows subexponentially. On the other hand, since the total algebra $R[x_1,\ldots,x_n]$ is isomorphic to $\Q[S_n]$  \cite{chevalley}, the overall contribution of all $R_i[x_1,\ldots,x_n]$ for $i \geq \alpha n$  is bounded by $q^{-\alpha n} \langle P, \Q[S_n]\rangle = n\cdot q^{-\alpha n} P(n,0,0,\ldots)$ which goes to $0$ as $n \ra \infty$.

\subsection{Specific statistics for maximal tori in $\GL_n(\F_q)$}
The twisted Grothendieck--Lefschetz formula also lets us compute certain statistics of maximal tori in $\GL_n(\F_q)$ explicitly for fixed $n$, not just in the limit as $n\to \infty$. The results of this section can be obtained by other methods, but we include them as examples of how the Grothendieck--Lefschetz formula may be applied.

\para{Explicit formula for $\langle \chi_V,R_i\rangle$}
To make use of Theorem~\ref{th:chitori}, we need to be able to calculate the multiplicities $\langle \chi_V,R_i[x_1,\ldots,x_n]\rangle$. 
Chevalley~\cite{chevalley} proved that when the grading is ignored,
  $R[x_1,\ldots,x_n]$ is isomorphic
  to the regular representation $\C S_n$.  Therefore each irreducible $S_n$-representation $V_\lambda$  occurs in $R[x_1,\ldots,x_n]$ with multiplicity $\dim V_\lambda$, and we would like to know how these $\dim V_\lambda$ copies are distributed among the 
  \[R_i[x_1,\ldots,x_n]\approx H^{2i}(\tT_n;\Q_\ell)\approx H^{2i}(\cF_n(\C);\Q_\ell).\]  
  The answer is given by the following theorem of Stanley, Lusztig, and
  Kraskiewicz--Weyman; see, e.g.\ \cite{Re}, Theorem 8.8.

The irreducible $S_n$-representations $V_\lambda$ are in bijections with partitions $\lambda$ of $n$. A \emph{standard tableau of shape $\lambda$} is a
  bijective labeling of the boxes of the Young diagram for $\lambda$
  by the numbers $1,\ldots,n$ with the property that in each row and
  in each column the labels are increasing.  The
  \emph{descent set} of such a tableau is the set of numbers $i\in \{1,\ldots,n\}$ for which the box labeled
  $i+1$ is in a lower row than the box labeled $i$. The \emph{major
    index} of a tableau is the sum of the numbers in its descent set.
    
    The following theorem   is sometimes stated with the assumption $i\leq \binom{n}{2}$, but the same formula holds in general, as can been seen by applying Poincar\'{e} duality to $\cF_n$.

  \begin{theorem}[{\bf \cite[Theorem 8.8]{Re}}]
    \label{theorem:kw}
For any $i$, any $n$,  and any $\lambda\vdash n$, the
    multiplicity $\langle V_\lambda,R_i[x_1,\ldots,x_n]\rangle$ of $V_\lambda$ in $R_i[x_1,\ldots,x_n]$ is the
    number of standard tableaux of shape $\lambda$ with major index equal to $i$.
  \end{theorem}

% Fix a representation $V$ of $S_n$.  
%Combining \eqref{eq:GLtorus3} with the isomorphism 
%$H^i(\tT_n(\C);\C)\approx H^i({\cal F}_n;\C)$, we obtain the formula: 
%
%\begin{equation}
%\label{eq:GLtorus4}
%\sum_{A\in {\cal T}_n(\F_q)}\chi_V(c(A))
%=\sum_i(-1)^i q^{N-i/2} \langle V, H^i(\widetilde{\cal F}_n;\C)\rangle
%\end{equation}
%where, as above, $N\coloneq n^2-n$.  
%
%Plugging in for $H^i(\widetilde{\cal F}_n;\C)$ by applying Theorem~\ref{theorem:Borel2} gives: 
%
%\begin{equation}
%\label{eq:GLtorus5}
%\sum_{A\in {\cal T}_n(\F_q)}\chi_V(c(A))
%=\sum_iq^{N-i} \langle V,R_i[x_1,\ldots ,x_n]\rangle
%\end{equation}
%
%Equation~\ref{eq:GLtorus5} agrees with Corollary 1.10' in \cite{lehrer:rationaltori}, which is arrived at without direct use of the cohomology of the flag variety.
%
%The power of formula \eqref{eq:GLtorus5} is that we can compute the right hand side using Theorem~\ref{theorem:kw}.
We begin with the easiest case, which is the case when $V_\lambda$ is the trivial representation. The following theorem was first proved by Steinberg, and has been reproved many times; a proof using the Grothendieck--Lefschetz formula was given by Srinivasan \cite[Theorem 5.13]{Sr}, and a closely related proof is given in \cite[Corollary 1.11]{lehrer:rationaltori}.

\begin{theorem}[{\bf Steinberg}]
\label{theorem:Steinberg}
For any prime power $q$, there are $q^{n^2-n}$ maximal tori in 
$\GL_n(\F_q)$.
\end{theorem}

\begin{proof}
A maximal torus $T\in \T_n(\F_q)$ is defined by the subgroup $T(\F_q)$ of $\GL_n(\F_q)$, so the number in question is $\left|\T_n(\F_q)\right|$.    
Let $V_{(n)}=\Q$ be the trivial representation of $S_n$, so that the character $\chi_\Q$ is just the constant function 1. This corresponds to the partition $\lambda=(n)$ whose Young diagram is just $n$ boxes in a single row.  The only standard tableau of this shape is $\young(12\cdots n)$, which 
has major index $0$.  Thus  $\langle \Q,R_i[x_1,\ldots ,x_n]\rangle =0$ 
except for $i=0$, when it equals $1$. (In retrospect this is obvious, since $R[x_1,\ldots,x_n]$ is defined by killing all $S_n$-invariant polynomials.) This tells us that none of the \'etale cohomology beyond $H^0_{\et}$ contributes to $\left|\T_n(\F_q)\right|$. Theorem~\ref{th:chitori} therefore  gives:
\[\left|\T_n(\F_q)\right|=\sum_{T\in {\cal T}_n(\F_q)}1=
\sum_i q^{n^2-n-i} \langle \Q,R_i[x_1,\ldots ,x_n]\rangle=q^{n^2-n}+0+\cdots+0\qedhere
\]
%\[
%\begin{array}{ll}
%\#{\cal T}_n(\F_q)&= \sum_{A\in {\cal T}_n(\F_q)}1\\
%&\\
%&=\sum_{A\in {\cal T}_n(\F_q)}\chi_{V(0)}(c(A))\\
%&\\
%&=\sum_iq^{n^2-n-i} \langle V(0),R_i[x_1,\ldots ,x_n]\rangle\\
%&\\
%&=q^{n^2-n}\cdot 1 + 0+\cdots +0
%\end{array}
%\]
%as desired.
\end{proof}

Applying the Grothendieck--Lefschetz formula with nontrivial coefficients gives us a more 
detailed picture of the typical maximal torus in $\GL_n(\F_q)$.  For example, we have the following, which verifies (2) from Table A in the introduction; we emphasize that we abuse notation by writing ``number of eigenvectors'' for the number of lines in $\PP^{n-1}(\F_q)$ fixed by $T$, which is always between 0 and $n$.

\begin{theorem}[{\bf Expected number of eigenvalues}]
\label{theorem:evals1}
The expected number of eigenvectors in ${\F_q}^n$ of a random maximal torus in $\GL_n(\F_q)$ equals $1+\frac{1}{q}+\frac{1}{q^2}+\cdots +\frac{1}{q^{n-1}}$.
\end{theorem}

\begin{proof}
Let $\Q^n$ be the  permutation representation of $S_n$, whose character $\chi_{\Q^n}(\sigma)$ is the number of fixed points of $\sigma$. We saw above that for a torus $T$ over $\F_q$, the fixed points of $\sigma_T$ correspond to 1-dimensional subtori; when $T$ is maximal these are in bijection with lines in ${\F_q}^n$ fixed by $T$.  We are trying to compute the ratio between total number of such eigenvectors for all maximal tori, which is $\sum_{T\in \T_n(\F_q)}\chi_V(T)$, divided by the total number of maximal tori, which by Theorem~\ref{theorem:Steinberg} equals $q^{n^2-n}$.

The permutation representation $\Q^n$ is not irreducible, so to apply Theorem~\ref{theorem:kw} we decompose it as the sum $V_{(n)}\oplus V_{(n-1,1)}$ of the trivial 1-dimensional representation $V_{(n)}$ and the $(n-1)$-dimensional representation $V_{(n-1,1)}$. 
A Young diagram of shape $\lambda=(n-1,1)$ has $n-1$ boxes in the top row, and one box in the second row. There are precisely $n-1$ standard tableau $Y_1,\ldots ,Y_{n-1}$ with this shape, with $Y_i$ being the unique standard tableau of this shape with $i+1$ in the second row.  The major index of $Y_i$ is clearly $i$, since $i$ is the only descent in $Y_i$.    Theorem~\ref{theorem:kw} thus implies that $\langle V_{(n-1,1)},R_i[x_1,\ldots ,x_n]\rangle=1$ for each $1\leq i\leq n-1$, and is $0$ otherwise.  Since we found in the proof of  
Theorem~\ref{theorem:Steinberg} that  $\langle V_{(n)},R_i[x_1,\ldots ,x_n]\rangle=1$ for $i=0$ and $=0$ for all $i>0$, we conclude that:
\[\langle \Q^n,R_i[x_1,\ldots,x_n]\rangle = \begin{cases} 1&\text{ for }0\leq i<n\\0&\text{ for }n\leq i\end{cases}\]
Theorem~\ref{th:chitori} thus gives, as claimed:
\[q^{-(n^2-n)}\sum_{T\in \T_n(\F_q)}\chi_{\Q^n}(T)= \sum_i q^{-i} \langle \Q^n,R_i[x_1,\ldots,x_n]\rangle=1+q^{-1}+\cdots+q^{-(n-1)}\qedhere\]
%\[
%\begin{array}{ll}
%\sum_{T\in \T_n(\F_q)}\chi_{V}(c(A))&=\sum_{A\in {\cal T}_n(\F_q)}\chi_{V(0)}(c(A))+
%\sum_{A\in {\cal T}_n(\F_q)}\chi_{V(1)}(c(A))\\
%&\\
%&=q^{n^2-n}+\sum_iq^{n^2-n-i} \langle V(1),R_i[x_1,\ldots ,x_n]\rangle\\
%&\\
%&=q^{n^2-n}+q^{n^2-n}[q^{-1}+q^{-2}+\cdots +q^{-(n-1)}]
%\end{array}
%\]
%
%This number is the total number of eigenvalues defined over $\F_q$ of all maximal tori in $\GL_n(\F_q)$.  The expected number is this number divided by the number of maximal tori in $\GL_n(\F_q)$, which is $q^{n^2-n}$ by Theorem~\ref{theorem:Steinberg}.  This gives the theorem.
\end{proof}

\begin{remark}  Once again, we emphasize that some of these counting statements are also accessible by more elementary means, once Steinberg's theorem is given.  For instance, the expected number of $\F_q$-eigenvectors of a random torus is the expected number of splittings of $V$ into a direct sum $V_0 \oplus W$ preserved by the torus such that $\dim V_0 = 1$.  This number can be broken up as a sum over such splittings. For each such splitting of $V\approx V_0\oplus W$ with $\dim V_0=1$, the number of tori compatible with the splitting is just the number of maximal tori in $GL(W)$, which by Steinberg's theorem is $q^{n^2-3n+2}$.  The number of such splittings is $q^{n-1}\frac{q^n-1}{q-1}$.  So the total number of compatible pairs of a splitting and a compatible torus is \beq
q^{(n^2-2n})\frac{q^n -1}{1-1/q}.
\eeq
Dividing by the total number of tori $q^{n^2-n}$ one finds that the mean number of splittings per torus is
\beq
1 + \frac{1}{q} + \frac{1}{q^2} + \ldots + \frac{1}{q^{n-1}}
\eeq
as shown in Theorem~\ref{theorem:evals1}.
\end{remark}

\medskip

We can also use twisted Grothendieck--Lefschetz, as we did for squarefree polynomials in \S\ref{s:polynomials}, to compute the expected difference between the number of reducible 2-dimensional factors of a random torus and the number of irreducible 2-dimensional factors, which verifies (3) from Table A in the introduction.

\begin{theorem}[{\bf Reducible versus irreducible $2$-tori}]
Fix a prime power $q$. Given a torus $T\in \T_n(\F_q)$, let ${\cal R}_n(T)$ (resp.\ ${\cal I}_n(T)$) denote the number of reducible (resp.\ irreducible) $2$-dimensional subtori of $T$ over $\F_q$.   Then the expected value of the function 
${\cal R}_n-{\cal I}_n$ over all maximal tori of $\GL_n(\F_q)$ approaches 
\[\frac{1}{q}+\frac{1}{q^2}+\frac{2}{q^3}+\frac{2}{q^4}
+\frac{3}{q^5}+\frac{3}{q^6}+\frac{4}{q^7}+\frac{4}{q^8}+\cdots\]
as $n\to \infty$.
\end{theorem}

\begin{proof}
The desired statistic is given by the character $\chi_V(T)$ with $V=\bwedge^2 \Q^n$, so 
we need to compute $\langle \bwedge^2 \Q^n,R_i[x_1,\ldots,x_n]\rangle$. This representation decomposes into irreducibles  as $\bwedge^2\Q^n=V_{(n-1,1)}\oplus V_{(n-2,1,1)}$. We computed in the proof of Theorem~\ref{theorem:evals1} that
$\langle V_{(n-1,1)},R_i[x_1,\ldots ,x_n]\rangle =1$ for each $1\leq i\leq n-1$, and is $0$ otherwise.   

To compute $\langle V_{(n-2,1,1)},R_i[x_1,\ldots ,x_n]\rangle$, we again apply Theorem~\ref{theorem:kw}. %We consider Young diagrams with $n-2$ boxes in the first row, one box in the second, and one box in the third. 
The possible standard tableau of  shape $(n-2,1,1)$ are 
$Y_{st}$ for each $1\leq s<t< n$, where $Y_{st}$ is the tableau:
\[\young(12{\cdots}n,{\sss},{\ttt})\]
with $\sss\coloneq s+1$ and $\ttt\coloneq t+1$ missing from the first row. The descents of $Y_{st}$ are $s$ and $t$, so $Y_{st}$ has major index $s+t$. Theorem~\ref{theorem:kw} then implies that $\langle V_{(n-2,1,1)},R_i[x_1,\ldots ,x_n]\rangle$  is the cardinality of the set $\{(s,t)\,|\, 1\leq s<t< n, s+t=i\}$. Once $n\geq i$ the condition $t<n$ is irrelevant, and this cardinality is just $\left\lfloor\frac{i-1}{2}\right\rfloor$. Putting these computations together, we conclude that for $1\leq i\leq n$ we have $\langle \bwedge^2 \Q^n,R_i[x_1,\ldots,x_n]\rangle=1+\left\lfloor\frac{i-1}{2}\right\rfloor=\left\lfloor\frac{i+1}{2}\right\rfloor$ (and for larger $i$ the cardinality is $<\left\lfloor\frac{i+1}{2}\right\rfloor$, so there is no issue about convergence as in Definition~\ref{def:convergent}).
Therefore Theorem~\ref{th:chitori} gives:
\begin{align*}
q^{-(n^2-n)}\sum_{T\in \T_n(\F_q)}\chi_V(T)&=\sum_i q^{-i}\langle \bwedge^2 \Q^n,R_i[x_1,\ldots,x_n]\rangle\\
&=q^{-1}+q^{-2}+2q^{-3}+2q^{-4}+3q^{-5}+3q^{-6}+\cdots + O(q^{-n})
\end{align*}
Letting $n\to \infty$, we obtain the claimed result.
%
%
%\[
%\begin{array}{ll}
%\sum_{A\in {\cal T}_n(\F_q)}\chi_{V}(c(A))&=\sum_{A\in {\cal T}_n(\F_q)}\chi_{V(1)}(c(A))+
%\sum_{A\in {\cal T}_n(\F_q)}\chi_{V(1,1)}(c(A))\\
%&\\
%&=\sum_iq^{n^2-n-i} \langle V(1),R_i[x_1,\ldots ,x_n]\rangle+\sum_iq^{n^2-n-i} \langle V(1,1),R_i[x_1,\ldots ,x_n]\rangle\\
%&\\
%&=q^{n^2-n}[q^{-1}+q^{-2}+\cdots +q^{-(n-1)}] +q^{n^2-n}[q^{-3}+q^{-4}+2q^{-5}+2q^{-6}+\cdots]\\
%&\\
%&=q^{n^2-n}[q^{-1}+q^{-2}+2q^{-3}+2q^{-4}+3q^{-5}+3q^{-6}+\cdots ]
%\end{array}
%\]
%
%To obtain the expected excess of reducible tori over irreducible tori we must divide by the total number of maximal tori in $\GL_n(\F_q)$, which equals $q^{n^2-n}$ by Theorem~\ref{theorem:Steinberg} above. This gives the desired answer.
\end{proof}

\para{Other statistics} As  in Section~\ref{s:polynomials}, we need not limit ourselves to character polynomials.  For instance, we can consider the average value of the sign of $\sigma_T$; by contrast with polynomials, where this average was exactly $0$, we see a bias in favor of even permutations. The following proposition, which verifies (4) from Table A in the introduction, can be deduced from a formula of Srinivasan \cite[Corollary 7.6.7]{carter}; it is also proved directly by Lehrer in \cite[Corollary~1.12]{lehrer:rationaltori}.

\begin{proposition}[{\bf Parity bias for number of irreducible factors}]
The number of irreducible factors in a maximal torus in $\GL_n(\F_q)$ is more likely to be $\equiv n\bmod{2}$ than $\not\equiv n\bmod{2}$, with bias exactly $q^{\binom{n}{2}}$ (the square-root of the number of maximal tori $q^{n^2-n}$).
\end{proposition}

\begin{proof}
As in \eqref{eq:Moebius}, we have $\epsilon(\sigma_T)=(-1)^n \mu(T)$, where $\mu(T)$ is $1$ or $-1$ depending on whether $T$ has an even or odd number of factors.
Here $\epsilon$ is the sign representation  of $S_n$, which corresponds to the partition $n=(1,\ldots,1)$.  The only standard tableau of this shape has $i$ in column $i$. Therefore its descent set is $\{1,\ldots,n-1\}$, and its major index is $1+2+\cdots +(n-1)=\binom{n}{2}$.  Theorem~\ref{th:chitori} therefore gives as claimed
\[\sum_{T\in \T_n(\F_q)}\chi_\epsilon(T)=\sum_i q^{n^2-n-i} \langle\epsilon,R_i[x_1,\ldots ,x_n]\rangle=q^{n^2-n-\binom{n}{2}}=q^{\binom{n}{2}}.\qedhere
\]
%\[
%\begin{array}{ll}
%\sum_{A\in {\cal T}_n(\F_q)}\chi_\epsilon(c(A))&
%=\sum_iq^{n^2-n-i} \langle \epsilon,R_i[x_1,\ldots ,x_n]\rangle \\
%&=q^{(n^2-n)-n(n-1)/2}\\
%&=q^{n(n-1)/2}
%\end{array}
%\]
%as claimed.
%
%

\end{proof}

Finally, we discuss various versions of the ``Prime Number Theorem'' by analogy with the corresponding discussion in \S\ref{s:polynomials}.  The following proposition, which verifies (5) from Table A in the introduction, can also be proved directly (see \cite[Lemma 1.4]{lehrer:rationaltori}).

%\begin{prop}  Suppose that for each $n$ a class function $\chi_n$ on $S_n$ is given so that $|\chi_n(\sigma)| \leq 1$ for all $n$ and all $\sigma \in S_n$, and such that for each character polynomial the inner product $\langle P, \chi_n \rangle$ approaches a limit as $n \ra \infty$.  For each $i$, define
%\beq
%a(i) = \lim_{n \ra \infty} \langle P_i, \chi_n \rangle
%\eeq
%where $P_i$ is the unique character polynomial which gives the character of the action of $S_n$ on $R_i[x_1,\ldots ,x_n]$.
%
%Then
%\beq
%\lim_{n \ra \infty} q^{-(n^2-n)} \sum_{A\in {\cal T}_n(\F_q)}\chi_n(c(A)) = \sum_{i=0}^\infty a(i)q^{-i}.
%\eeq
%\end{prop}

\begin{prop}[{\bf Prime Number Theorem for maximal tori}]
\label{pr:PNTtori}
Let $\pi(q,n)$ be the number of irreducible maximal tori in $\GL_n(\F_q)$.  Then
\beq
\pi(q,n) = \frac{q^{\binom{n}{2}}}{n}(q-1)(q^2-1)\cdots(q^{n-1}-1)
\eeq
\end{prop}
We first need the following identity of characters. Let $\chi_1\colon S_n\to \{0,1\}$ be the class function from Lemma~\ref{le:vanishingchi}, taking the value 1 on $n$-cycles and 0 on all other elements. 
\begin{lemma}
\label{le:chi1}
\[\chi_1=\frac{1}{n}\cdot \sum_{k=0}^{n-1}(-1)^k\chi_{V_k},\qquad\quad \text{where }V_k=\bwedge^k\big(\Q^n/\Q)\]
\end{lemma}
\noindent This identity is well-known to representation theorists, but we will give a topological proof.
\begin{proof}
  Consider the torus $\Torus^n=\R^n/\Z^n$, with its 1--dimensional
  subtorus $\Delta=\{(t,\ldots,t)\}$. The action of $S_n$ on $\Torus^n$ by
  permuting the coordinates descends to an action on the
  $(n-1)$--dimensional torus $\Torus^n/\Delta$. This action has the 
  key property that if $\sigma\in S_n$ decomposes into $i$
  cycles, then the fixed set $\Fix_{\Torus^n/\Delta}(\sigma)$ is the union
  of finitely many $(i-1)$--dimensional tori. We can describe these fixed sets explicitly as follows.

  It is trivial that $\Fix_{\Torus^n}(\sigma)$ is an $i$--dimensional
  torus, consisting of those vectors whose coordinates are constant on
  each subset determined by a cycle. For example, if $\sigma=(1\ 2\
  3)(4\ 5\ 6)$ then $\Fix_{\Torus^n}(\sigma)$ is the subtorus 
  $\{(x,x,x,y,y,y)\}$. The fixed set $\Fix_{\Torus^n}(\sigma)$ always
  descends to an $(i-1)$--dimensional torus contained in
  $\Fix_{\Torus^n/\Delta}(\sigma)$. This need not exhaust
  $\Fix_{\Torus^n/\Delta}(\sigma)$, though. For example, the
  2--dimensional torus \[{\textstyle\{(x,x+\frac{1}{3},x+\frac{2}{3},
  y,y+\frac{1}{3},y+\frac{2}{3})\}}\] upon which $\sigma$ acts by
  rotation descends to a 1-dimensional torus in $\Torus^n/\Delta$ which is
  fixed by $\sigma$, different from the previous one.

  However, any vector $v\in \Torus^n$ which descends to
  $\Fix_{\Torus^n/\Delta}(\sigma)$ satisfies $\sigma\cdot v=v+A$ for some
  $A\in \Delta$. If $m$ is the order of $\sigma$, the identity
  $v=\sigma^m\cdot v=v+mA$ implies $mA=0$. Then $mv$ satisfies
  $\sigma\cdot mv=mv+mA=mv$. Thus for any $\sigma$, the degree
  $(n!)^n$ isogeny $\Torus^n\to \Torus^n$ given by multiplication by $n!$
  takes the preimage of $\Fix_{\Torus^n/\Delta}(\sigma)$ onto
  $\Fix_{\Torus^n}(\sigma)$. (This is certainly overkill; in fact
  multiplication by $n$ suffices, since we can check that the
  existence of such a $v$ implies that $\sigma$ is a product of
  $m$--cycles.) Thus the preimage of $\Fix_{\Torus^n/\Delta}(\sigma)$ is a
  union of finitely many $i$--dimensional tori, each of which descends
  to an $(i-1)$--dimensional torus in $\Fix_{\Torus^n/\Delta}(\sigma)$ as
  claimed.
  
  In particular, if $\sigma$ is an $n$--cycle we must have
  $A=(\frac{a}{n},\ldots,\frac{a}{n})$ for some $a\in \Z/n\Z$. For
  $\sigma=(1\ \cdots\ n)$, say, the fixed set
  $\Fix_{\Torus^n/\Delta}(\sigma)$ consists of the $n$ points represented
  by $\{(0,\frac{a}{n},\frac{2a}{n},\ldots,\frac{(n-1)a}{n})\}$ for
  $a\in\Z/n\Z$.

  Since a union of positive-dimensional tori has Euler characteristic
  0, we see that $\chi(\Fix\sigma)=0$ unless $\sigma$ is an
  $n$--cycle, in which case $\chi(\Fix\sigma)=n$. This is exactly the
  class function $n\chi_{(n)}$. By the Lefschetz fixed point theorem,
  \[n\chi_{(n)}(\sigma)=\chi(\Fix \sigma)
  =\sum_{k=0}^{n-1}\tr\big(\sigma_*|H^k(\Torus^n/\Delta;\Q)\big)\] As $S_n$-representations we have
  $H^1(\Torus^n;\Q)\approx \Q^n$, with $\Delta$
  representing the unique trivial subrepresentation. Therefore 
  $H^1(\Torus^n/\Delta;\Q)$ is isomorphic to the standard representation
  $V_1=\Q^n/Q$ of $S_n$. As with any torus, we have \[H^i(\Torus^n/\Delta;\Q)\approx \bwedge^i H^1(\Torus^n/\Delta;\Q)\approx \bwedge^i (\Q^n/\Q)=V_k.\] Therefore the Lefschetz formula becomes the desired formula
  \[n\chi_{(n)}=\sum_{k=0}^{n-1}\chi_{V_k}.\qedhere\] 
\end{proof}
%(One enjoyable proof goes by applying the Lefschetz fixed point theorem to the action of $S_n$ on the quotient of the torus $\R^n/\Z^n$ the diagonal $\R/\Z\subset \R^n/\Z^n$, since the cohomology of this quotient $X$ is $H^k(X;\Z)\approx \bwedge^k(\Z^n/\Z)$.) 

\begin{proof}[Proof of Proposition~\ref{pr:PNTtori}]
With $\chi_1$ as in Lemma~\ref{le:chi1}, we have 
$\pi(q,n)=\sum_{T\in \T_n(\F_q)}\chi_1(T)$. 
%We need the following identity of characters:
%\[\chi_1=\frac{1}{n}\cdot \sum_{k=0}^{n-1}(-1)^k\chi_{V_k},\qquad\quad \text{where }V_k=\bwedge^k\big(\Q^n/\Q)\]
%(One enjoyable proof goes by applying the Lefschetz fixed point theorem to the action of $S_n$ on the quotient of the torus $\R^n/\Z^n$ the diagonal $\R/\Z\subset \R^n/\Z^n$, since the cohomology of this quotient $X$ is $H^k(X;\Z)\approx \bwedge^k(\Z^n/\Z)$.) 
Therefore Theorem~\ref{th:chitori} gives
\begin{equation}
\label{eq:piqn2}
\pi(q,n)=\frac{q^{n^2-n}}{n}\sum_{i,k}(-1)^k q^{-i}\langle V_k,R_i[x_1,\ldots,x_n]\rangle.
\end{equation}
The representation $V_k$ is irreducible with partition $(n-k,1,\ldots,1)$, and its Young diagram is the \emph{hook} $\lambda_k$ with $n-k$ boxes in the first row, followed by a column of $k$ boxes. A standard tableau of shape $\lambda_k$ is determined by the labels of the $k$ boxes in the column. The descent set of such a tableau is easy to describe: $s$ is a descent if and only if $s+1$ labels a box in the column. Therefore a standard tableau of shape $\lambda_k$ is determined by its descent set $S$, which is a $k$-element subset of $\{1,\ldots,n-1\}$.

To sum up, each subset $S\subset \{1,\ldots,n-1\}$ occurs as the descent set of a unique standard tableau of shape $\lambda_k$ when $k=|S|$, and for no other shape $\lambda_{k'}$, and its major index is $ \sum_{s\in S}s$. Then we can rewrite \eqref{eq:piqn2} as:
\[\pi(q,n)=\frac{q^{n^2-n}}{n}\sum_{S\subset \{1,\ldots,n-1\}}(-1)^{|S|}q^{-\sum_{s\in S}s}\]
This sum can be factored over $j\in \{1,\ldots,n-1\}$ as
\[\sum_{S\subset \{1,\ldots,n-1\}}(-1)^{|S|}q^{-\sum_{s\in S}s}=\prod_{j=1}^{n-1}(1-q^{-j})\]
Pulling a factor of $q^j$ out of $q^{n^2-n}$ for each $j=1,\ldots,n-1$ gives 
\[\pi(q,n)=\frac{q^{\binom{n}{2}}}{n}(q-1)(q^2-1)\cdots(q^{n-1}-1) \]
as claimed.
\qedhere
\end{proof}

\small

\noindent
Dept.\ of Mathematics\\
Stanford University\\
450 Serra Mall\\
Stanford, CA 94305\\
E-mail: church@math.stanford.edu
\medskip

\noindent
Dept.\ of Mathematics\\
University of Wisconsin\\
480 Lincoln Drive\\
Madison, WI 53706\\
E-mail: ellenber@math.wisc.edu
\medskip

\noindent
Dept.\ of Mathematics\\
University of Chicago\\
5734 University Ave.\\
Chicago, IL 60637\\
E-mail: farb@math.uchicago.edu

\end{document}